\documentclass[a4paper,leqno]{amsart}

\usepackage{latexsym}
\usepackage[english]{babel}
\usepackage{fancyhdr}
\usepackage[mathscr]{eucal}
\usepackage{amsmath}
\usepackage{mathrsfs}
\usepackage{amsthm}
\usepackage{amsfonts}
\usepackage{amssymb}
\usepackage{amscd}
\usepackage{bbm}
\usepackage{graphicx}
\usepackage{graphics}
\usepackage{latexsym}
\usepackage{color}
\usepackage{pifont}
\usepackage[caption=false]{subfig}

\newcommand{\ud}{\mathrm{d}}

\newcommand{\ii}{\mathrm{i}}
\newcommand{\cH}{\mathcal{H}}

\newcommand{\hh}{\mathfrak{h}}

\theoremstyle{plain}
\newtheorem{theorem}{Theorem}[section]
\newtheorem{lemma}[theorem]{Lemma}
\newtheorem{corollary}[theorem]{Corollary}
\newtheorem{proposition}[theorem]{Proposition}

\theoremstyle{definition}

\newtheorem{remark}[theorem]{Remark}
\newtheorem*{remark*}{Remark}

\numberwithin{equation}{section}

\begin{document}\hyphenation{Rieman-nian}

\title[Quantum particle across Grushin singularity]
{Quantum particle across Grushin singularity}
\author[M.~Gallone]{Matteo Gallone}
\address[M.~Gallone]{Mathematics Department ``F.~Enriques'', University of Milan \\ via C.~Saldini 50 \\ Milano 20133 (Italy).}
\email{matteo.gallone@unimi.it}
\author[A.~Michelangeli]{Alessandro Michelangeli}
\address[A.~Michelangeli]{Institute of Applied Mathematics and Hausdorff Center for Mathematics, University of Bonn \\ Endenicher Allee 60 \\ 
Bonn 53115  (Germany).}
\email{michelangeli@iam.uni-bonn.de}

\begin{abstract}
A class of models is considered for a quantum particle constrained on degenerate Riemannian manifolds known as Grushin cylinders, and moving freely subject only to the underlying geometry: the corresponding spectral and scattering analysis is developed in detail in view of the phenomenon of transmission across the singularity that separates the two half-cylinders. Whereas the classical counterpart always consists of a particle falling in finite time along the geodesics onto the metric's singularity locus, the quantum models may display geometric confinement, or on the opposite partial transmission and reflection. All the local realisations of the free (Laplace-Beltrami) quantum Hamiltonian are examined as non-equivalent protocols of transmission/reflection and the structure of their spectrum is characterised, including when applicable their ground state and positivity. Besides, the stationary scattering analysis is developed and transmission and reflection coefficients are calculated. This allows to comprehend the distinguished status of the so-called `bridging' transmission protocol previously identified in the literature, which we recover and study within our systematic analysis.
\end{abstract}

\date{\today}


\thanks{Partially supported by the MIUR-PRIN 2017 project MaQuMA cod.~2017ASFLJR, and by the Alexander von Humboldt foundation}

\maketitle


\section{Introduction and background. \\ Quantum Hamiltonians of free evolution over Grushin cylinders}

This work ideally completes a cycle we recently started in \cite{GMP-Grushin-2018} and continued in \cite{GMP-Grushin2-2020}, both works in collaboration with E.~Pozzoli, and is part of the fast growing subject of geometric quantum confinement away from the metric's singularity, and transmission across it, for quantum particles on degenerate Riemannian manifolds. Such theme is particularly active with reference to Grushin structures on cylinder, cone, and plane \cite{Nenciu-Nenciu-2009,Boscain-Laurent-2013,Boscain-Prandi-Seri-2014-CPDE2016,Prandi-Rizzi-Seri-2016,Boscain-Prandi-JDE-2016,Franceschi-Prandi-Rizzi-2017,GMP-Grushin-2018,Boscain-Neel-2019,Boscain-Beschastnnyi-Pozzoli-2020}, as well as, more generally, on two-dimensional orientable compact manifolds \cite{Boscain-Laurent-2013}, and $d$-dimensional incomplete Riemannian manifolds \cite{Prandi-Rizzi-Seri-2016}.

In the present setting we are concerned with Grushin cylinders, i.e., Riemannian manifolds $M_\alpha\equiv(M,g_\alpha)$, with parameter $\alpha\in\mathbb{R}$, where
\begin{equation}
	M^\pm \; := \; \mathbb{R}^\pm_x \times \mathbb{S}^1_y \, \qquad \mathcal{Z} \;:=\; \{0\} \times \mathbb{S}^1_y \, , \qquad M\;:=\; M^+ \cup M^- \, 
\end{equation}
and with degenerate Riemannian metric 
\begin{equation}\label{eq:Grushin_Metric}
	g_\alpha \;:=\; \ud x \otimes \ud x +|x|^{-2\alpha} \ud y \otimes \ud y\,.
\end{equation}
Thus, $M_\alpha$ is a two-dimensional manifold built upon the cylinder $\mathbb{R} \times \mathbb{S}^1$, with singularity locus $\mathcal{Z}$ and incomplete Riemannian metric both on the right and the left half-cylinder $\mathbb{R}^\pm \times \mathbb{S}^1$. The values $\alpha=-1$, $\alpha=0$, and $\alpha=1$ select, respectively, the flat cone, the Euclidean cylinder, and the standard `\emph{Grushin cylinder}' \cite[Chapter 11]{Calin-Chang-SubRiemannianGeometry}: in the latter case one has an `\emph{almost-Riemannian structure}' on $\mathbb{R}\times\mathbb{S}^1=M^+\cup\mathcal{Z}\cup M^-$ in the rigorous sense of \cite[Sec.~1]{Agrachev-Boscain-Sigalotti-2008} or \cite[Sect.~7.1]{Prandi-Rizzi-Seri-2016}. In fact, $M_\alpha$ is a hyperbolic manifold whenever $\alpha>0$, with Gaussian (sectional) curvature $K_\alpha(x,y)=-\alpha(\alpha+1)/x^2$.

A pictorial representation of the ``distortion effect'' of the metric in the vicinity of the singularity locus is provided in Figure \ref{fig:cylinders}.

\begin{figure}[t]
\captionsetup[subfigure]{labelformat=empty} 
  \centering
  \subfloat[][$\alpha<0$]
  {\raisebox{1.4cm}{\includegraphics[width=0.3\textwidth]{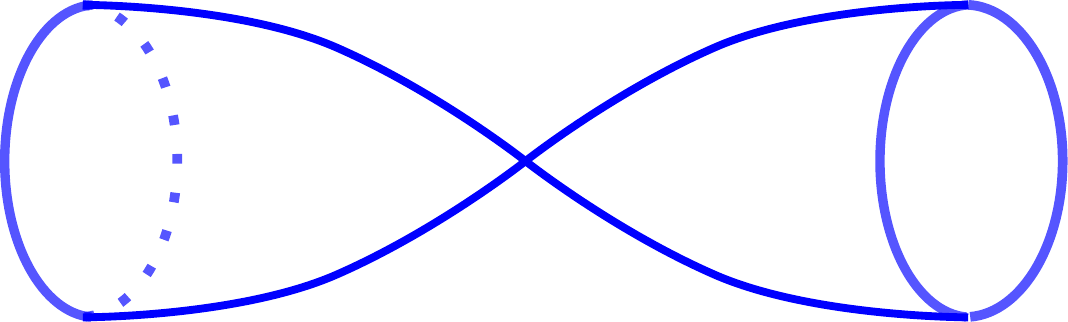} }}
  \quad \subfloat[][$\alpha=0$]
  {\raisebox{1.4cm}{\includegraphics[width=0.3\textwidth]{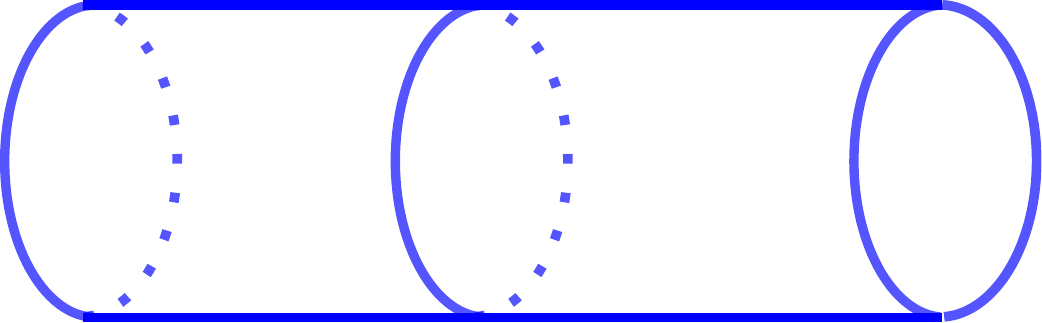} } }
    \quad \subfloat[][$\alpha>0$]
  {\includegraphics[width=0.3\textwidth]{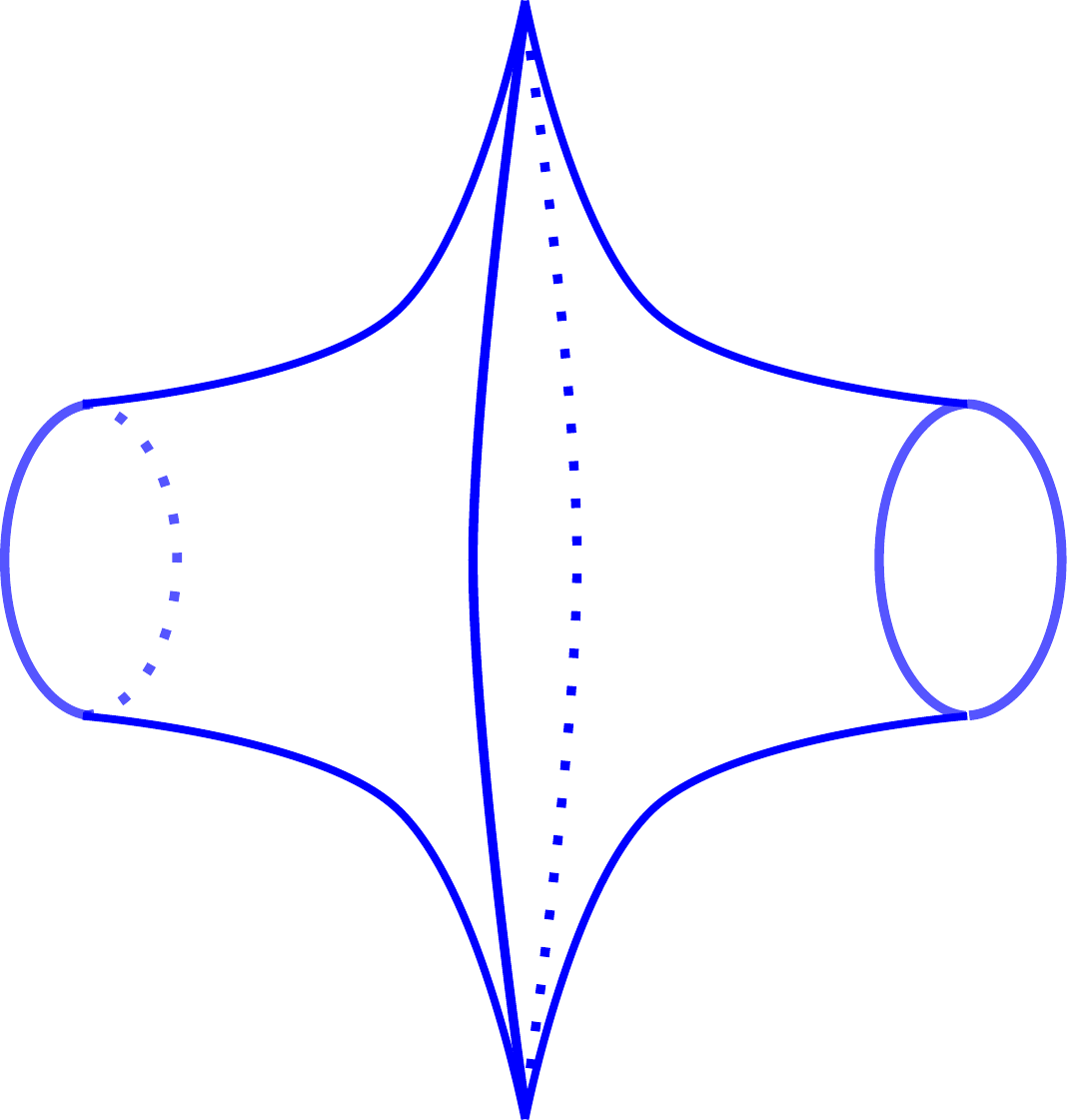} }
  \caption{Manifold $M_\alpha$ for different $\alpha$.}
  \label{fig:cylinders}
\end{figure}

Quantum-mechanically, Grushin cylinders provide the underlying structure for the following prototypical, relevant model: a non-relativistic quantum particle is constrained on $M_\alpha$ and is only subject to the geometric effects due to the non-flat metric, in other words it evolves ``freely'' under the Hamiltonian whose formal action is given by the Laplace-Beltrami operator. The latter, denoted henceforth as $\Delta_\alpha$, has the form
\begin{equation}\label{eq:IntroLaplaceBeltrami}
	\Delta_{\alpha} \; = \; \frac{\partial^2}{\partial x^2} +|x|^{2\alpha}\frac{\partial^2}{\partial y^2} -\frac{\alpha}{|x|} \frac{\partial}{\partial x}\,,
\end{equation}
as is easy to compute (see, e.g. \cite[Sect.~2]{GMP-Grushin-2018}), and acts on functions on $M$.

The Hilbert space for the considered quantum model is therefore
\begin{equation}
	\mathcal{H}_\alpha \;:=\; L^2(M,\ud \mu_\alpha) \, ,
\end{equation}
understood as the completion of $C^\infty_c(M)$ with respect to the scalar product
\begin{equation}
	\langle \psi, \phi \rangle_\alpha \;:=\; \iint_{\mathbb{R} \times \mathbb{S}^1} \overline{\psi(x,y)} \, \phi(x,y) \, \frac{1}{|x|^\alpha} \, \ud x \, \ud y \, ,
\end{equation}
where $\mu_\alpha$ is the Riemannian volume form associated to each $M_\alpha$, namely 
\begin{equation}
	\mu_\alpha\;:=\; \text{vol}_{g_\alpha} \;=\; \sqrt{\det g_\alpha} \, \ud x \wedge \ud y\,.
\end{equation}
Of course, the model becomes quantum-mechanically unambiguous upon realising the formal free Hamiltonian $-\Delta_{\alpha}$ self-adjointly on $\mathcal{H}_\alpha$.

The first (but not only) conceptual relevance of the above physical system is due to the following contrast between the classical and the quantum behaviour of a particle constrained on $M_\alpha$.

\emph{Classically}, for $\alpha>0$ the manifold $M_\alpha$ is geodesically incomplete (indeed, $M_\alpha$ is obviously incomplete as a metric space, which is seen by the non-convergent Cauchy sequence of points $(n^{-1},y_0)\in M$ as $n\to\infty$, then incompleteness follows from a standard Hopf-Rinow theorem \cite[Theorem 2.8, Chapter 7]{DoCarmo-Riemannian}), but not just in the sense that there exist \emph{one} geodesic curve that passes through a given arbitrary point $(x_0,y_0)\in M$ and reaches the boundary $\partial M=\mathcal{Z}$ in finite time in the past or in the future (this is evidently the horizontal line $y=y_0$). One has the much more restrictive feature that, as shown in \cite[Sect.~4.1]{GMP-Grushin-2018}, \emph{all} geodesics passing through $(x_0,y_0)$ at $t=0$ intercept the $y$-axis at finite times $t_\pm$ with $t_-<0<t_+$, with the sole exception of the geodesic line $y=y_0$ along which the boundary is reached only in one direction of time. The classical particle \emph{always} hits the singularity in finite time.

\emph{Quantum mechanically}, on the contrary, a particle whose wave-function is initially supported on one half-cylinder can reach the boundary and trespass it, or instead remain in the original half-cylinder while staying separated from the boundary at all later times, depending on the regimes of $\alpha$.

More precisely, to describe such an alternative one defines the `\emph{minimal free Hamiltonian}' on the space of smooth and compactly supported functions on $M$ 
\begin{equation}\label{Halpha}
	H_\alpha\;:=\; - \Delta_{\mu_\alpha} \, , \qquad \mathcal{D}(H_\alpha) \;:= \; C^\infty_c(M)\,,
\end{equation}
and proves the following \cite{Boscain-Laurent-2013,Boscain-Prandi-JDE-2016,GMP-Grushin-2018}:
  \begin{itemize}
   \item[(i)] if $\alpha\in(-\infty,-3]\cup[1,+\infty)$, then the operator $H_\alpha$ is essentially self-adjoint;
   \item[(ii)] if $\alpha\in(-3,-1]$, then the operator $H_\alpha$ is not essentially self-adjoint and it has deficiency index $2$;
   \item[(iii)] if $\alpha\in(-1,+1)$, then the operator $H_\alpha$ is not essentially self-adjoint and it has infinite deficiency index.
  \end{itemize}
 Let us recall that the deficiency index $\mathsf{d}$ of a densely defined and positive symmetric operator on Hilbert space, as \eqref{Halpha} above, which is technically defined as the cardinal number $\mathsf{d}:=\dim\ker( H_\alpha^*+\mathbbm{1})$, is an indicator that measures the size of the family of all possible self-adjoint realisations of $H_\alpha$, hence all extensions of $H_\alpha$ that have the physical meaning of quantum Hamiltonian. Such extension family is parametrised by $\mathsf{d}^2$ real parameters. Each self-adjoint realisations is thus the generator of a unitary quantum dynamics, and distinct realisations give rise to different unitary evolutions, hence different physics. The case of zero deficiency index corresponds to $H_\alpha$ being essentially self-adjoint, meaning that there is a unique way to extend $H_\alpha$ self-adjointly, which simply consists of taking its operator closure $\overline{H_\alpha}$. In this case there is no ambiguity in the physics described by $H_\alpha$.

Thus, case (i) above is interpreted as the regime of `\emph{geometric quantum confinement}': as $\overline{H_\alpha}$ is self-adjoint with respect to $L^2(M,\ud \mu_\alpha)$, so too are $\overline{H_\alpha^\pm}$ with respect to $L^2(M^\pm,\ud \mu_\alpha)$, having defined $H_\alpha^{\pm}$ analogously to \eqref{Halpha} with domain $C^\infty_c(M^{\pm})$, and with respect to the decomposition
 \begin{equation}\label{eq:decomp+-}
  L^2(M,\ud\mu_\alpha)\;\cong\;L^2(M^-,\ud \mu_\alpha)\oplus L^2(M^+,\ud \mu_\alpha)
 \end{equation}
 the overall unitary propagator is reduced as 
  \begin{equation}
  e^{-\ii t \overline{H_\alpha}}\;=\;e^{-\ii t \overline{H_\alpha^-}}\oplus e^{-\ii t \overline{H_\alpha^+}}\,,\qquad \forall\,t\in\mathbb{R}\,.
 \end{equation}
 Therefore, if a particle is initially prepared in a state $\psi_0\in\mathcal{D}(\overline{H_\alpha})$ with support only within $M^+$, the \emph{unique} solution $\psi\in C^1(\mathbb{R}_t,L^2(M,\ud\mu_\alpha))$ to the Schr\"{o}dinger initial value problem
 \begin{equation}
  \begin{cases}
   \;\ii\partial_t \psi \!\!&=\;\overline{H_\alpha}\,\psi \\
   \;\psi|_{t=0}\!\!&=\;\psi_0
  \end{cases}
 \end{equation}
 remains for all times supported (`confined') in $M^+$, thus never crossing the $y$-axis towards the left half-cylinder. Such confinement is, in a sense, only a consequence of the geometry of $M_\alpha$, because $H_\alpha$ is not qualified by boundary conditions at the singularity and the free Hamiltonian $\overline{H_\alpha}$ is realised canonically upon $H_\alpha$ as its operator closure.

 The regimes (ii) and (iii) listed above produce instead a variety of distinct physical `\emph{protocols of quantum transmission}' across the singularity, depending on the specific boundary conditions of self-adjointness imposed at $\partial M$, and this represents the other relevant feature of the class of models under consideration.

 Such occurrence is in fact a signature of the incompleteness of $M_\alpha$. For, the minimally defined Laplace-Beltrami operator is always essentially self-adjoint on complete Riemannian manifolds, whereas in general essential self-adjointness is broken if the manifold is incomplete (see, e.g., \cite{Masamune-2005} and references therein).

 The emergent physics is particularly rich and interesting in the regime $\alpha\in[0,1)$, owing to the simultaneous infinity of the deficiency index of the minimal operator $H_\alpha$ and singularity of the metric $g_\alpha$. Of course, a large portion of the self-adjoint realisations of $H_\alpha$ in this regime would correspond to Hamiltonians of scarce physical interest, owing to the non-local character of their boundary conditions as $|x|\to 0$, yet there remains an ample class of physically meaningful Hamiltonians, characterised by local boundary conditions at the singularity, which govern the transmission across it.

 In the work \cite{GMP-Grushin2-2020}, together with Pozzoli, we established an extensive and fairly explicit classification of the ``physical'' self-adjoint realisations of the Laplace-Beltrami operator on the manifold $M_\alpha$ in the significant regime $\alpha\in(0,1)$. The result is summarised as follows.


    \begin{theorem}[\cite{GMP-Grushin2-2020}]\label{thm:H_alpha_fibred_extensions}
 Let $\alpha\in[0,1)$. The densely defined and symmetric operator $H_\alpha$ defined in \eqref{Halpha} admits, among others, the following families of self-adjoint extensions with respect to $L^2(M,\ud\mu_\alpha)$:
 \begin{itemize}
  \item \underline{Friedrichs extension}: $H_{\alpha,\mathrm{F}}$;
  \item \underline{Family $\mathrm{I_R}$}: $\{H_{\alpha,\mathrm{R}}^{[\gamma]}\,|\,\gamma\in\mathbb{R}\}$;
  \item \underline{Family $\mathrm{I_L}$}: $\{H_{\alpha,\mathrm{L}}^{[\gamma]}\,|\,\gamma\in\mathbb{R}\}$;
  \item \underline{Family $\mathrm{II}_a$} with $a\in\mathbb{C}$: $\{H_{\alpha,a}^{[\gamma]}\,|\,\gamma\in\mathbb{R}\}$;
  \item \underline{Family $\mathrm{III}$}: $\{H_{\alpha}^{[\Gamma]}\,|\,\Gamma\equiv(\gamma_1,\gamma_2,\gamma_3,\gamma_4)\in\mathbb{R}^4\}$.
 \end{itemize}
 Each member of any such family acts precisely as the differential operator $-\Delta_{\mu_\alpha}$ on a domain of functions $f\in L^2(M,\ud\mu_\alpha)$ satisfying the following properties.
  \begin{itemize}
  \item[(i)] \underline{Integrability and regularity}:
  \begin{equation}\label{eq:DHalpha_cond1}
  \sum_{\pm}\;\iint_{\mathbb{R}_x^\pm\times\mathbb{S}^1_y}\big|(\Delta_{\mu_\alpha}f^\pm)(x,y)\big|^2\,\ud\mu_\alpha(x,y)\;<\;+\infty\,.
 \end{equation}
  \item[(ii)] \underline{Boundary condition}: The limits
 \begin{eqnarray}
  f_0^\pm(y)\!\!&=&\!\!\lim_{x\to 0^\pm}f^\pm(x,y) \label{eq:DHalpha_cond2_limits-1}\\
  f_1^\pm(y)\!\!&=&\!\!\pm(1+\alpha)^{-1}\lim_{x\to 0^\pm}\Big(\frac{1}{\:|x|^\alpha}\,\frac{\partial f(x,y)}{\partial x}\Big) \label{eq:DHalpha_cond2_limits-2}
  \end{eqnarray}
 exist and are finite for almost every $y\in\mathbb{S}^1$, and depending on the considered type of extension, and for almost  every $y\in\mathbb{R}$,
 \begin{eqnarray}
  f_0^\pm(y)\,=\,0 \qquad \quad\;\;& & \textrm{if }\;  f\in\mathcal{D}(H_{\alpha,\mathrm{F}})\,, \label{eq:DHalpha_cond3_Friedrichs}\\
  \begin{cases}
   \;f_0^-(y)= 0  \\
   \;f_1^+(y)=\gamma f_0^+(y)
  \end{cases} & & \textrm{if }\;  f\in\mathcal{D}(H_{\alpha,\mathrm{R}}^{[\gamma]})\,, \\
   \begin{cases}
   \;f_1^-(y)=\gamma f_0^-(y) \\
   \;f_0^+(y)= 0 
  \end{cases} & & \textrm{if }\;  f\in\mathcal{D}(H_{\alpha,\mathrm{L}}^{[\gamma]}) \,, \label{eq:DHalpha_cond3_L}\\
     \begin{cases}
   \;f_0^+(y)=a\,f_0^-(y) \\
   \;f_1^-(y)+\overline{a}\,f_1^+(y)=\gamma f_0^-(y)
  \end{cases} & & \textrm{if }\;  f\in\mathcal{D}(H_{\alpha,a}^{[\gamma]})\,, \label{eq:DHalpha_cond3_IIa} \\
   \begin{cases}
   \;f_1^-(y)=\gamma_1 f_0^-(y)+(\gamma_2+\ii\gamma_3) f_0^+(y) \\
   \;f_1^+(y)=(\gamma_2-\ii\gamma_3) f_0^-(y)+\gamma_4 f_0^+(y)
  \end{cases} & & \textrm{if }\;  f\in\mathcal{D}(H_{\alpha}^{[\Gamma]})\,. \label{eq:DHalpha_cond3_III}
 \end{eqnarray} 
 \end{itemize} 
     Moreover,
 \begin{equation}\label{eq:traceregularity}
  f_0^\pm \in H^{s_{0,\pm}}(\mathbb{S}^1, \ud y)\qquad\textrm{ and }\qquad f_1^\pm\in H^{s_{1,\pm}}(\mathbb{S}^1,\ud y)
 \end{equation}
 with
 \begin{itemize}
 	\item $s_{1,\pm}=\frac{1}{2}\frac{1-\alpha}{1+\alpha}$\qquad\qquad\qquad\qquad\qquad\; for the Friedrichs extension,
 	\item $s_{1,-}=\frac{1}{2}\frac{1-\alpha}{1+\alpha}$, $s_{0,+}=s_{1,+}=\frac{1}{2}\frac{3+\alpha}{1+\alpha}$ \quad for extensions of type $\mathrm{I_R}$,
 	\item  $s_{1,+}=\frac{1}{2}\frac{1-\alpha}{1+\alpha}$, $s_{0,-}=s_{1,-}=\frac{1}{2}\frac{3+\alpha}{1+\alpha}$ \quad for extensions of type $\mathrm{I_L}$,
 	\item $s_{1,\pm}=s_{0,\pm}=\frac{1}{2}\frac{1-\alpha}{1+\alpha}$ \qquad\qquad\qquad \;\;\;\,for extensions of type $\mathrm{II}_a$,
 	\item $s_{1,\pm}=s_{0,\pm}=\frac{1}{2}\frac{3+\alpha}{1+\alpha}$ \qquad\qquad\qquad \;\; for extensions of type $\mathrm{III}$.
 \end{itemize}
\end{theorem}

Theorem \ref{thm:H_alpha_fibred_extensions} has a partial precursor in the work \cite{Boscain-Prandi-JDE-2016} by Boscaini and Prandi, where a few distinguished realisations (the Friedrichs extension and the one of type-$\mathrm{II}_a$ corresponding to $a=1$ and $\gamma=0$) were identified by direct methods.

Let us stress that \eqref{eq:DHalpha_cond3_Friedrichs}-\eqref{eq:DHalpha_cond3_III} all express \emph{local} boundary conditions.
                         
In short, here is the physical picture emerging from Theorem \ref{thm:H_alpha_fibred_extensions}.
\begin{itemize}
 \item The Friedrichs extension $H_{\alpha,\mathrm{F}}$ models quantum confinement on each half of the Grushin cylinder, with no interaction of the particle with the boundary and no dynamical transmission between the two halves. 
 \item Type-$\mathrm{I_R}$ and type-$\mathrm{I_L}$ extensions model systems with no dynamical transmission across $\mathcal{Z}$, but with possible non-trivial interaction of the quantum particle with the boundary respectively from the right or from the left, with confinement on the opposite side. For instance, a quantum particle governed by $H_{\alpha,\mathrm{R}}^{[\gamma]}$ may `touch' the boundary from the right, but not from the left, and moreover it cannot trespass the singularity region.
 \item Type-$\mathrm{II}_a$ and type-$\mathrm{III}$ extensions model in general, dynamical transmission through the boundary.
\end{itemize}

\section{Main results}

The variety of physically meaningful protocols of quantum transmission emerging from the analysis of Theorem \ref{thm:H_alpha_fibred_extensions} poses a number of  questions that are fundamental in the applications of each such model, and that constitute the goals of the present work.


\textbf{I. Positivity.} Namely, the identification of those Laplace-Beltrami realisations that are non-negative self-adjoint operators on $\cH_\alpha$. Observe that the minimal model $H_\alpha$ introduced in \eqref{Halpha} is non-negative, so here one is inquiring which self-adjoint realisations $\widetilde{H}_\alpha$ preserve non-negativity -- all others creating strictly negative bound states. Apart from its quantum-mechanical relevance in the Schr\"{o}dinger evolution, hence for the \emph{unitary group} $(e^{-\ii t \widetilde{H}_\alpha})_{t\in\mathbb{R}}$, this information is crucial for the associated \emph{semi-group} $(e^{-t \widetilde{H}_\alpha})_{t\geqslant 0}$ for the study of the corresponding heat equation 
\begin{equation}\label{eq:heat}
 \partial_t f\;=\;\Delta_\alpha f\,.
\end{equation}
Among non-negative generators, it then becomes of interest to select those that in addition are `\emph{Markovian}' and hence generate `\emph{Markovian semi-groups}', defined by the property
\begin{equation}
 0\leqslant f\leqslant 1\quad(x,y)\textrm{-a.e.}\qquad\Rightarrow\qquad 0\leqslant e^{-t \widetilde{H}_\alpha}f\leqslant 1\quad(x,y)\textrm{-a.e.},\;\;\;\forall t\geqslant 0\,. 
\end{equation}
Each such Markovian extension $\widetilde{H}_\alpha$ therefore generates a Markov processs $(X_t)_{t\geqslant 0}$, for which it is relevant to inquire the possible stochastic completeness and recurrence. The former property, in particular, expresses the circumstance that the process $(X_t)_{t\geqslant 0}$ has infinite lifespan almost surely, which is interpreted as the fact that along the evolution \eqref{eq:heat} the heat is not absorbed by $\mathcal{Z}$. Such a programme (see, e.g., \cite{Fukushima-Oshima-Takeda}) was carried out in \cite{Boscain-Prandi-JDE-2016} for certain distinguished non-negative realisations of $\Delta_\alpha$ and, in systematic form, is the object of a forthcoming separate work of ours.

\textbf{II. Negative point spectrum.} This concerns the low-energy behaviour of those transmission protocols that produce confinement around the singularity locus $\mathcal{Z}$, and in particular the quantification of the number of negative bound states for those self-adjoint Laplace-Beltrami realisations with negative point spectrum.

\textbf{III. Ground state.} Namely, the quantification for each transmission protocol of the (negative) lowest-energy bound state and its explicit wave function (together with the control of its non-degeneracy). In connection to that, it is of relevance to characterise, at least to estimate the lowest (strictly positive) eigenvalue embedded in the continuum spectrum of the \emph{non-transmitting} protocol, namely the Friedrichs extension $H_{\alpha,\mathrm{F}}$.

\textbf{IV. Scattering.} For those protocols that allow for an actual transmission, a relevant information is the quantification of the transmitted flux of particles across the Grushin singularity, and the reflected flux of particles bouncing backwards, once a given incident flux is injected into the manifold at given positive energy and shot at large distances towards the origin. In particular, one would like to characterise the transmission and reflection coefficients in terms of the energy of the incident particles, including the high- and low-energy regimes.

In view of the above goals, let us present now our main results. First, we do characterise all transmission protocols that are generated by a positive (meaning: non-negative) self-adjoint realisation of $-\Delta_\alpha$.

\begin{theorem}[Positive extensions]\label{thm:positivity}
Let $\alpha\in[0,1)$. With respect to the self-adjoint extensions of the minimal operator $H_\alpha$ classified in Theorem \ref{thm:H_alpha_fibred_extensions}, and in terms of the extension parameters $\gamma$ and $\Gamma$ introduced therein,
\begin{itemize}
		\item the Friedrichs extension $H_{\alpha,\mathrm{F}}$ is non-negative;
		\item extensions in the family $\mathrm{I_R}$, $\mathrm{I_L}$, and $\mathrm{II}_a$, $a\in\mathbb{C}$, are non-negative if and only if $\gamma\geqslant 0$;
		\item extensions in the family $\mathrm{III}$ are non-negative if and only if so is the matrix
		\[
		 \widetilde{\Gamma}\;:=\;\begin{pmatrix}
		  \gamma_1 & \gamma_2+\ii\gamma_3 \\
		  \gamma_2-\ii\gamma_3 & \gamma_4
		 \end{pmatrix},
		\]
                i.e., if and only if $\gamma_1 + \gamma_4 > 0$ and $\gamma_1 \gamma_4\geqslant \gamma_2^2 + \gamma_3^2$.
\end{itemize}
\end{theorem}

In practice, the positivity of each realisation is determined by the positivity of the corresponding extension parameter ($\gamma$ or $\widetilde{\Gamma}$), a characterisation that takes such simple form thanks to the efficient choice of the labelling for each extension, made in \cite{GMP-Grushin2-2020} by exploiting the Kre\u{\i}n-Vi\v{s}ik-Birman extension scheme \cite{GMO-KVB2017}.

Next, we describe the spectra of each self-adjoint Hamiltonian of the family of the local transmission protocols described by Theorem \ref{thm:H_alpha_fibred_extensions}. The structure of each spectrum turns out to consist of a common essential spectrum, the non-negative half line, in which an infinity of eigenvalues are embedded, plus a finite negative discrete spectrum.

We shall use the convenient notation
	\begin{equation}
		\lfloor x \rfloor \;:=\; \begin{cases}
			\;n & \textrm{if $x \in (n,n+1]$ for some }n\in\mathbb{N}_0\,, \\
			\;0 & \textrm{if $x \leqslant 0$}\,.
		\end{cases}
	\end{equation}

\begin{theorem}[Spectral analysis of $H_\alpha$]\label{thm:MainSpectral}
Let $\alpha\in[0,1)$. With respect to the self-adjoint extensions of the minimal operator $H_\alpha$ classified in Theorem \ref{thm:H_alpha_fibred_extensions}, and in terms of the extension parameters $\gamma$ and $\Gamma$ introduced therein, each such extension has finite negative discrete spectrum, and essential spectrum equal to $[0,+\infty)$. In particular, any such operator is lower semi-bounded. The essential spectrum contains in each case a (countable) infinity of embedded eigenvalues, with no accumulation, each of finite multiplicity. The number $\mathcal{N}_-(\widetilde{H}_\alpha)$ of negative eigenvalues for each considered extension $\widetilde{H}_\alpha$, counted with their multiplicity, is computed as follows.
\begin{equation}
 \mathcal{N}_-\big(H_{\alpha,\mathrm{F}}\big)\;=\;0\,.
\end{equation}
\begin{equation}
 \mathcal{N}_-\big(H_{\alpha,\mathrm{R}}^{[\gamma]}\big)\;=\;\mathcal{N}_-\big(H_{\alpha,\mathrm{L}}^{[\gamma]}\big)\;=\;2 \lfloor (1+\alpha) |\gamma| \rfloor +1\,,\qquad\gamma<0\,.
\end{equation}
\begin{equation}
 \mathcal{N}_-\big(H_{\alpha,a}^{[\gamma]}\big)\;=\;2\left\lfloor \frac{1+\alpha}{1+|a|^2} |\gamma| \right\rfloor	 +1\,,\qquad \gamma<0\,,\;a\in\mathbb{C}\,.
\end{equation}
  \begin{equation}\label{eq:multnegspec-III}
   \begin{split}
    \mathcal{N}_-\big(H_{\alpha}^{[\Gamma]}\big)\;&=\;\textstyle 2\lfloor -(1+\alpha)\big(\gamma_1+\gamma_4+\sqrt{(\gamma_1-\gamma_4)^2+4 (\gamma_2^2 + \gamma_3^2)}\,\big)\rfloor \\
    & \quad + \textstyle 2\lfloor -(1+\alpha)\big(\gamma_1+\gamma_4-\sqrt{(\gamma_1-\gamma_4)^2+4 (\gamma_2^2 + \gamma_3^2)}\,\big)\rfloor \\
    & \quad + n_0(\Gamma)
   \end{split}
  \end{equation}
   with
   \begin{equation}
    n_0(\Gamma)\,:=\,\begin{cases}
   \;2 & \textrm{if}\quad\gamma_1 \, \gamma_4>\gamma_2^2 + \gamma_3^2 \quad \text{and} \quad \gamma_1+\gamma_4<0 \\
   \;1 & \textrm{if}\quad \gamma_1 \, \gamma_4<\gamma_2^2 + \gamma_3^2\quad\textrm{or}\quad
   \begin{cases}
    \;\gamma_1 \, \gamma_4 =\gamma_2^2 + \gamma_3^2 \\
    \;\gamma_1+\gamma_4<0
   \end{cases} \\
   \;0 & \textrm{if}\quad \gamma_1 \, \gamma_4\geqslant\gamma_2^2 + \gamma_3^2 \quad \text{and} \quad \gamma_1+\gamma_4>0\,.
  \end{cases}
   \end{equation}
\end{theorem}

For those transmission protocols whose Hamiltonian admits negative bound states and hence a negative lowest-energy eigenstate (the ground state), we are able to characterize the ground state's energy and wave function.

The latter shall be expressed in terms of the special function $K_{\frac{1+\alpha}{2}}$, where $K_\nu$ denotes the modified Bessel function 
\begin{equation}
 K_\nu(z)\;=\;{\textstyle\frac{\pi}{\,2\,\sin\nu\pi}}\big(e^{\ii\frac{\pi}{2}\nu}J_{-\nu}(\ii z)-e^{-\ii\frac{\pi}{2}\nu}J_{\nu}(\ii z)\big)\,,
\end{equation}
and $J_\nu$ is the ordinary Bessel function of first kind \cite[Eq.~(9.6.2)-(9.6.3)]{Abramowitz-Stegun-1964}. $K_\nu$ is smooth on $\mathbb{R}^+$; further details will be given in Sect.~\ref{subsec:Spectrum0}.

\begin{theorem}[Ground-states]\label{thm:GroundStateCH}
Let $\alpha\in[0,1)$, $\gamma\in\mathbb{R}$, $a\in\mathbb{C}$, $\Gamma\equiv(\gamma_1,\gamma_2,\gamma_3,\gamma_4)\in\mathbb{R}^4$.
\begin{itemize}
 \item[(i)] When $H_{\alpha,\mathrm{R}}^{[\gamma]}$ (respectively, $H_{\alpha,\mathrm{L}}^{[\gamma]}$) has negative spectrum, i.e., when $\gamma<0$, it has a unique ground state with energy $ E_0\big(H_{\alpha,\mathrm{R}}^{[\gamma]}\big)$ (respectively, $ E_0\big(H_{\alpha,\mathrm{L}}^{[\gamma]}\big)$) given by
 \begin{equation}\label{eq:gsenergyR}
  E_0\big(H_{\alpha,\mathrm{R}}^{[\gamma]}\big)\;=\;E_0\big(H_{\alpha,\mathrm{L}}^{[\gamma]}\big)\;=\; -\left(2\,\frac{\Gamma(\frac{1+\alpha}{2})}{\,\Gamma(-\frac{1+\alpha}{2})}\, \gamma \right)^{\frac{2}{1+\alpha}} 
 \end{equation}
 and non-normalised eigenfunction $\Phi_{\alpha}^{(\mathrm{I_R})}$ (respectively, $\Phi_{\alpha}^{(\mathrm{I_L})}$) given by
 \begin{equation}\label{eq:gs--R}
 \begin{split}
  \Phi_{\alpha}^{(\mathrm{I_R})}(x,y)\;&=\;
  \begin{cases}
   \qquad 0\,, & x<0\,, \\
   \;x^{\frac{1+\alpha}{2}}K_{\frac{1+\alpha}{2}}(x\sqrt{E})\,, & x>0\,,
  \end{cases} \\
 \Phi_{\alpha}^{(\mathrm{I_L})}(x,y)\;&=\;
  \begin{cases}
   \;(-x)^{\frac{1+\alpha}{2}}K_{\frac{1+\alpha}{2}}(-x\sqrt{E})\,, & x<0\,, \\
   \qquad 0\,, & x>0\,,
  \end{cases}
 \end{split}
\end{equation}
  where for short $E\equiv -E_0\big(H_{\alpha,\mathrm{R}}^{[\gamma]}\big)=-E_0\big(H_{\alpha,\mathrm{L}}^{[\gamma]}\big)$.
  \item[(ii)] When $H_{\alpha,a}^{[\gamma]}$ has negative spectrum, i.e., when $\gamma<0$, it has a unique ground state with energy $E_0\big(H_{\alpha,a}^{[\gamma]}\big)$ given by
 \begin{equation}
  E_0\big(H_{\alpha,a}^{[\gamma]}\big)\;=\;-\left(\frac{2\,\Gamma(\frac{1+\alpha}{2})}{(1+|a|^2)\Gamma(-\frac{1+\alpha}{2})} \,\gamma \right)^{\frac{2}{1+\alpha}}
 \end{equation}
 and non-normalised eigenfunction $\Phi_{\alpha}^{(\mathrm{II}_a)}$ given by
 \begin{equation}
  \Phi_{\alpha}^{(\mathrm{II}_a)}(x,y)\;=\;
  \begin{cases}
   \; (-x)^{\frac{1+\alpha}{2}}K_{\frac{1+\alpha}{2}}(-x\sqrt{E})\,, & x<0\,, \\
   \;a x^{\frac{1+\alpha}{2}}K_{\frac{1+\alpha}{2}}(x\sqrt{E})\,, & x>0\,,
  \end{cases}
\end{equation}
  where for short $E\equiv -E_0\big(H_{\alpha,a}^{[\gamma]}\big)$.
 \item[(iii)] When $H_{\alpha}^{[\Gamma]}$ has negative spectrum (see Theorems \ref{thm:positivity}-\ref{thm:MainSpectral}), its ground state energy $E_0\big(H_{\alpha}^{[\Gamma]}\big)$ is given by
 \begin{equation}
  E_0\big(H_{\alpha}^{[\Gamma]}\big)\;=\;-\left(\frac{\,2^\alpha\,\Gamma(\frac{1+\alpha}{2})}{\Gamma(-\frac{1+\alpha}{2})} \Big( \gamma_1+\gamma_4 - \sqrt{(\gamma_1-\gamma_4)^2 +4 (\gamma_2^2+\gamma_3^2)} \,\Big)\right)^{\frac{2}{1+\alpha}}.
 \end{equation}
  The ground state has at most two-fold degeneracy. It is non-degenerate if and only if, under the conditions for its negativity (thus, $\gamma_1+\gamma_4<\sqrt{(\gamma_1-\gamma_4)^2 +4 (\gamma_2^2+\gamma_3^2)}$), additionally one has $\gamma_1\neq\gamma_4$ or $\gamma_2^2+\gamma_3^2>0$, in which case the non-normalised  eigenfunction $\Phi_{\alpha}^{(\mathrm{III})}$ is given by
\begin{equation}
\begin{split}
 & \Phi_{\alpha}^{(\mathrm{III})}(x,y)\;=\;\begin{cases}
  \!\!\begin{array}{l}
   \big({\textstyle \gamma_1-\gamma_4 - \sqrt{(\gamma_1-\gamma_4)^2 +4 (\gamma_2^2+\gamma_3^2)}}\big)\,\times \\
   \qquad \times \,(-x)^{\frac{1+\alpha}{2}}K_{\frac{1+\alpha}{2}}(-x\sqrt{E})\,,
  \end{array}
  & x<0\,, \\
  \;2(\gamma_2-\ii\gamma_3)\,x^{\frac{1+\alpha}{2}}\,K_{\frac{1+\alpha}{2}}(x\sqrt{E})\,, & x>0\,,
 \end{cases}
\end{split}
\end{equation}
where for short $E\equiv -E_0\big(H_{\alpha}^{[\Gamma]}\big)$.
 The ground state is two-fold degenerate if and only if $\gamma_1=\gamma_4<0$ and $\gamma_2=\gamma_3=0$, in which case its eigenspace is spanned by the non-normalised eigenfunctions $\Phi_{\alpha,c_1,c_2}^{(\mathrm{III})}$ given by
  \begin{equation}
  \Phi_{\alpha,c_1,c_2}^{(\mathrm{III})}(x,y)\;=\;
  \begin{cases}
   \; c_1 (-x)^{\frac{1+\alpha}{2}}K_{\frac{1+\alpha}{2}}(-x\sqrt{E})\,, & x<0\,, \\
   \; c_2\, x^{\frac{1+\alpha}{2}}K_{\frac{1+\alpha}{2}}(x\sqrt{E})\,, & x>0\,,
  \end{cases}
  \end{equation}
 where $c_1,c_2\in\mathbb{C}$, and again $E\equiv -E_0\big(H_{\alpha}^{[\Gamma]}\big)$.
\end{itemize}
\end{theorem}

 The Friedrichs extension $H_{\alpha,\mathrm{F}}$ is not covered by Theorem \ref{thm:GroundStateCH} because it has no negative spectrum, its spectrum being $[0,+\infty)$ and purely essential. It is of interest, nevertheless, to get information on the first (lowest) energy eigenstate of $H_{\alpha,\mathrm{F}}$ embedded in the essential spectrum (as described in Theorem \ref{thm:MainSpectral}). We find the following.

  \begin{proposition}\label{prop:Friedrichs-groundstate}
  For given $\alpha\in[0,1)$, the (lowest) energy eigenstate of $H_{\alpha,\mathrm{F}}$ has eigenvalue $E_0(H_{\alpha,\mathrm{F}})$ embedded in $[0,+\infty)$, which is two-fold degenerate, and is estimated as
 \begin{equation}\label{eq:estimate-Fgroundstate}
	(1+\alpha)\left({\textstyle\frac{2+\alpha}{4} }\right)^{\frac{\alpha}{1+\alpha}}  \; \leqslant \;    E_0(H_{\alpha,\mathrm{F}}) \; \leqslant \;  \frac{2^{\frac{1-\alpha}{1+\alpha}} (1+\alpha)^{\frac{1+3\alpha}{1+\alpha}}}{\alpha^{\frac{\alpha}{1+\alpha}} \Gamma(\frac{3+\alpha}{1+\alpha})}\,.
\end{equation}
  In particular, $E_0(H_{\alpha,\mathrm{F}})$ is strictly positive, and the lowest spectral point $0$ in the spectrum of $H_{\alpha,\mathrm{F}}$ is not an eigenvalue.  
 \end{proposition}


 Figure \ref{fig:variational-estimate} shows the effectiveness and narrowness of the bounds \ref{eq:estimate-Fgroundstate} for all admissible values of $\alpha$.


The structure of the ground state wave functions described in Theorem \ref{thm:GroundStateCH} is qualitatively the same for each model and consists of a behaviour $|x|^{\frac{1+\alpha}{2}}K_{\frac{1+\alpha}{2}}(|x|)$, on both half lines when applicable, while being constant in $y$. The latter feature, as will emerge from next sections' analysis, is due to the compactness of the $y$-variable in $M$: the lowest energy level of the Hamiltonian is a contribution from the ``zero-th'' mode of functions in $y$, the constant-in-$y$ functions. The function $|x|^{\frac{1+\alpha}{2}}K_{\frac{1+\alpha}{2}}(|x|)$ is localised around $x=0$ with exponential fall off at infinity, thus expressing the localisation of the ground states around the Grushin singularity of the manifold. The $\sqrt{E}$-factor in the argument of the Bessel function does not affect such conclusion. In fact, irrespective of $\alpha$, given any negative number $-E$ there is one model out of each family of Theorem \ref{thm:H_alpha_fibred_extensions} with ground state energy level precisely equal to $-E$: indeed, it is always possible to choose the extension parameters $\gamma=\gamma^{(\mathrm{I})}$ for $\mathrm{I_R}$ or $\mathrm{I_L}$,  $\gamma=\gamma^{(\mathrm{II}_a)}$ for $\mathrm{II}_a$, and $\Gamma$ for $\mathrm{III}$ such that
\[
 \textstyle\gamma^{(\mathrm{I})}\;=\;\frac{1}{\,1+|a|^2}\,\gamma^{(\mathrm{II}_a)}\;=\;2^{\alpha-1}\big(\gamma_1+\gamma_4 - \sqrt{(\gamma_1-\gamma_4)^2 +4 (\gamma_2^2+\gamma_3^2)}\big)\;=:\;\theta\;<0\,,
\]
in which case Theorem \ref{thm:GroundStateCH} implies 
\[
 -E\;\equiv\;E_0\big(H_{\alpha,\mathrm{R}}^{[ \gamma^{(\mathrm{I})}]}\big)\;=\;E_0\big(H_{\alpha,a}^{[\gamma^{(\mathrm{II}_a)}]}\big)\;=\;E_0\big(H_{\alpha}^{[\Gamma]}\big)\;=\;-\textstyle\left(\frac{\,2\,\theta\,\Gamma(\frac{1+\alpha}{2})}{\,\Gamma(-\frac{1+\alpha}{2})}\right)^{\frac{2}{1+\alpha}}.
\]
In particular, we observe (Figure \ref{fig:3}) that at larger $\alpha$'s the delocalisation away from $x=0$ becomes more pronounced, as the Grushin metric becomes more singular.

\begin{figure}[t!]
	\includegraphics[width=6.5cm]{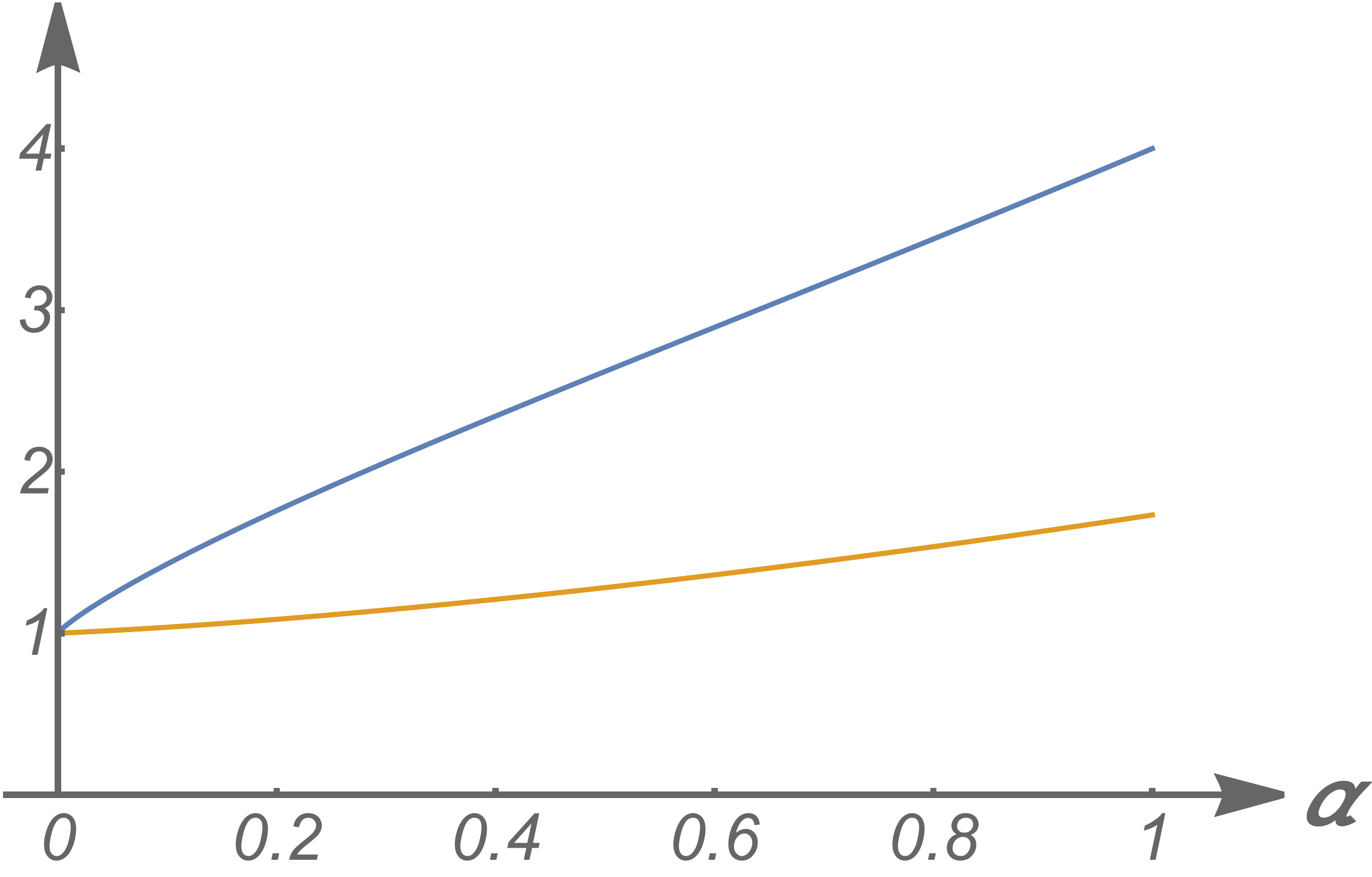}
	\caption{Lower and upper bounds \eqref{eq:estimate-Fgroundstate}}
	\label{fig:variational-estimate}
\end{figure}%

\begin{figure}[t!]
	\includegraphics[width=0.46\textwidth]{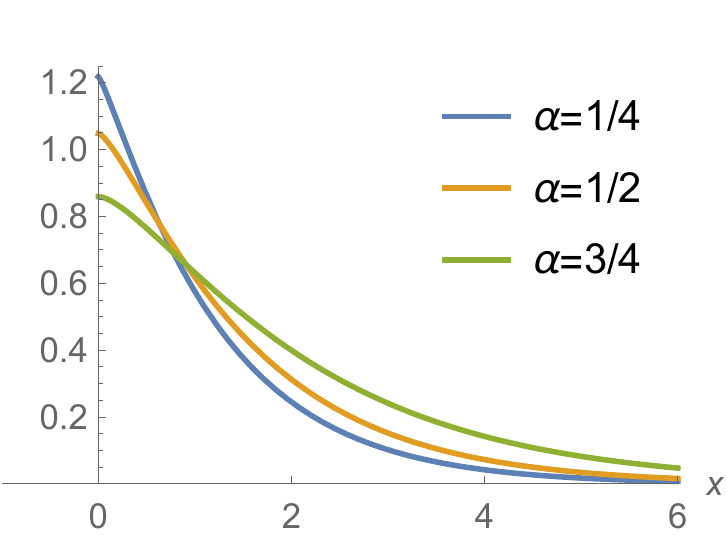} $\quad$  
\includegraphics[width=0.46\textwidth]{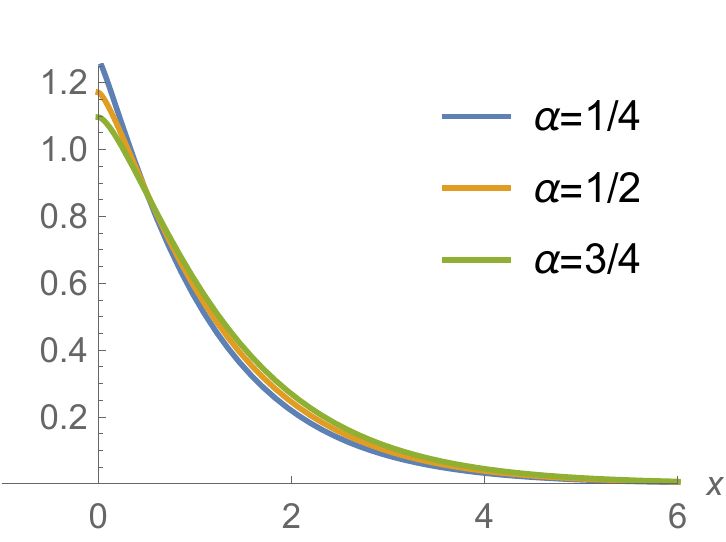}
	\caption{$x$-profile of the (constant in $y$) normalised ground state wave function of the Hamiltonian $H_{\alpha,\mathrm{R}}^{[\gamma]}$ for various $\alpha$. Left: case $\gamma=-1$, thus with different ground state energies depending on $\alpha$ according to \eqref{eq:gs--R}. Right: case with ground state energy $E_0\big(H_{\alpha,\mathrm{R}}^{[\gamma]}\big)=-1$, thus with different $\gamma$ for each $\alpha$, according to \eqref{eq:gsenergyR}. The smaller the parameter $\alpha$, the more pronounced the localisation of the wave function around the metric's singularity.}
	\label{fig:3}
\end{figure}

\begin{figure}[!t]
\captionsetup[subfigure]{labelformat=empty} 
  \centering
  \subfloat[][$\begin{array}{c}\textrm{excited state of mode $k=\pm 1$} \\ \textrm{for $H_{\alpha,a}^{[\gamma]}$ with $a=1,\gamma=1$}\end{array}$]
  {\includegraphics[width=0.45\textwidth]{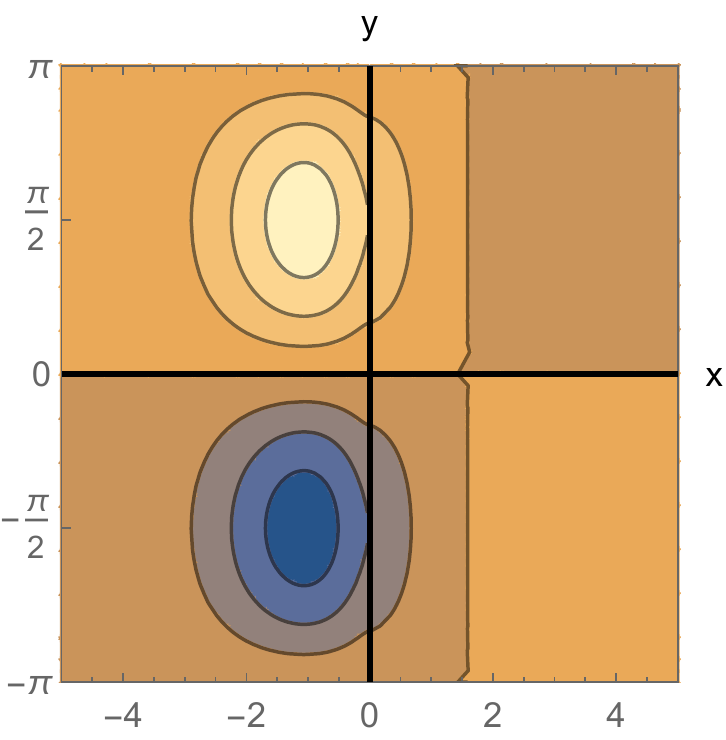} }
  \subfloat[][$\begin{array}{c}\textrm{excited state of mode $k=\pm 1$} \\ \textrm{for $H_{\alpha,a}^{[\gamma]}$ with $a=-2,\gamma=-1$}\end{array}$]
  {\includegraphics[width=0.45\textwidth]{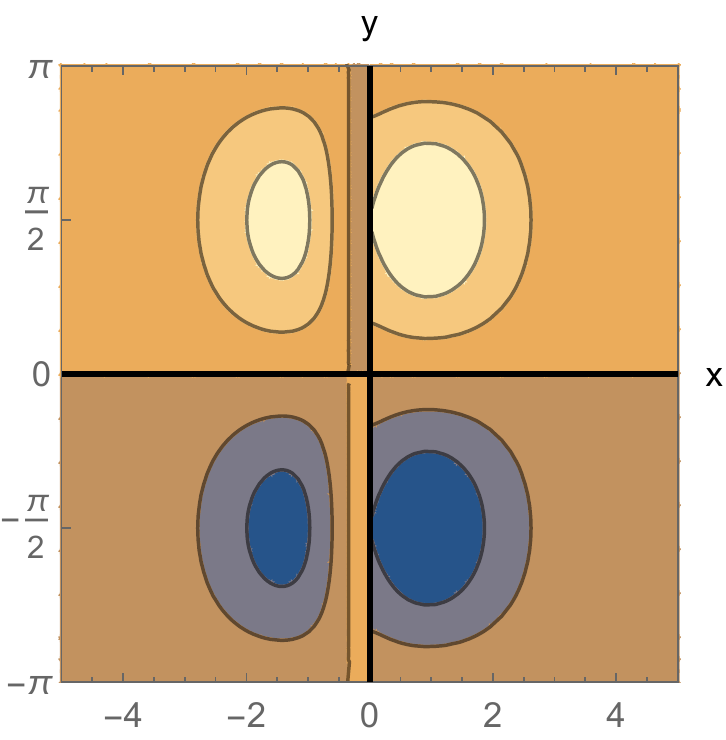} } \\
  \vspace{-0.4cm}
  \subfloat[][$\begin{array}{c}\textrm{eigenstate of mode $k=\pm 1$} \\ \textrm{for $H_{\alpha,\mathrm{F}}$}\end{array}$]
  {\includegraphics[width=0.45\textwidth]{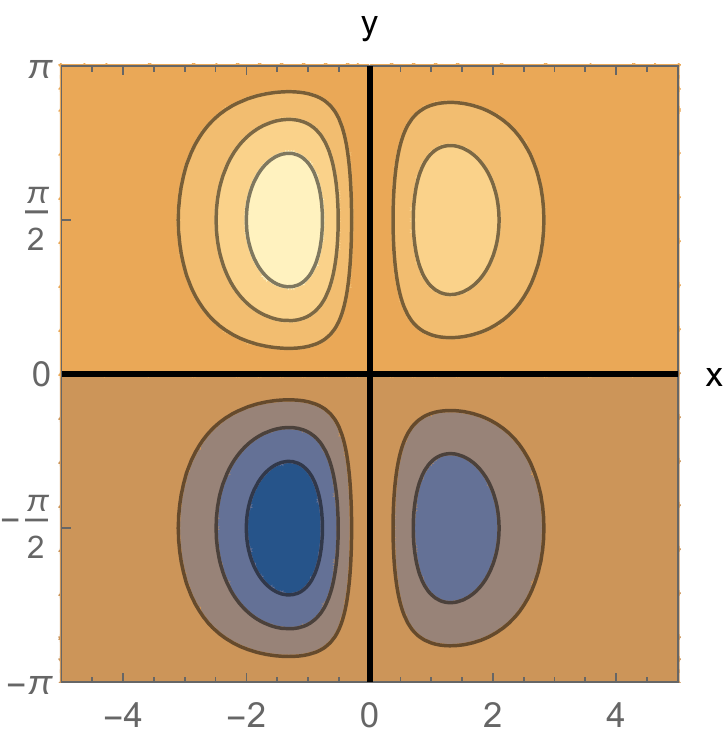} } 
  \subfloat[][$\begin{array}{c}\textrm{excited state of mode $k=\pm 1$} \\ \textrm{for $H_{\alpha,a}^{[\gamma]}$ with $a=1,\gamma=0$  (bridging)}\end{array}$]
  {\includegraphics[width=0.45\textwidth]{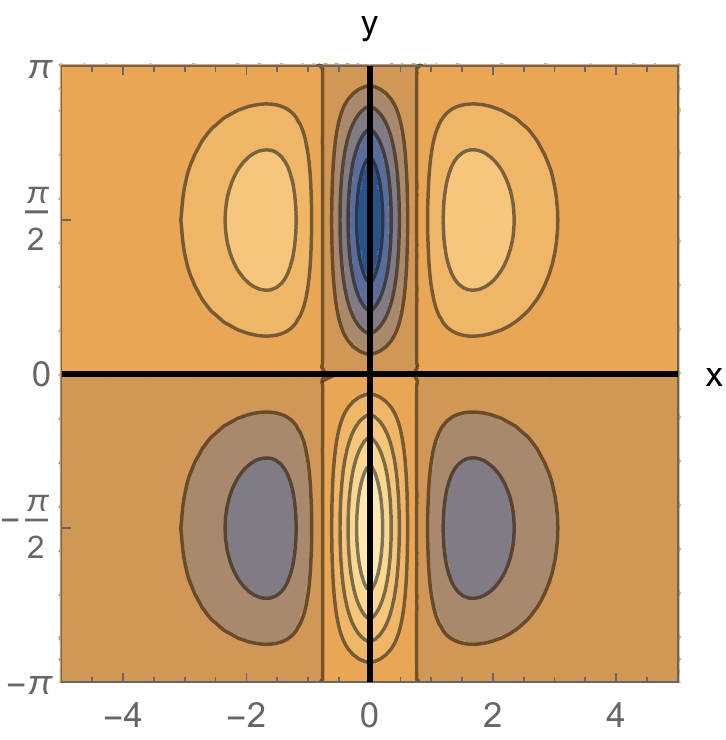} }
  \caption{Contour plot of bound states for various Hamiltonians of Theorem \ref{thm:H_alpha_fibred_extensions}, all with $\alpha=\frac{1}{2}$. The considered bound states, whose energy is embedded in the spectral interval $[0,+\infty)$, are produced by the mode $k=\pm 1$ of the fibred representation of the Hamiltonian described in Section \ref{sec:UEProb} and Theorem \ref{thm:fibredHalpha-spectrum}: they thus display the simplest admissible $y$-oscillation ($\sin y$).}
  \label{fig:contourplots}
\end{figure}

Above the ground state described in Theorem \ref{thm:GroundStateCH}, each Hamiltonian $H_{\alpha,\mathrm{F}}$, $H_{\alpha,\mathrm{R}}^{[\gamma]}$, $H_{\alpha,\mathrm{L}}^{[\gamma]}$, $H_{\alpha,a}^{[\gamma]}$, $H_{\alpha}^{[\Gamma]}$ exhibits an infinite multitude of eigenvalues, all of finite multiplicity, a finite number of them negative, all the others embedded in the essential spectrum $[0,+\infty)$ (Theorem \ref{thm:MainSpectral}). Unlike the ground state wave function, all such excited states do display oscillatory behaviour in the $y$-variable.

Figure \ref{fig:contourplots} displays the contour plots on $(-L,L)_x\times [0,2\pi]_y$, for conveniently large $L>0$, of eigenstate wave functions of various types of Hamiltonians, in all cases with energy level given by the ``$k=\pm 1$ mode'' (in the precise sense of Theorem \ref{thm:fibredHalpha-spectrum} below), meaning that the $y$-oscillation is of the form $\sin y$. At one extreme of the range of possible behaviours there is the Friedrichs extension, whose bound states are well confined \emph{away} from the Grushin singularity $x=0$. Intermediate behaviours are those with some degree of discontinuity at $x=0$, in which case the transmission protocol governed by the corresponding Hamiltonian is affected by partial absorption at the Grushin singularity. At the other extreme, the distinguished protocol governed by $H_{\alpha,a}^{[\gamma]}$ with $a=1$ and $\gamma=0$ results instead in a smooth behaviour of the eigenstate wave functions (like the one displayed in Figure \ref{fig:contourplots}) around $x=0$, a signature of complete communication between the two half-cylinders.

In fact, the Hamiltonian $H_{\alpha,a}^{[\gamma]}$ with $a=1$ and $\gamma=0$ imposes the local behaviour (see  \eqref{eq:DHalpha_cond2_limits-1}-\eqref{eq:DHalpha_cond2_limits-2} and \eqref{eq:DHalpha_cond3_IIa} above) 
\begin{equation}\label{eq:bridging_conditions}
 \begin{split}
  \lim_{x\to 0^-}f(x,y)\;&=\;\lim_{x\to 0^+}f(x,y) \\
  \lim_{x\to 0^-}\Big(\frac{1}{\:|x|^\alpha}\,\frac{\partial f(x,y)}{\partial x}\Big)\;&=\;\lim_{x\to 0^+}\Big(\frac{1}{\:|x|^\alpha}\,\frac{\partial f(x,y)}{\partial x}\Big)\,,
 \end{split}
\end{equation}
which quantum-mechanically is interpreted as the continuity of the spatial probability density of the particle in the region around $\mathcal{Z}$ and of the momentum in the direction orthogonal to $\mathcal{Z}$, defined with respect to the weight $|x|^{-\alpha}$ induced by the metric. This occurrence corresponds to the `optimal' transmission across the boundary, with the dynamics developing the best `bridging' between left and right half-cylinder. Such Hamiltonian is indeed referred to as the `\emph{bridging}' realisation of the free Hamiltonian on Grushin cylinder, an extension identified first in \cite[Proposition 3.11]{Boscain-Prandi-JDE-2016} (clearly, here we are able to recover and study it as a distinguished element of the general classification of Theorem \ref{thm:H_alpha_fibred_extensions}).

To clarify the peculiarity of the bridging transmission protocol, we finally come to the last object of the present study, namely the scattering over the Grushin cylinder.

Intuitively speaking, far away from the Grushin singularity $\mathcal{Z}$ the metric tends to become flat and the action $-\Delta_\alpha$ of each free Hamiltonians considered so far tends to resemble that of the free Laplacian $-\Delta$, plus the correction due to the $(|x|^{-1}\partial_x)$-term, on wave functions $f(x,y)$ that are constant in $y$. This suggests that at very large distances a particle evolves free from the effects of the underlying geometry, and one can speak of scattering states of energy $E>0$. The precise shape of the wave function $f_{\mathrm{scatt}}$ of such a scattering state, at this informal level, can be easily guessed to be of the form
\begin{equation}\label{eq:fscatt}
 f_{\mathrm{scatt}}(x,y)\;\sim\;|x|^{\frac{\alpha}{2}}e^{\pm\ii x\sqrt{E}}\qquad \textrm{as }|x|\to +\infty\,.
\end{equation}
Indeed, $-\Delta_\alpha f_{\mathrm{scatt}}\sim Ef_{\mathrm{scatt}}+\frac{\alpha(2+\alpha)}{4|x|^2}f_{\mathrm{scatt}}$, that is, up to a very small $O(|x|^{-2})$-correction, $f_{\mathrm{scatt}}$ is a generalised eigenfunction of $-\Delta_\alpha$ with eigenvalue $E$. In fact, this is fully justified on a mathematically rigorous level, once one studies the scattering of a convenient unitarily equivalent version of $-\Delta_\alpha$ on flat space $L^2(\mathbb{R}_x\times\mathbb{S}^1_y,\ud x\ud y)$: we shall introduce such unitary equivalence in Section \ref{sec:UEProb} and exploit it throughout the present work, and based on the specific analysis of Section \ref{sec:scattering-fibre} we shall see that indeed scattering states are constant in $y$ and of the form \eqref{eq:fscatt}.

We can then examine, as in the standard stationary scattering analysis, the scattering of a flux of particles injected into the Grushin cylinder at large distances and shot towards the singularity $\mathcal{Z}$ with given energy $E>0$: by monitoring the spatial density of the transmitted flux and the reflected flux, normalised with respect to the density of the incident flux, one quantifies in the usual way the `\emph{transmission coefficient}' and `\emph{reflection coefficient}' for the scattering.

Obviously no scattering across the singularity occurs for Friedrichs, or type-$\mathrm{I}_\mathrm{R}$, or type-$\mathrm{I}_\mathrm{L}$ quantum protocols. We shall focus on type-$\mathrm{II}_a$ scattering, as it includes in particular the bridging protocol; the conclusions for type-$\mathrm{III}$ scattering are qualitatively analogous.

\begin{theorem}[Scattering]\label{thm:scattering} Let $\alpha\in[0,1)$, $a\in\mathbb{C}$, $\gamma\in\mathbb{R}$.
 The transmission coefficient $T_{\alpha,a,\gamma}(E)$ and the reflection coefficient $R_{\alpha,a,\gamma}(E)$ at given energy $E>0$, as defined above for the transmission protocol governed by the Hamiltonian $H_{\alpha,a}^{[\gamma]}$, are given by
  \begin{equation}\label{eq:TR-IIa-thm}
  \begin{split}
   T_{\alpha,a,\gamma}(E)\;&=\;\left|\frac{\,E^{\frac{1+\alpha}{2}}(1+e^{\ii\pi\alpha})\,\Gamma(\frac{1-\alpha}{2})\,\overline{a}}{\,E^{\frac{1+\alpha}{2}}\Gamma(\frac{1-\alpha}{2})(1+|a|^2)+\ii\,\gamma\,2^{1+\alpha}e^{\ii\frac{\pi}{2}\alpha}\Gamma(\frac{3+\alpha}{2})} \right|^2 \,, \\
   R_{\alpha,a,\gamma}(E)\;&=\;\left|\frac{\,E^{\frac{1+\alpha}{2}}\Gamma(\frac{1-\alpha}{2})\,(1-|a|^2\,e^{\ii\pi\alpha})+\ii\,\gamma\,2^{1+\alpha}e^{\ii\frac{\pi}{2}\alpha}\Gamma(\frac{3+\alpha}{2})}{\,E^{\frac{1+\alpha}{2}}\Gamma(\frac{1-\alpha}{2})(1+|a|^2)+\ii\,\gamma\,2^{1+\alpha}e^{\ii\frac{\pi}{2}\alpha}\Gamma(\frac{3+\alpha}{2})}\right|^2\,.
  \end{split}
 \end{equation}
  They satisfy 
  \begin{equation}\label{eq:TpEi1-thm}
  T_{\alpha,a,\gamma}(E)+R_{\alpha,a,\gamma}(E)\;=\;1\,,
 \end{equation}
 and when $\gamma=0$ they are independent of $E$ .
 The scattering is reflection-less ($R_{\alpha,a,\gamma}(E)=0$) when
 \begin{equation}\label{eq:Etransm-thm}
   E\;=\;\bigg(\frac{\,2^{1+\alpha}\,\gamma\,\Gamma(\frac{3+\alpha}{2})\sin\frac{\pi}{2}\alpha\,}{\,\Gamma(\frac{1-\alpha}{2})(1-\cos\pi\alpha)\,}\bigg)^{\!\frac{2}{1+\alpha}},	
 \end{equation}
 provided that $\alpha\in(0,1)$, $|a|=1$, and $\gamma>0$. In the high energy limit the scattering is independent of the extension parameter $\gamma$ and one has
 \begin{equation}\label{eq:highenergyscatt-thm}
  \begin{split}
   \lim_{E\to+\infty}T_{\alpha,a,\gamma}(E)\;&=\;\frac{\,2\,|a|^2(1+\cos\pi\alpha)}{(1+|a|^2)^2}\,, \\
   \lim_{E\to+\infty}R_{\alpha,a,\gamma}(E)\;&=\;\frac{\,1+|a|^4-2|a|^2\cos\pi\alpha}{(1+|a|^2)^2}\,,
  \end{split}
 \end{equation}
 whereas in the low energy limit, for $\gamma\neq 0$,
  \begin{equation}\label{eq:lowenergyscatt-thm}
  \begin{split}
   \lim_{E\downarrow 0}T_{\alpha,a,\gamma}(E)\;&=\;0\,, \\
   \lim_{E\downarrow 0}R_{\alpha,a,\gamma}(E)\;&=\;1\,.
  \end{split}
 \end{equation}
\end{theorem}

 Figure \ref{fig:genericTR} displays two representative behaviours of $ T_{\alpha,a,\gamma}(E)$ and $ R_{\alpha,a,\gamma}(E)$.
 
 \begin{figure}[!h]
\captionsetup[subfigure]{labelformat=empty} 
  \centering
  \subfloat[][$\alpha=\frac{2}{3},a=\ii,\gamma=\frac{1}{2}$]
  {\includegraphics[width=0.45\textwidth]{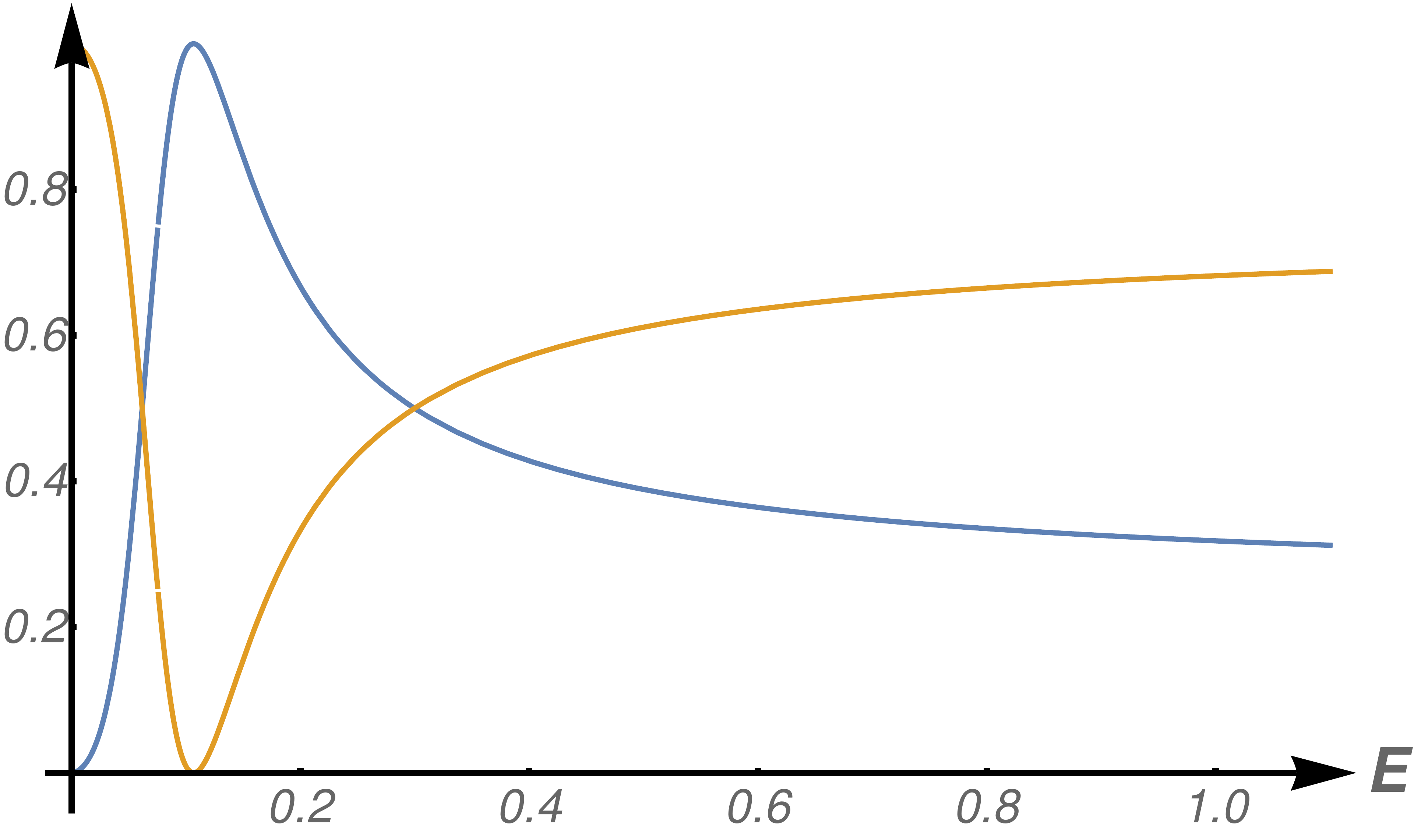} }
  \subfloat[][$\alpha=\frac{2}{5},a=\frac{\ii}{2},\gamma=-\frac{1}{5}$]
  {\includegraphics[width=0.45\textwidth]{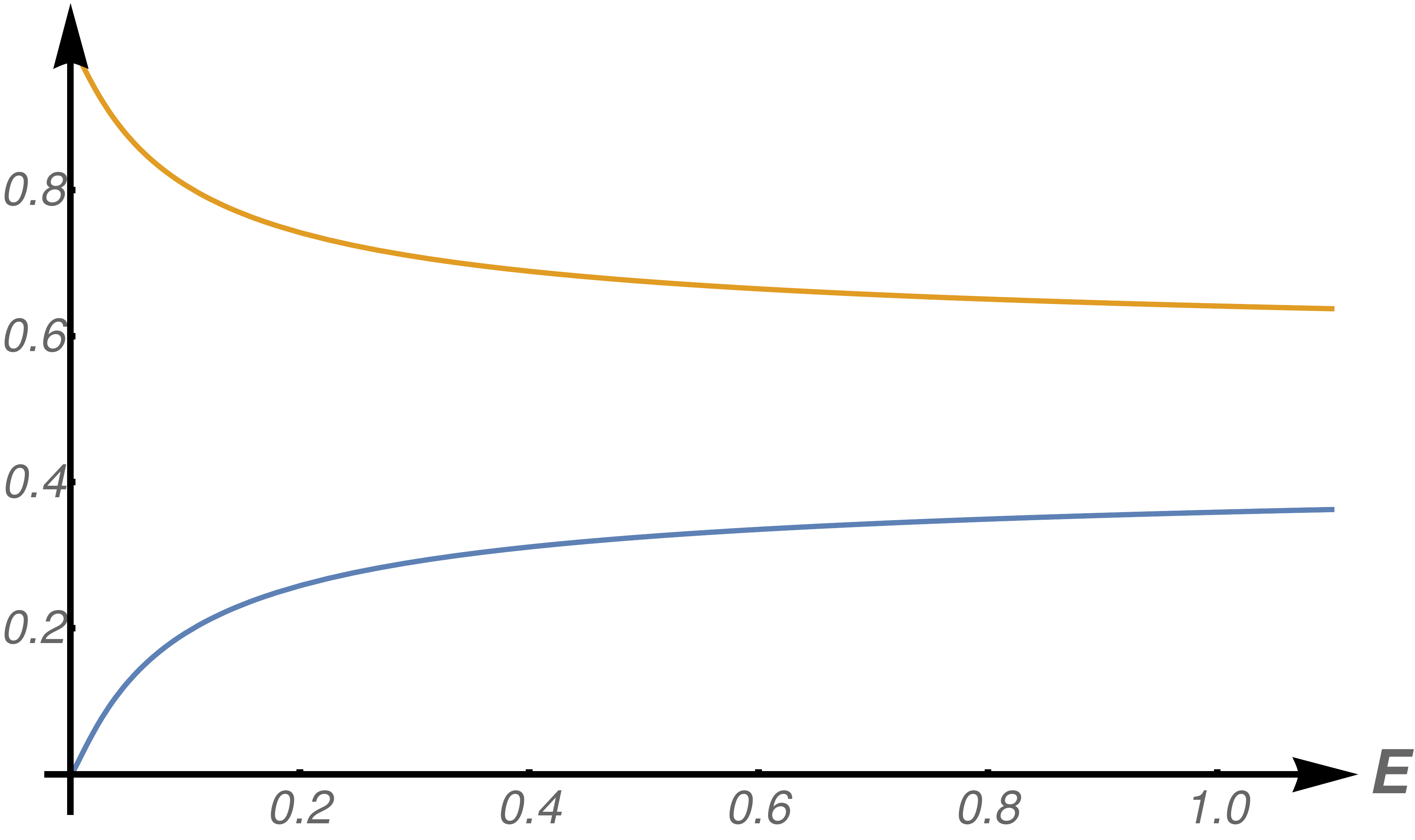} } \\
  \caption{Plot of the coefficients $ T_{\alpha,a,\gamma}(E)$ (blue curve) and $ R_{\alpha,a,\gamma}(E)$ (orange curve) according to formula \eqref{eq:TR-IIa-thm}. In the first case, for a special value of $E$ the scattering is reflection-less.}
  \label{fig:genericTR}
\end{figure}

Specialising the general results of Theorems \ref{thm:H_alpha_fibred_extensions} and \ref{thm:scattering} for the bridging protocol ($a=1$, $\gamma=0$), one may summarise its distinguished status as follows:
\begin{itemize}
 \item \emph{no spatial filter:} the bridging Hamiltonian, as well as all type-$\mathrm{II}_a$ protocols with $a=1$, imposes the local continuity of the wave function at $\mathcal{Z}$ (first identity in \eqref{eq:bridging_conditions}), thus a transmission with no jump in the particle's probability density from one side to the other of the singularity;
 \item \emph{no energy filter}: in the scattering governed by the bridging Hamiltonian, as well as by all type-$\mathrm{II}_a$ protocols with $\gamma=0$, the fraction of transmitted (and reflected) flux does not depend on the incident energy,  
 \begin{equation}
 \begin{split}
   T^{\mathrm{bridg.}}(E)\;:=\;T_{\alpha,1,0}(E)\;&=\;\frac{1}{2}\,(1+\cos\pi\alpha)\,, \\
   R^{\mathrm{bridg.}}(E)\;:=\; R_{\alpha,1,0}(E)\;&=\;\frac{1}{2}\,(1-\cos\pi\alpha)\,,
  \end{split}
\end{equation}
 meaning that the singularity does not act as a filter in the energy.  
\end{itemize}

%

Scattering-wise, one last observation is surely worthwhile. At the upper edge of the considered range for the parameter $\alpha$, based on \eqref{eq:TR-IIa-thm} above we find
	\begin{equation}
		\begin{split}
			&\lim_{\alpha \uparrow 1} T_{\alpha,a,\gamma}(E) \;=\; 0 \, , \\
			&\lim_{\alpha \uparrow 1} R_{\alpha,a,\gamma}(E) \;=\; 1 \, .
		\end{split}
	\end{equation}
Thus, the more singular the metric, the less transmitting the type-$\mathrm{II}_a$ protocol, up to the threshold $\alpha=1$ corresponding to the regime of geometric quantum confinement, where indeed the scattering becomes transmission-less (complete reflection).

We should like to make two additional remarks on the previous results. First, it is worth underlying that the occurrence of the scattering only in the mode $k=0$ is inherently connected with the ``shape'' of the Grushin cylinder as $|x| \to +\infty$. If $\alpha \in(0,1)$, the expression \eqref{eq:Grushin_Metric} of the metric indicates that the larger $|x|$ the more the cylinder shrinks in the transversal direction, up to closing to a single point at infinity: this forces the incident particle incoming from infinity to only have zero angular momentum. When $\alpha=0$ the Grushin cylinder is an actual infinite cylinder in the Euclidean metric: thus, if one were to replace the singular interaction, supported on the circle $\{0\}\times\mathbb{S}^1_y$, of the models considered here, with a localised potential around $x=0$ and with $y$ dependence, one would then be able to engineer a flow of incoming particles with non-zero angular momentum, thus spiraling around the cylinder's axis and scattering through different sectors. In the present model, instead, the \emph{local}  boundary conditions at the singularity locus are $y$-\emph{independent}, also when $\alpha=0$: this makes the analysis independent in each sector.

Second, concerning the possible existence of the special energy value \eqref{eq:Etransm-thm} at which the scattering is reflection-less, this is a phenomenon one is familiar with already from toy models such as the one-dimensional scattering over a finite rectangular barrier, and occurs when the incoming wave at that energy can ``conspire in the most efficient way with the boundary conditions at the scattering centre, so as to have zero reflection. It is worth observing that as $\alpha\to 1^-$, namely when the magnitude of the singularity of the metric reaches the threshold beyond which there is only geometric quantum confinement, thus no scattering, the reflection-less energy \eqref{eq:Etransm-thm} decreases with $\alpha$ up to vanishing: this means that at the $\alpha=1$ threshold the reflection-less scattering disappears, consistently with the fact that the whole scattering is inhibited.

Theorems \ref{thm:positivity}, \ref{thm:MainSpectral}, \ref{thm:GroundStateCH}, and \ref{thm:scattering}, and Proposition \ref{prop:Friedrichs-groundstate} are proved in Sect.~\ref{sec:inverseunitarytransf} after an amount of preparations, that is: the analysis set-up in a convenient, unitarily equivalent framework (Section \ref{sec:UEProb}) in which the fibred structure of the extension problem emerges, as a consequence of the compactness of the $y$-variable of $M$; the spectral analysis in each fibre (Section \ref{sec:spectral-in-fibre}); the scattering analysis on the zero-mode fibre (Section \ref{sec:scattering-fibre}); the reconstruction of the spectral content of the fibred extensions (Sections \ref{prop:spectra-in-direct-sum} and \ref{sec:inverseunitarytransf}).

\section{Unitarily equivalent problem} \label{sec:UEProb}

As a matter of fact, the self-adjoint extension problem for $H_\alpha$ in $L^2(M, \ud \mu_\alpha)$ is more conveniently dealt with in a suitable unitarily equivalent re-formulation that exploits the natural fibered structure of the Hilbert space once the Fourier transform is taken in the compact variable $y$. In this Section we collect the relevant properties of this construction required for our subsequent analysis, referring to our previous works \cite{GMP-Grushin-2018,GMP-Grushin2-2020} for further details and proofs.

Keeping in mind the left-right orthogonal decomposition \eqref{eq:decomp+-}, we switch from the Hilbert spaces $L^2(M^\pm,\ud\mu_\alpha)$ to the new Hilbert spaces
\begin{equation}\label{eq:global-unitary-pm}
  \cH^\pm\;:=\;\mathcal{F}_2^{\pm} U_\alpha^{\pm}L^2(M^\pm,\ud\mu_\alpha)
\end{equation}
where $U_\alpha^{\pm}$ and $\mathcal{F}_2^{\pm}$ are unitary transformations defined, respectively, as
\begin{equation}\label{eq:unit1}
\begin{split}
 U_\alpha^\pm:L^2(\mathbb{R}^\pm\times\mathbb{S}^1,|x|^{-\alpha}\ud x\ud y)&\stackrel{\cong}{\longrightarrow}L^2(\mathbb{R}^\pm\times\mathbb{S}^1,\ud x\ud y) \\
 f &\; \mapsto\;\phi\;:=\;  |x|^{-\frac{\alpha}{2}}f\,,
\end{split}
\end{equation}
and
\begin{equation}\label{eq:defF2}
 \begin{split}
  \mathcal{F}_2^{\pm}:L^2(\mathbb{R}^\pm\times\mathbb{S}^1,\ud x\ud y)&\stackrel{\cong}{\longrightarrow}
 L^2(\mathbb{R}^\pm,\ud x)\otimes\ell^2(\mathbb{Z})\,, \\
  \phi &\;\mapsto\;\psi\;\equiv\;(\psi_k)_{k\in\mathbb{Z}}\,, \\
  e_k(y)\;:=\;\frac{e^{\ii k y}}{\sqrt{2\pi}}\,,&\qquad \psi_k(x)\,:=\int_0^{2\pi}\overline{e_k(y)}\,\phi(x,y)\,\ud y\,,\qquad x\in\mathbb{R}^{\pm}
 \end{split}
\end{equation}
(thus, $\phi(x,y)=\sum_{k\in\mathbb{Z}}\psi_k(x)e_k(y)$ in the $L^2$-convergent sense).

Up to canonical isomorphisms,
\begin{equation}\label{eq:EquivalentHSpaces}
	\mathcal{H}^\pm \;= \; \bigoplus_{k \in \mathbb{Z}} L^2(\mathbb{R}^\pm, \ud x) \; \cong  \; \ell^2(\mathbb{Z},L^2(\mathbb{R}^\pm)) \; \cong \; L^2(\mathbb{R}^\pm) \otimes \ell^2(\mathbb{Z})\,,
\end{equation}
therefore $\mathcal{H}^+$ and $\mathcal{H}^-$ display a natural \emph{`constant-fibre' orthogonal sum structure}
\begin{equation}\label{L^2directDecomp}
  \cH^\pm\;=\;\bigoplus_{k\in\mathbb{Z}} \;\mathfrak{h}^\pm\,,\qquad \mathfrak{h}_\pm\;:=\;L^2(\mathbb{R}^\pm,\ud x)
\end{equation}
with \emph{constant fiber} $\mathfrak{h}_\pm$ and scalar product 
\begin{equation}
 \big\langle (\psi_k)_{k\in\mathbb{Z}} , (\widetilde{\psi}_k)_{k\in\mathbb{Z}} \big\rangle_{\cH^{\pm}}=\;\sum_{k\in\mathbb{Z}}\,\int_{\mathbb{R}^\pm}\overline{\psi_k(x)}\,\widetilde{\psi}_k(x)\,\ud x\;\equiv\;\sum_{k\in\mathbb{Z}}\,\langle \psi_k,\widetilde{\psi}_k\rangle_{\mathfrak{h}^\pm}\,.
\end{equation}

In complete analogy one defines $\mathcal{F}_2:=\mathcal{F}_2^-\oplus\mathcal{F}_2^+$, $U_\alpha:=U_\alpha^-\oplus U_\alpha^+$, whence $\mathcal{F}_2U_\alpha=\mathcal{F}_2^- U_\alpha^-\oplus \mathcal{F}_2^+ U_\alpha^+$, and
\begin{equation}\label{eq:Hxispace}
 \begin{split}
   \cH\;&:=\;\mathcal{F}_2U_\alpha L^2(M,\ud\mu_\alpha)\;\cong\;\ell^2(\mathbb{Z},L^2(\mathbb{R},\ud x))\;\cong\;\cH^-\oplus\cH^+\;\cong\;\bigoplus_{k\in\mathbb{Z}}\;\mathfrak{h}\,, \\
    \mathfrak{h}\;&:=\;L^2(\mathbb{R}^-,\ud x)\oplus L^2(\mathbb{R}^+,\ud x)\;\cong\;L^2(\mathbb{R},\ud x)\,.
 \end{split}
\end{equation}

In terms of the transformation \eqref{eq:global-unitary-pm} we then switch from the operators $H_\alpha^\pm$ to their unitarily equivalent counterparts
\begin{equation}\label{eq:unitary_transf_pm}
 \mathscr{H}_\alpha^\pm\;:=\;\mathcal{F}^{\pm}_2\, U_\alpha^\pm \,H_\alpha^\pm \,(U_\alpha^\pm)^{-1}(\mathcal{F}_2^{\pm})^{-1}
\end{equation}
acting on $\cH^{\pm}$. Explicitly,
\begin{equation}\label{eq:actiondomainHalpha}
  \begin{split}
  \mathcal{D}(\mathscr{H}_\alpha^\pm)\;&=\;\Big\{\psi\equiv(\psi_k)_{k\in\mathbb{Z}}\in \bigoplus_{k\in\mathbb{Z}} L^2(\mathbb{R}^\pm,\ud x)\,\Big|\,\psi\in\mathcal{F}_2^{\pm}C^\infty_c(\mathbb{R}^\pm_x\times\mathbb{S}^1_y)\Big\} \\
    \mathscr{H}_\alpha^\pm\psi\;&=\;\Big(\Big(-\frac{\ud^2}{\ud x^2}+k^2 |x|^{2\alpha}+\frac{\,\alpha(2+\alpha)\,}{4x^2}\Big)\psi_k\Big)_{k\in\mathbb{Z}} \,.
 \end{split}
\end{equation}
Completely analogous formulas hold for  $\mathscr{H}_\alpha$, defined in the obvious way, and acting on the Hilbert space $\cH$ (see \eqref{eq:Hxispace} above).


The self-adjoint extensions of $H_\alpha$ with respect to $L^2(M, \ud \mu_\alpha)$ and their spectral properties are more conveniently read out through the above unitary equivalence from the corresponding $\mathscr{H}_\alpha$ with respect to the Hilbert spaces $\cH$.

In particular, the self-adjoint extension problem was solved in \cite{GMP-Grushin2-2020} based on the crucial circumstance that the adjoint 
of $\mathscr{H}_\alpha$ 
is reduced, with respect to the Hilbert space orthogonal decomposition \eqref{eq:Hxispace}, 
as
\begin{equation}\label{eq:Halphaadj-closure-decomp}
 \begin{split}
   \mathscr{H}_\alpha^*\;&=\;\bigoplus_{k\in\mathbb{Z}} \,A_\alpha(k)^*\,, \\
 \end{split}
\end{equation}
where the auxiliary operators $A_\alpha(k)=A^-_\alpha(k)\oplus A^+_\alpha(k)$ all act on the `bilateral fibre' Hilbert space $ \mathfrak{h}\cong L^2(\mathbb{R}^-,\ud x)\oplus L^2(\mathbb{R}^+,\ud x)$ and are defined by
\begin{equation}\label{eq:Axi}
 \begin{split}
  \mathcal{D}(A_\alpha(k))\;&:=\;C^\infty_c(\mathbb{R}^-)\boxplus C^\infty_c(\mathbb{R}^+) \\
  A_\alpha^\pm(k)\;&:=\;-\frac{\ud^2}{\ud x^2}+k^2 |x|^{2\alpha}+\frac{\,\alpha(2+\alpha)\,}{4x^2}\,.
 \end{split}
\end{equation}
(It is worth remarking that instead $\mathscr{H}_\alpha\varsubsetneq\bigoplus_{k\in\mathbb{Z}} A_\alpha(k)$.) Tacitly, the symbols for adjoint have different meaning in the two sides of \eqref{eq:Halphaadj-closure-decomp}: each one refers to the corresponding Hilbert space.

This observation bridges the self-adjoint extension problem for $\mathscr{H}_\alpha$ (which has \emph{infinite} deficiency index) to the collection of the self-adjoint extension problems for the $A_\alpha(k)$'s (each of which has deficiency index \emph{equal to two}), up to a final reconstruction of the global extensions by means of the direct sum \eqref{eq:Halphaadj-closure-decomp}.

And based on the very structure \eqref{eq:Halphaadj-closure-decomp} we shall proceed in this work to characterise the spectral properties of self-adjoint extensions of $\mathscr{H}_\alpha$ 
by determining first the spectral properties of self-adjoint extensions of the $A_\alpha(k)$'s.

For the time being let us continue with listing relevant properties and formulas from the previous analysis \cite{GMP-Grushin2-2020}, which are going to be useful in the next Sections.

Let us re-cap first of all the solution to the self-adjoint extension problem for each $A_\alpha(k)$. Each such operator is densely defined, symmetric, and lower semi-bounded on $L^2(\mathbb{R})$, with
\begin{equation}\label{eq:lowerboundAak}
 A_\alpha(k)\;\geqslant\;(1+\alpha)\big(\textstyle{\frac{2+\alpha}{4}}\big)^{\frac{\alpha}{1+\alpha}}|k|^{\frac{2}{1+\alpha}}\,\mathbbm{1}\,.
\end{equation}

One finds
\begin{equation}\label{eq:Afstar}
 \begin{split}
  A_{\alpha}(k)^*\;&=\;A_{\alpha}^-(k)^*\oplus A_{\alpha}^+(k)^*\,, \\
  \mathcal{D}(A_{\alpha}^\pm(k)^*)\;&=\;
  \left\{\!\!
  \begin{array}{c}
   g^\pm\in L^2(\mathbb{R}^\pm,\ud x)\;\;\textrm{such that} \\
   \big(-\frac{\ud^2}{\ud x^2}+k^2 |x|^{2\alpha}+\frac{\,\alpha(2+\alpha)\,}{4x^2}\big)g^\pm\in L^2(\mathbb{R}^\pm,\ud x)
  \end{array}
  \!\!\right\}, \\
   A_{\alpha}^\pm(k)^*g^\pm\;&=\;\Big(-\frac{\ud^2}{\ud x^2}+k^2 |x|^{2\alpha}+\frac{\,\alpha(2+\alpha)\,}{4x^2}\Big) g^\pm\,,
 \end{split}
\end{equation}
and proves that a generic $g\in \mathcal{D}(A_\alpha(k)^*)$ has the short-range asymptotics
\begin{equation}\label{eq:gbilateralasympt}
 g(x)\;\equiv\;\begin{pmatrix} g^-(x) \\ g^+(x) \end{pmatrix}\;\stackrel{x\to 0}{=} \;\begin{pmatrix} g_0^- \\ g_0^+ \end{pmatrix} |x|^{-\frac{\alpha}{2}}+\begin{pmatrix} g_1^- \\ g_1^+ \end{pmatrix}|x|^{1+\frac{\alpha}{2}}+o(|x|^{\frac{3}{2}})
\end{equation}
for suitable $g_0^\pm,g_1^\pm\in\mathbb{C}$ given by the limits
 \begin{equation}\label{eq:bilimitsg0g1}
  \begin{split}
   g_0^\pm\;&=\;\lim_{x\to 0^\pm} |x|^{\frac{\alpha}{2}}g^\pm(x) \\
   g_1^\pm\;&=\;\lim_{x\to 0^\pm} |x|^{-(1+\frac{\alpha}{2})}\big(g^\pm(x)-g_0^\pm |x|^{-\frac{\alpha}{2}}\big)\,.
  \end{split}
 \end{equation}

In turn, for each fixed $k\in\mathbb{Z}$, one demonstrates that the self-adjoint extensions of $A_\alpha(k)$ can be grouped into the following families
\begin{equation}\label{eq:families-ext-Aak}
 \begin{split}
  & A_{\alpha,\mathrm{F}}(k) \qquad\qquad\quad\;\;\, \textrm{(the Friedrichs extension)}\\
  & \{A_{\alpha,\mathrm{R}}^{[\gamma]}(k)\,|\,\gamma\in\mathbb{R}\} \\
  & \{A_{\alpha,\mathrm{L}}^{[\gamma]}(k)\,|\,\gamma\in\mathbb{R}\} \\
  & \{A_{\alpha,a}^{[\gamma]}(k)\,|\,\gamma\in\mathbb{R}\}\qquad\textrm{for fixed $a\in\mathbb{C}$} \\
  & \{A_{\alpha}^{[\Gamma]}(k)\,|\,\Gamma\equiv(\gamma_1,\gamma_2,\gamma_3,\gamma_4)\in\mathbb{R}^4\}\,,
 \end{split}
\end{equation}
where each operator appearing in \eqref{eq:families-ext-Aak} is a restriction of $A_\alpha(k)^*$ (hence acts as the differential operator in \eqref{eq:Afstar} above) on the corresponding domain, given, in terms of the boundary values $g_0^\pm,g_1^\pm$, respectively by 
\begin{eqnarray}
 \mathcal{D}(A_{\alpha,\mathrm{F}}(k))\!\!&:=&\!\!\{g\in\mathcal{D}(A_\alpha(k)^*)\,|\,g_0^-=0=g_0^+\}\,,  \label{eq:extAak-F} \\
 \mathcal{D}(A_{\alpha,\mathrm{R}}^{[\gamma]}(k))\!\!&:=&\!\!\{g\in\mathcal{D}(A_\alpha(k)^*)\,|\,g_0^-=0\,,\;g_1^+=\gamma g_0^+\}\,,  \label{eq:extAak-IR}\\
 \mathcal{D}(A_{\alpha,\mathrm{L}}^{[\gamma]}(k))\!\!&:=&\!\!\{g\in\mathcal{D}(A_\alpha(k)^*)\,|\,g_0^+=0\,,\;g_1^-=\gamma g_0^-\}\,,  \label{eq:extAak-IL} \\
 \mathcal{D}(A_{\alpha,a}^{[\gamma]}(k))\!\!&:=&\!\!\left\{g\in\mathcal{D}(A_\alpha(k)^*)\left|\!
   \begin{array}{c}
    g_0^+=a\,g_0^- \\
    g_1^-+\overline{a}\,g_1^+=\gamma\, g_0^-
   \end{array}\!\!
   \right.\right\},  \label{eq:extAak-IIa}\\
 \mathcal{D}(A_{\alpha}^{[\Gamma]}(k))\!\!&:=&\!\!\left\{g\in\mathcal{D}(A_\alpha(k)^*)\left|\!
   \begin{array}{c}
    g_1^-=\gamma_1 g_0^-+\zeta g_0^+ \\
    g_1^+=\overline{\zeta} g_0^-+\gamma_4 g_0^+ \\
    \zeta:=\gamma_2+\ii\gamma_3
   \end{array}\!\!
   \right.\right\}. \label{eq:extAak-III}
\end{eqnarray}

This completes the self-adjoint extension problem for $A_\alpha(k)$.

Now, the actual family of self-adjoint extensions of $\mathscr{H}_\alpha$ (self-adjoint restrictions of $\mathscr{H}_\alpha^*$) is a vast collection of operators identified by suitable boundary conditions as $x\to 0$ that in general couple different $k$-modes and in this sense are non-local (in momentum) and hence physically not relevant (the size of this enormous family, in fact its full classification, can be inferred from \cite[Eq.~(6.16)]{GMP-Grushin2-2020}).

Yet, a sub-class of physically meaningful self-adjoint extensions can be singled out, constituted by the operators
\begin{eqnarray}
  \mathscr{H}_{\alpha,\mathrm{F}}\!\!&:=&\!\!\bigoplus_{k\in\mathbb{Z}}\,A_{\alpha,\mathrm{F}}(k)\,,\label{eq:HalphaFriedrichs_unif-fibred} \\
  \mathscr{H}_{\alpha,\mathrm{R}}^{[\gamma]}\!\!&:=&\!\!\bigoplus_{k\in\mathbb{Z}}A_{\alpha,\mathrm{R}}^{[\gamma]}(k)\,, \label{eq:HalphaR_unif-fibred} \\
  \mathscr{H}_{\alpha,\mathrm{L}}^{[\gamma]}\!\!&:=&\!\!\bigoplus_{k\in\mathbb{Z}}A_{\alpha,\mathrm{L}}^{[\gamma]}(k)\,, \label{eq:HalphaL_unif-fibred} \\
  \mathscr{H}_{\alpha,a}^{[\gamma]}\!\!&:=&\!\!\bigoplus_{k\in\mathbb{Z}}A_{\alpha,a}^{[\gamma]}(k)\,,\label{eq:Halpha-IIa_unif-fibred} \\
  \mathscr{H}_{\alpha}^{[\Gamma]}\!\!&:=&\!\!\bigoplus_{k\in\mathbb{Z}}A_{\alpha}^{[\Gamma]}(k)\,,\label{eq:Halpha-III_unif-fibred}
\end{eqnarray}
for all possible $\gamma\in\mathbb{R}$, $\Gamma\in\mathbb{R}^4$, $a\in\mathbb{C}$, where the above sums are referred to the Hilbert space orthogonal decomposition \eqref{eq:Hxispace} for the space $\cH$ (analogously to \eqref{eq:Halphaadj-closure-decomp} above). Each operator of type \eqref{eq:HalphaFriedrichs_unif-fibred}-\eqref{eq:Halpha-III_unif-fibred} displays in its domain boundary conditions of self-adjointness, as $x\to 0$, that have both the same form and the same `magnitude' (hence the same $\gamma$-parameter, or $\Gamma$-parameter) irrespective of the transversal momentum number $k$. For this feature, we referred to them as `\emph{uniformly fibred extensions}' of $\mathscr{H}_\alpha$. In particular, one proves that $\mathscr{H}_{\alpha,\mathrm{F}}$ is the Friedrichs extension.

Finally, unfolding back the unitary transformation \eqref{eq:unitary_transf_pm} one demonstrates that in correspondence to \eqref{eq:HalphaFriedrichs_unif-fibred}-\eqref{eq:Halpha-III_unif-fibred} one obtains, acting on the original physical space $L^2(M,\ud\mu_\alpha)$, precisely the families of extension of Theorem \ref{thm:H_alpha_fibred_extensions}, that is,
\begin{eqnarray}
 H_{\alpha,\mathrm{F}} \!\!&=&\!\!(U_\alpha)^{-1}(\mathcal{F}_2)^{-1} \mathscr{H}_{\alpha,\mathrm{F}}\,\mathcal{F}_2\,U_\alpha\,, \label{eq:inverseUniary_Fri}\\
 H_{\alpha,\mathrm{R}}^{[\gamma]} \!\!&=&\!\!(U_\alpha)^{-1}(\mathcal{F}_2)^{-1} \mathscr{H}_{\alpha,\mathrm{R}}^{[\gamma]}\,\mathcal{F}_2\,U_\alpha\,, \label{eq:inverseUniary_IR}\\
 H_{\alpha,\mathrm{L}}^{[\gamma]} \!\!&=&\!\!(U_\alpha)^{-1}(\mathcal{F}_2)^{-1} \mathscr{H}_{\alpha,\mathrm{L}}^{[\gamma]}\,\mathcal{F}_2\,U_\alpha\,, \\
 H_{\alpha,a}^{[\gamma]} \!\!&=&\!\!(U_\alpha)^{-1}(\mathcal{F}_2)^{-1}\mathscr{H}_{\alpha,a}^{[\gamma]}\,\mathcal{F}_2\,U_\alpha\,,\\
 H_{\alpha}^{[\Gamma]} \!\!&=&\!\!(U_\alpha)^{-1}(\mathcal{F}_2)^{-1}\mathscr{H}_{\alpha}^{[\Gamma]}\,\mathcal{F}_2\,U_\alpha\,. \label{eq:inverseUniary_III}
\end{eqnarray}

This concludes the concise summary of relevant results from the rather laborious analysis of \cite{GMP-Grushin2-2020}.

There is in fact one last technical piece of information that we need to extract from \cite{GMP-Grushin2-2020}, owing to its relevance in the present work. It concerns the family \eqref{eq:families-ext-Aak} of self-adjoint extensions of the operator $A_\alpha(k)$, for fixed $k\in\mathbb{Z}$, with respect to the bilateral fibre Hilbert space $\mathfrak{h}\cong L^2(\mathbb{R})$.

Formulas \eqref{eq:families-ext-Aak}-\eqref{eq:extAak-III} label all such extensions in terms of convenient parameters ($\gamma\in\mathbb{R}$ or $\Gamma\in\mathbb{R}^4$) that allow in each case for a transparent expression of the boundary conditions of self-adjointness in terms of the boundary values $g_0^\pm,g_1^\pm$. This is the final version of a \emph{more intrinsic} parametrisation of the extension family, which is provided by the application of the  Kre\u{\i}n-Vi\v{s}ik-Birman extension scheme (see \cite[Eq.~(5.14)]{GMP-Grushin2-2020}). According to such scheme, the self-adjoint extensions of $A_\alpha(k)$ are in one-to-one correspondence with the self-adjoint operators acting on the one- and two-dimensional spaces $\mathbb{C}$ and $\mathbb{C}^2$, and hence, respectively, with the multiplications by a real number and by a hermitian $2\times 2$ matrix. Such intrinsic labelling operators are the `\emph{Birman operator parameters}' of the Kre\u{\i}n-Vi\v{s}ik-Birman theory \cite[Theorem 5]{GMO-KVB2017} and their explicit form in the present case is given as follows when $k\neq 0$.

\begin{itemize}
	\item The Birman parameter for the sub-families $\{A_{\alpha,\mathrm{R}}^{[\gamma]}(k)\,|\,\gamma\in\mathbb{R}\}$ and $\{A_{\alpha,\mathrm{L}}^{[\gamma]}(k)\,|\,\gamma\in\mathbb{R}\}$ is the $\mathbb{C}\to\mathbb{C}$ multiplication by the real number $\beta_k$, that is linked to the parameter $\gamma$ by the formula
	\begin{equation}\label{eq:GammaIRIL}
	\beta_k\;=\; \frac{2^{\frac{1-\alpha}{1+\alpha}}(1+\alpha)^{\frac{2\alpha}{1+\alpha}}|k|^{\frac{2}{1+\alpha}}}{\Gamma(\frac{1-\alpha}{1+\alpha})} \left( \frac{1+\alpha}{|k|} \gamma + 1 \right)\qquad (k\neq 0),
	\end{equation}
	and the extension $A_{\alpha,\mathrm{R}}^{[\gamma]}(k)$ (resp., $A_{\alpha,\mathrm{L}}^{[\gamma]}(k)$) is identified by $\beta_k$.
	\item The Birman parameter for the sub-family $\{A_{\alpha,a}^{[\gamma]}(k)\,|\,\gamma\in\mathbb{R}\}$, at fixed $a\in\mathbb{C}$, is the $\mathbb{C}\to\mathbb{C}$ multiplication by the real number $\tau_k$, that is linked to the parameter $\gamma$ by the formula
	\begin{equation}\label{eq:GammaTauIIa}
	\tau_k \;=\; \frac{2^{\frac{1-\alpha}{1+\alpha}}(1+\alpha)^{\frac{2\alpha}{1+\alpha}}|k|^{\frac{2}{1+\alpha}}}{\Gamma(\frac{1-\alpha}{1+\alpha})} \left(1+\frac{(1+\alpha) \gamma}{(1+|a|^2) |k|} \right)\qquad (k\neq 0),
	\end{equation}
	and the extension $A_{\alpha,a}^{[\gamma]}(k)$ is identified by $\tau_k$.
	\item The Birman parameter for the sub-family $\{A_{\alpha}^{[\Gamma]}(k)\,|\,\Gamma\in\mathbb{R}^4\}$ is the $\mathbb{C}^2\to\mathbb{C}^2$ linear map induced by the hermitian matrix
	\begin{equation}\label{eq:IIkTheorem51}
	\begin{split}
	 			T_k\;&\equiv\;\begin{pmatrix}
			\tau_{1,k} & \tau_{2,k} + \ii \tau_{3,k} \\
			\tau_{2,k} - \ii \tau_{3,k} & \tau_{4,k}
			\end{pmatrix} \\
			&=\; \frac{\,2^{\frac{1-\alpha}{1+\alpha}}(1+\alpha)^{\frac{2\alpha}{1+\alpha}}|k|^{\frac{2}{1+\alpha}}}{\Gamma(\frac{1-\alpha}{1+\alpha})} \begin{pmatrix}
				1+ \frac{1+\alpha}{|k|} \gamma_1 & \frac{1+\alpha}{|k|}(\gamma_2 + \ii \gamma_3) \\
				\frac{1+\alpha}{|k|}(\gamma_2 - \ii \gamma_3) & 1+ \frac{1+\alpha}{|k|} \gamma_4
			\end{pmatrix}\quad (k\neq 0),
	\end{split}
	\end{equation}
	and the extension $\{A_{\alpha}^{[\Gamma]}(k)$ is identified by $T_k$.
\end{itemize}

Observe that the Birman parameters $\beta_k,\tau_k,T_k$ do depend on the integer $k$. Their expressions \eqref{eq:GammaIRIL}-\eqref{eq:IIkTheorem51} (that are derived directly from \cite[Eq.~(3.59), (5,19), (5.20)]{GMP-Grushin2-2020}) would certainly make the final formulas for the boundary conditions of self-adjointness \eqref{eq:extAak-F}-\eqref{eq:extAak-III} unessentially cumbersome -- the $\gamma$- or $\Gamma$-parametrisations are obviously cleaner. However, Birman parameters are relevant because, as a key feature of the general Kre\u{\i}n-Vi\v{s}ik-Birman theory, they encode particularly useful information on the spectral properties of the extensions that they label. We shall make a crucial use of such information in the following.

\section{Spectral analysis in each fibre}\label{sec:spectral-in-fibre}

The goal of this Section is the spectral analysis of the self-adjoint operators \eqref{eq:families-ext-Aak}, all acting on the fibre Hilbert space $\mathfrak{h}$, for different values of the transversal momentum quantum number $k\in\mathbb{Z}$.

\subsection{Spectral analysis in the fibre $k=0$}~ \label{subsec:Spectrum0}

The mode $k=0$ requires a separate analysis, essentially due to the fact that for the self-adjoint extensions of $A_\alpha(0)$ we do not have explicit expressions for the Birman extension parameter (that, as said, carries direct information on the spectra).


There is nothing conceptually deep preventing one to identify the Birman parameter also when $k=0$. Simply, unlike the $A_\alpha(k)$'s when $k\neq 0$, the non-negative expectations $\langle g,A_\alpha(0)g \rangle_{\mathfrak{h}}$ include  zero at their bottom, therefore the Friedrichs extension of $A_\alpha(0)$ does not admit an everywhere defined and bounded inverse on $L^2(\mathbb{R})$. This makes the standard application of the Kre\u{\i}n-Vi\v{s}ik-Birman scheme not applicable to the pivot extension $A_{\alpha,F}(0)$. Yet, the analysis of the extensions of $A_\alpha(0)$ can be equally achieved by characterising first the extensions of the shifted operator $A_\alpha(0)+\mathbbm{1}$ (see \cite[Sect.~4]{GMP-Grushin2-2020} for details), as well as by direct Green function methods (as in \cite{Bruneau-Derezinski-Georgescu-2011}) and the final result is expressed by formulas \eqref{eq:families-ext-Aak}-\eqref{eq:extAak-III}.

In this Subsection we establish the following picture.

\begin{proposition}\label{prop:SpectrumK0} Let $\alpha\in[0,1)$. 
\begin{itemize}
 \item[(i)] The essential spectrum of any self-adjoint extension \eqref{eq:families-ext-Aak} of $A_\alpha(0)$ satisfy
\begin{equation}\label{eq:sess0}
 \begin{split}
  &\sigma_{\mathrm{ess}}(A_{\alpha,\mathrm{F}}(0))\;=\;\sigma_{\mathrm{ess}}\big(A_{\alpha,\mathrm{R}}^{[\gamma]}(0)\big)\;=\;\sigma_{\mathrm{ess}}\big(A_{\alpha,\mathrm{L}}^{[\gamma]}(0)\big) \\
  &\quad =\;\sigma_{\mathrm{ess}}\big(A_{\alpha,a}^{[\gamma]}(0)\big)\;=\;\sigma_{\mathrm{ess}}\big(A_{\alpha}^{[\Gamma]}(0)\big)\;=\;[0,+\infty)
 \end{split}
\end{equation}
for any $\gamma\in\mathbb{R}$, $a\in\mathbb{C}$, $\Gamma\in\mathbb{R}^4$. In all cases, the essential spectrum does not contain embedded eigenvalues.
\item[(ii)] The discrete spectrum of any such extension can be therefore only strictly negative, and moreover it is empty for $A_{\alpha,\mathrm{F}}(0)$, consists of at most one negative non-degenerate eigenvalue for $A_{\alpha,\mathrm{R}}^{[\gamma]}(0)$, $A_{\alpha,\mathrm{L}}^{[\gamma]}(0)$, and $A_{\alpha,a}^{[\gamma]}(0)$, and consists of at most two negative eigenvalues for $A_{\alpha}^{[\Gamma]}(0)$, counted with multiplicity.
\item[(iii)] $A_{\alpha,\mathrm{R}}^{[\gamma]}(0)$ and $A_{\alpha,\mathrm{L}}^{[\gamma]}(0)$ admit one negative eigenvalue, denoted respectively as $E_0\big(A_{\alpha,\mathrm{R}}^{[\gamma]}(0)\big)$ and $E_0\big(A_{\alpha,\mathrm{L}}^{[\gamma]}(0)\big)$, if and only if $\gamma<0$, in which case
\begin{equation}\label{eq:EigenvalueIR}
E_0\big(A_{\alpha,\mathrm{R}}^{[\gamma]}(0)\big)\;=\;E_0\big(A_{\alpha,\mathrm{L}}^{[\gamma]}(0)\big)\;=\; -\left(2\,\frac{\Gamma(\frac{1+\alpha}{2})}{\,\Gamma(-\frac{1+\alpha}{2})}\, \gamma \right)^{\frac{2}{1+\alpha}} \, .
\end{equation}
The corresponding (non-normalised) eigenfunctions are, respectively,
\begin{equation}\label{eq:eigenf-IRIL}
 \begin{split}
  g_{\alpha}^{(\mathrm{I_R})}(x)\;&:=\;
  \begin{cases}
   \qquad 0\,, & x<0\,, \\
   \;\sqrt{x}\,K_{\frac{1+\alpha}{2}}(x\sqrt{E})\,, & x>0\,,
  \end{cases} \\
  g_{\alpha}^{(\mathrm{I_L})}(x)\;&:=\;
  \begin{cases}
   \;\sqrt{-x}\,K_{\frac{1+\alpha}{2}}(-x\sqrt{E})\,, & x<0\,, \\
   \qquad 0\,, & x>0\,,
  \end{cases}
 \end{split}
\end{equation}
 where for short $E\equiv -E_0\big(A_{\alpha,\mathrm{R}}^{[\gamma]}(0)\big)=-E_0\big(A_{\alpha,\mathrm{L}}^{[\gamma]}(0)\big)$.
\item[(iv)] For given $a\in\mathbb{C}$, $A_{\alpha,a}^{[\gamma]}(0)$ admits one negative eigenvalue, denoted as $E_0\big(A_{\alpha,a}^{[\gamma]}(0)\big)$, if and only if $\gamma<0$, in which case
\begin{equation}\label{eq:EigenvalueIIa}
E_0\big(A_{\alpha,a}^{[\gamma]}(0)\big)\;=\; -\left(\frac{2\,\Gamma(\frac{1+\alpha}{2})}{(1+|a|^2)\Gamma(-\frac{1+\alpha}{2})} \,\gamma \right)^{\frac{2}{1+\alpha}} \, .
\end{equation}
The corresponding (non-normalised) eigenfunction is
\begin{equation}\label{eq:eigenf-IIa}
  g_{\alpha}^{(\mathrm{II}_a)}(x)\;:=\;
  \begin{cases}
   \; \sqrt{-x}\,K_{\frac{1+\alpha}{2}}(-x\sqrt{E})\,, & x<0\,, \\
   \;a \sqrt{x}\,K_{\frac{1+\alpha}{2}}(x\sqrt{E})\,, & x>0\,,
  \end{cases}
\end{equation}
 where for short $E\equiv -E_0\big(A_{\alpha,a}^{[\gamma]}(0)\big)$.
\item[(v)] $A_{\alpha}^{[\Gamma]}(0)$ admits at most two negative eigenvalues: exactly two, given by
\begin{equation}\label{eq:III0-evs}
 \begin{split}
  E_0\big(A_{\alpha}^{[\Gamma]}(0)\big)\;&:=\;-\left(\frac{\,2^\alpha\,\Gamma(\frac{1+\alpha}{2})}{\Gamma(-\frac{1+\alpha}{2})} \Big( \gamma_1+\gamma_4 - \sqrt{(\gamma_1-\gamma_4)^2 +4 (\gamma_2^2+\gamma_3^2)} \,\Big)\right)^{\frac{2}{1+\alpha}} \\
  E_1\big(A_{\alpha}^{[\Gamma]}(0)\big)\;&:=\; -\left(\frac{\,2^\alpha\,\Gamma(\frac{1+\alpha}{2})}{\Gamma(-\frac{1+\alpha}{2})}\Big( \gamma_1+\gamma_4 + \sqrt{(\gamma_1-\gamma_4)^2 +4 (\gamma_2^2+\gamma_3^2)} \,\Big)\right)^{\frac{2}{1+\alpha}},
 \end{split}
\end{equation}
with $E_0\big(A_{\alpha}^{[\Gamma]}(0)\big)\leqslant E_1\big(A_{\alpha}^{[\Gamma]}(0)\big)$, 
 if and only if
		\begin{equation}\label{eq:III0-2neg}
			\gamma_1+\gamma_4 \; < \; 0 \qquad \text{and} \qquad \gamma_1 \, \gamma_4 \; > \; \gamma_2^2 + \gamma_3^2\,,
		\end{equation}
 only one negative eigenvalue, the quantity $E_0\big(A_{\alpha}^{[\Gamma]}(0)\big)$ above, if and only if
\begin{equation}\label{eq:III0-only1neg}
  \gamma_1 \, \gamma_4 \; < \; \gamma_2^2 + \gamma_3^2\qquad\textrm{or}\qquad
  \begin{cases}
   \;\gamma_1 \, \gamma_4 \; = \; \gamma_2^2 + \gamma_3^2 \\
   \;\gamma_1+\gamma_4\;<\;0\,,
  \end{cases}
\end{equation}
or no negative eigenvalue at all, if and only if
\begin{equation}\label{eq:III0-none-neg}
  \gamma_1+\gamma_4 - \sqrt{(\gamma_1-\gamma_4)^2 +4 (\gamma_2^2+\gamma_3^2)} \;\geqslant \;0\,.
\end{equation}
The lowest negative eigenvalue $E_0\big(A_{\alpha}^{[\Gamma]}(0)\big)$, if existing, is non-degenerate when additionally $\gamma_1\neq\gamma_4$ or $\gamma_2^2+\gamma_3^2>0$, in which case its (non-normalised) eigenfunction is
\begin{equation}\label{eq:eigenf-III-nondegen}
 g_{\alpha}^{(\mathrm{III})}(x)\;:=\;\begin{cases}
  \!\!
  \begin{array}{l}
   \big({\textstyle \gamma_1-\gamma_4 - \sqrt{(\gamma_1-\gamma_4)^2 +4 (\gamma_2^2+\gamma_3^2)}}\big)\,\times \\
   \qquad \times\,\sqrt{-x}\,K_{\frac{1+\alpha}{2}}(-x\sqrt{E})\,,
  \end{array} 
  & x<0\,, \\
  \;2(\gamma_2-\ii\gamma_3)\,\sqrt{x}\,K_{\frac{1+\alpha}{2}}(x\sqrt{E})\,, & x>0\,,
 \end{cases}
\end{equation}
 where for short $E\equiv -E_0\big(A_{\alpha}^{[\Gamma]}(0)\big)$.
 Instead, the negative $E_0\big(A_{\alpha}^{[\Gamma]}(0)\big)$, if existing, is two-fold degenerate when  additionally $\gamma_1=\gamma_4$ and $\gamma_2=\gamma_3=0$, in which case its two-dimensional eigenspace is spanned by the (non-normalised) eigenfunctions
 \begin{equation}\label{eq:eigenf-III-degen}
 \begin{split}
  g_{\alpha,+}^{(\mathrm{III})}(x)\;&:=\;
  \begin{cases}
   \qquad 0\,, & x<0\,, \\
   \;\sqrt{x}\,K_{\frac{1+\alpha}{2}}(x\sqrt{E})\,, & x>0\,,
  \end{cases} \\
  g_{\alpha,-}^{(\mathrm{III})}(x)\;&:=\;
  \begin{cases}
   \;\sqrt{-x}\,K_{\frac{1+\alpha}{2}}(-x\sqrt{E})\,, & x<0\,. \\
   \qquad 0\,, & x>0\,.
  \end{cases}
 \end{split}
\end{equation}
\end{itemize}
\end{proposition}

As each self-adjoint extension of $A_\alpha(0)$ is a restriction of the adjoint \eqref{eq:Afstar}, the eigenvalue problem takes in all cases the form 
\begin{equation}\label{eq:evproblem0}
 \Big(-\frac{\ud^2}{\ud x^2}+\frac{\,\alpha(2+\alpha)\,}{4x^2}+E\Big) g^\pm\;=\;0
\end{equation}
where $-E$ is the eigenvalue and $g\equiv\begin{pmatrix} g^- \\ g^+ \end{pmatrix}\in\mathfrak{h}$ is the corresponding eigenfunction. This yields two independent ODEs, one on $\mathbb{R}^-$ and one on $\mathbb{R}^+$, yet identical in form: it then suffices to only solve, say, the one with $x>0$. Negative eigenvalues are found by restricting to $E>0$.

Upon re-scaling $\xi:=x\sqrt{E}$, $w(x\sqrt{E}):= x^{-\frac{1}{2}} g^+(x)$, the positive half-line version of \eqref{eq:evproblem0} becomes
\begin{equation}
	\xi^2 \frac{\ud^2 w}{\ud \xi^2}+ \xi \frac{\ud w}{\ud \xi} -(\xi^2 + \nu^2) w \; = \; 0\qquad\quad (\textstyle{\nu:=\frac{1+\alpha}{2}})\,,
\end{equation}
namely a modified Bessel equation, the two linearly independent solution of which are the modified Bessel functions $K_\nu$ and $I_\nu$ \cite[Sect.~9.6]{Abramowitz-Stegun-1964}. These are smooth functions over $\mathbb{R}^+$ satisfying the asymptotics
\begin{equation}\label{eq:k0asymptx0}
 \begin{split}
 I_{\frac{1+\alpha}{2}}(\xi)\;&\stackrel{\xi\downarrow 0}{=}\;\frac{1}{2^{\frac{1+\alpha}{2}} \Gamma({\textstyle\frac{3+\alpha}{2}})} \xi^{\frac{1+\alpha}{2}} + O(\xi^{\frac{5+\alpha}{2}}) \\
 K_{\frac{1+\alpha}{2}}(\xi) \;&\stackrel{\xi\downarrow 0}{=}\;2^{\frac{\alpha-1}{2}} \Gamma({\textstyle\frac{1+\alpha}{2}}) \,	\xi^{-\frac{1+\alpha}{2}} + 2^{-\frac{3+\alpha}{2}} \Gamma({\textstyle\frac{1+\alpha}{2}}) \, \xi^{\frac{1+\alpha}{2}} + O(\xi^{\frac{3-\alpha}{2}})
 \end{split}
\end{equation}
(following from \cite[Eq.~(9.6.19)]{Abramowitz-Stegun-1964}) and 
\begin{equation}
\begin{split}
 I_{\frac{1+\alpha}{2}}(\xi)\;&\stackrel{\xi\to +\infty}{=}\;\frac{e^\xi}{\sqrt{2 \pi \xi}\,} (1+O(\xi^{-1})) \\
K_{\frac{1+\alpha}{2}}(\xi)\;&\stackrel{\xi\to +\infty}{=}\;{\textstyle\sqrt{\frac{\pi}{2 \xi}}}\, e^{-\xi}(1+O(\xi^{-1}))
\end{split}
\end{equation}
(following from \cite[Eq.~(9.7.1)-(9.7.2)]{Abramowitz-Stegun-1964}).

One therefore deduces that the square-integrable solutions to \eqref{eq:evproblem0} when $x>0$ form a one-dimensional space spanned by
\begin{equation}\label{eq:k0gensol}
	u_E(x) \; := \; \sqrt{x}\, K_{\frac{1+\alpha}{2}}( x\sqrt{E})\,.
\end{equation}
Thus, the general solution to \eqref{eq:evproblem0} belonging to $\mathfrak{h}\cong L^2(\mathbb{R}^-)\oplus  L^2(\mathbb{R}^+)$ has the form
\begin{equation}\label{eq:GenericGE}
	g_E(x) \; \equiv \; \begin{pmatrix}
		B u_E(-x) \\
		C u_E(x) 
	\end{pmatrix},\qquad B,C \in \mathbb{C}\,.
\end{equation}

Now, \eqref{eq:k0asymptx0} and \eqref{eq:k0gensol} yield
\begin{equation}
 \begin{split}
  & u_E(|x|) \;\stackrel{x\to 0}{=}\; 2^{-\frac{1-\alpha}{2}} \Gamma({\textstyle{\frac{1+\alpha}{2}}}) E^{-\frac{1+\alpha}{4}} |x|^{-\frac{\alpha}{2}} \\
  &\qquad\qquad\quad + 2^{-\frac{3+\alpha}{2}} \Gamma(\textstyle{-\frac{1+\alpha}{2}}) E^{\frac{1+\alpha}{4}} |x|^{1+\frac{\alpha}{2}}+ O(|x|^{2-\frac{\alpha}{2}}) \, .
 \end{split}
\end{equation}
By means of such asymptotics, the boundary values \eqref{eq:bilimitsg0g1} for the function \eqref{eq:GenericGE} are computed as
\begin{equation}\label{eq:gE_boudaryvalues}
	\begin{split}
		g_{E,0}^- \;&= \; B \, 2^{-\frac{1-\alpha}{2}} \Gamma({\textstyle{\frac{1+\alpha}{2}}}) E^{-\frac{1+\alpha}{4}} \\
		g_{E,0}^+ \;& = \; C \, 2^{-\frac{1-\alpha}{2}} \Gamma({\textstyle{\frac{1+\alpha}{2}}}) E^{-\frac{1+\alpha}{4}} \\
		g_{E,1}^- \;& = \; B \, 2^{-\frac{3+\alpha}{2}} \Gamma(\textstyle{-\frac{1+\alpha}{2}}) E^{\frac{1+\alpha}{4}} \\
		g_{E,1}^+ \;& = \; C \, 2^{-\frac{3+\alpha}{2}} \Gamma(\textstyle{-\frac{1+\alpha}{2}}) E^{\frac{1+\alpha}{4}}\,.
	\end{split}
\end{equation}

\begin{proof}[Proof of Proposition \ref{prop:SpectrumK0}]~

Part (i) follows from the fact that $A_\alpha(0)$ has finite deficiency index (equal to 2) and therefore all its self-adjoint extension have the same essential spectrum, say, of the Friedrichs extension. (Recall indeed this fact: self-adjoint extensions of the same densely defined and symmetric operator, when the deficiency index is finite, have the respective resolvents that differ by a finite rank operator, owing to the Kre\u{\i}n resolvent formula \cite[Corollary 6]{GMO-KVB2017}; in turn, two self-adjoint operators with compact resolvent difference have the same essential spectrum \cite[Theorem 8.12]{schmu_unbdd_sa}.)  For the latter, indeed $\sigma_{\mathrm{ess}}(A_{\alpha,\mathrm{F}}(0))=[0,+\infty)$, as is easily seen from the existence of a singular sequence relative to any spectral point $\lambda\geqslant 0$ (the Schr\"{o}dinger operator in \eqref{eq:evproblem0} behaves like the free operator at large distances). Moreover, \eqref{eq:evproblem0} being the eigenvalue problem for a Schr\"{o}dinger operator with potential $V(x)=c|x|^{-2}$, $c>0$, it has no $L^2$-solution if $E\geqslant 0$, which proves the absence of embedded eigenvalues.

Concerning part (ii), $A_{\alpha,\mathrm{F}}(0)$ has the same lower bound zero as the original $A_\alpha(0)$ (the Friedrichs extension preserves the lower bound), and therefore does not have negative spectrum. As for the other extensions, the largest multiplicity of their negative spectrum is computed explicitly in the proof of the next claims of the Proposition. In fact, it can be quantified a priori by observing that the $A_{\alpha,\mathrm{R}}^{[\gamma]}(0)$'s, $A_{\alpha,\mathrm{L}}^{[\gamma]}(0)$'s, and $A_{\alpha,a}^{[\gamma]}(0)$'s form one-parameter sub-families of extensions, meaning that their Birman parameter acts on a one-dimensional space, whereas the $A_{\alpha}^{[\Gamma]}(0)$'s form a four-parameter sub-family, meaning that their Birman parameter acts on a two-dimensional space. As the (finite) multiplicity of the negative spectrum of an extension and of its Birman parameter are the same  \cite[Corollary 4]{GMO-KVB2017}, the conclusion then follows.

For the remaining parts (iii)-(v), let us set for convenience
\begin{equation*}
\mu_0 \;:=\; - 2^{1+\alpha}\frac{\Gamma(\frac{1+\alpha}{2})}{\Gamma(-\frac{1+\alpha}{2})}\;>\;0
\end{equation*}
(observe, indeed, that $\Gamma(-\frac{1+\alpha}{2}) < 0$ and $\Gamma(\frac{1+\alpha}{2})>0$, as $\alpha \in [0,1)$).

The general square-integrable eigenfunction \eqref{eq:GenericGE} belongs to the domain of $A_{\alpha,\mathrm{R}}^{[\gamma]}(0)$ if and only if its boundary values \eqref{eq:gE_boudaryvalues} satisfy the conditions \eqref{eq:extAak-IR}, that in this case read $B=0$, as expected, and
\[
	E^{\frac{1+\alpha}{2}} \;=\; - \mu_0 \, \gamma
\] 
irrespective of $C\in\mathbb{C}$. $E$ being strictly positive, this only makes sense when $\gamma<0$, in which case the expression for the only negative eigenvalue $E_0\big(A_{\alpha,\mathrm{R}}^{[\gamma]}(0)\big)=-E=-(-\mu_0\gamma)^{\frac{2}{1+\alpha}}$ of $A_{\alpha,\mathrm{R}}^{[\gamma]}(0)$ takes the form \eqref{eq:EigenvalueIR}. Owing to \eqref{eq:GenericGE}, the corresponding (non-normalised) eigenfunction has the form \eqref{eq:eigenf-IRIL}. The reasoning for $A_{\alpha,\mathrm{L}}^{[\gamma]}(0)$ is completely analogous. Part (iii) is thus proved.

Along the same line, the eigenfunction \eqref{eq:GenericGE} belongs to $\mathcal{D}\big(A_{\alpha,a}^{[\gamma]}(0)\big)$ if and only if its boundary values \eqref{eq:gE_boudaryvalues} satisfy the conditions \eqref{eq:extAak-IIa}, that in this case read
\[
	\begin{split}
		C \; &= \; a B\,, \\
		E^{\frac{1+\alpha}{2}} \;&= \; -\frac{\mu_0}{1+|a|^2} \gamma\,.
	\end{split}
\]
As $E>0$, one must have $\gamma<0$, in which case the expression for the only negative eigenvalue $E_0\big(A_{\alpha,a}^{[\gamma]}(0)\big)=-E$ of $A_{\alpha,a}^{[\gamma]}(0)$ takes the form \eqref{eq:EigenvalueIIa}. Owing to \eqref{eq:GenericGE}, the corresponding (non-normalised) eigenfunction has the form \eqref{eq:eigenf-IIa}. This proves part (iv).

Last, the eigenfunction \eqref{eq:GenericGE} belongs to $\mathcal{D}\big(A_{\alpha}^{[\Gamma]}(0)\big)$ if and only if \eqref{eq:gE_boudaryvalues} matches \eqref{eq:extAak-III}, i.e.,
\[
	\begin{pmatrix}
		g_{E,1}^- \\
		g_{E,1}^+ 
	\end{pmatrix}
	\; = \;
	\begin{pmatrix}
		\gamma_1 & \gamma_2 + \ii \gamma_3 \\
		\gamma_2 - \ii \gamma_3 & \gamma_4
	\end{pmatrix}
	\begin{pmatrix}
		g_{E,0}^- \\
		g_{E,0}^+
	\end{pmatrix}.
\]
The latter is equivalent to
\[
	2^{1+\alpha}\,\frac{\Gamma(\frac{1+\alpha}{2})}{\,\Gamma(-\frac{1+\alpha}{2})} \begin{pmatrix}
		\gamma_1 & \gamma_2+ \ii \gamma_3 \\
		\gamma_2 - \ii \gamma_3 & \gamma_4
	\end{pmatrix} \begin{pmatrix}
		B \\ C
	\end{pmatrix}\;=\;E^{\frac{1+\alpha}{2}}\begin{pmatrix}
		B \\
		C 
	\end{pmatrix} ,
\]
meaning that $E^{\frac{1+\alpha}{2}}$ is an eigenvalue of
\[
	T \; := \; -\mu_0 \begin{pmatrix}
		\gamma_1 & \gamma_2+ \ii \gamma_3 \\
		\gamma_2 - \ii \gamma_3 & \gamma_4
	\end{pmatrix}  \, .
\]
Thus, $-E$ (with $E>0$) is a negative eigenvalue of $A_\alpha^{[\Gamma]}(0)$ if and only of $E^{\frac{1+\alpha}{2}}$ is a positive eigenvalue of the matrix $T$, and equating the the eigenvector of $T$ to $\begin{pmatrix} B \\ C \end{pmatrix}$ yields the condition on the constants $B$ and $C$ in \eqref{eq:GenericGE} for the corresponding eigenfunction of $A_\alpha^{[\Gamma]}(0)$.

Elementary arguments show that the eigenvalues of $T$ are two and real (due to hermiticity) and 
\begin{itemize}
 \item \emph{both positive} if and only if $\mathrm{Tr}(T)>0$ and $\det(T)>0$,
 \item \emph{only one positive} if and only if one of the following two possibilities occurs: $\det(T)< 0$ (corresponding to two distinct eigenvalues of opposite sign), or $\det(T)=0$ and $\mathrm{Tr}(T)>0$ (corresponding to a positive and a zero eigenvalue).
\end{itemize}
Computing
\[
 \mathrm{Tr}(T)\;=\;-\mu_0(\gamma_1+\gamma_4)\,,\qquad\det(T)\;=\;-\mu_0^2(\gamma_2^2+\gamma_3^2-\gamma_1\gamma_4)
\]
one then concludes that the conditions for $T$ to have, respectively, exactly one, or two positive eigenvalues, and hence for $A_{\alpha}^{[\Gamma]}(0)$ to have, respectively, exactly one, or two negative eigenvalues, take respectively the form \eqref{eq:III0-only1neg} and \eqref{eq:III0-2neg}.

Explicitly, the eigenvalues of $T$ are 
\[
	\lambda^\pm\;:=\;-\frac{\,\mu_0}{2}  \Big( \gamma_1+\gamma_4 \pm \sqrt{(\gamma_1-\gamma_4)^2 +4 (\gamma_2^2+\gamma_3^2)} \,\Big),
\]
with $\lambda^-\geqslant\lambda^+$, and the quantities
\[
 \begin{split}
  E_0\big(A_{\alpha}^{[\Gamma]}(0)\big)\;&:=\;-E^-=-(\lambda^-)^{\frac{2}{1+\alpha}} \\
  E_1\big(A_{\alpha}^{[\Gamma]}(1)\big)\;&:=\;-E^+=-(\lambda^+)^{\frac{2}{1+\alpha}}
 \end{split}
\]
(hence, the expressions \eqref{eq:III0-evs}), defined when applicable depending on the conditions \eqref{eq:III0-2neg}-\eqref{eq:III0-only1neg}, are the corresponding negative eigenvalues of $A_{\alpha}^{[\Gamma]}(0)$ (only the first, or both, when applicable). The condition $\lambda^-\leqslant 0$, hence \eqref{eq:III0-none-neg}, clearly identifies the case when $A_{\alpha}^{[\Gamma]}(0)$ does not have negative eigenvalues at all.

The eigenvalues of $T$ are distinct ($\lambda^->\lambda^+$) when $\gamma_1\neq\gamma_4$ or when at least one among $\gamma_2,\gamma_3$ is non-zero: with distinct eigenvalues, the largest has eigenvector (proportional to)
\[
 \begin{pmatrix}
  B \\ C
 \end{pmatrix} \;\equiv\;
 \begin{pmatrix}
  \gamma_1-\gamma_4 - \sqrt{(\gamma_1-\gamma_4)^2 +4 (\gamma_2^2+\gamma_3^2)} \\
  2(\gamma_2-\ii\gamma_3)
 \end{pmatrix},
\]
implying that the ground state eigenfunction \eqref{eq:GenericGE} of $A_{\alpha}^{[\Gamma]}(0)$ relative to the non-degenerate lowest negative eigenvalue $ E_0\big(A_{\alpha}^{[\Gamma]}(0)\big)$ has the form \eqref{eq:eigenf-III-nondegen}. Instead, $T$ has coincident eigenvalues ($\lambda^-=\lambda^+$) when $\gamma_1=\gamma_4$ and $\gamma_2=\gamma_3=0$, i.e., when $T$ is a multiple of the identity. In this case any vector $\begin{pmatrix}
  B \\ C
 \end{pmatrix}\in\mathbb{C}^2$ is eigenvector for $T$, implying that the $A_{\alpha}^{[\Gamma]}(0)$ has a two-dimensional eigenspace relative to the two-fold degenerate negative eigenvalue  $ E_0\big(A_{\alpha}^{[\Gamma]}(0)\big)=-\mu_0\gamma_1$, spanned by the eigenfunctions \eqref{eq:eigenf-III-degen}.

The proof of part (v) is thus completed.
\end{proof}

\subsection{Spectral analysis in the fibre $k\in\mathbb{Z}\setminus\{0\}$}~ \label{subsec:SpectrumK}

%
%

Next, we shall establish the following counterpart to Proposition \ref{prop:SpectrumK0} for the generic, non-zero mode $k$.

\begin{proposition}\label{prop:NegativeEigenvaluesKDiffZero} Let $\alpha\in[0,1)$, $k\in\mathbb{Z}\setminus\{0\}$, $a\in\mathbb{C}$.
\begin{itemize}
 \item[(i)] The spectrum of any self-adjoint extension \eqref{eq:families-ext-Aak} of $A_\alpha(k)$ is purely discrete. $A_{\alpha,\mathrm{F}}(k)$ has no negative eigenvalue, all other extensions have at most finitely many.
 \item[(ii)] $A_{\alpha,\mathrm{R}}^{[\gamma]}(k)$ and $A_{\alpha,\mathrm{L}}^{[\gamma]}(k)$ have at most one negative eigenvalue, denoted, respectively, by $E_0\big(A_{\alpha,\mathrm{R}}^{[\gamma]}(k)\big)$ and $E_0\big(A_{\alpha,\mathrm{L}}^{[\gamma]}(k)\big)$. Such negative eigenvalue exists if and only if $\gamma<-|k|/(1+\alpha)$, in which case it is non-degenerate with
 	\begin{equation}\label{eq:negEV-k_extIRIL}
	\left.
	\begin{array}{l}
	 E_0\big(A_{\alpha,\mathrm{R}}^{[\gamma]}(k)\big) \\
	 E_0\big(A_{\alpha,\mathrm{L}}^{[\gamma]}(k)\big)
	\end{array}\!\!
	\right\}\;\leqslant\; \frac{\,2^{\frac{1-\alpha}{1+\alpha}}(1+\alpha)^{\frac{2\alpha}{1+\alpha}}|k|^{\frac{2}{1+\alpha}}}{\Gamma(\frac{1-\alpha}{1+\alpha})} \big({\textstyle\frac{1+\alpha}{|k|}}\gamma+1 \big)\qquad(\gamma<-{\textstyle\frac{|k|}{1+\alpha}})\,.
	\end{equation}
 \item[(iii)] $A_{\alpha,a}^{[\gamma]}(k)$ has at most one negative eigenvalue, denoted by $E_0\big(A_{\alpha,a}^{[\gamma]}(k)\big)$. Such negative eigenvalue exists if and only if $\gamma<-|k|(1+|a|^2)/(1+\alpha)$, in which case
 \begin{equation}\label{eq:negEV-k_extIIa}
 \begin{split}
  E_0\big(A_{\alpha,a}^{[\gamma]}(k)\big)\;&\leqslant\;\frac{\,2^{\frac{1-\alpha}{1+\alpha}}(1+\alpha)^{\frac{2\alpha}{1+\alpha}}|k|^{\frac{2}{1+\alpha}}}{\Gamma(\frac{1-\alpha}{1+\alpha})}\big({\textstyle\frac{1+\alpha}{(1+|a|^2)|k|}}\gamma+1 \big) \\
  &\qquad\! (\gamma<{-\textstyle\frac{(1+|a|^2) |k|}{1+\alpha}})\,.
 \end{split}
 \end{equation}
 \item[(iv)] $A_{\alpha}^{[\Gamma]}(k)$ has at most two negative eigenvalues: exactly two if and only if
 \begin{equation}\label{eq:negEV-k_extIII_2ev}
 |k|\;<\;-\textstyle(1+\alpha)\big(\gamma_1+\gamma_4+\sqrt{(\gamma_1-\gamma_4)^2+4 (\gamma_2^2 + \gamma_3^2)}\,\big)\,,
 \end{equation}
 only one if and only if
 \begin{equation}\label{eq:negEV-k_extIII_1ev}
 \begin{split}
  & -\textstyle(1+\alpha)\big(\gamma_1+\gamma_4+\sqrt{(\gamma_1-\gamma_4)^2+4 (\gamma_2^2 + \gamma_3^2)}\,\big)\;\leqslant \\
  & \qquad\qquad \leqslant\; |k|\;<\;-\textstyle(1+\alpha)\big(\gamma_1+\gamma_4-\sqrt{(\gamma_1-\gamma_4)^2+4 (\gamma_2^2 + \gamma_3^2)}\,\big)\,,
 \end{split}
 \end{equation}
 and none if and only if
 \begin{equation}\label{eq:negEV-k_extIII_0ev}
  |k|\;\geqslant\;-\textstyle(1+\alpha)\big(\gamma_1+\gamma_4-\sqrt{(\gamma_1-\gamma_4)^2+4 (\gamma_2^2 + \gamma_3^2)}\,\big)\,.
 \end{equation}
%
%
In the first two cases the lowest negative eigenvalue $ E_0\big(A_{\alpha}^{[\Gamma]}(k)\big)$ satisfies
\begin{equation}\label{eq:negEV-k_extIIIbound}
\begin{split}
 E_0\big(A_{\alpha}^{[\Gamma]}(k)\big)\;&\leqslant\; \frac{\,2^{\frac{1-\alpha}{1+\alpha}}(1+\alpha)^{\frac{2\alpha}{1+\alpha}}|k|^{\frac{2}{1+\alpha}}}{\Gamma(\frac{1-\alpha}{1+\alpha})}\,\times \\
 &\quad \times\left(1+\frac{1+\alpha}{|k|}\Big(\gamma_1+\gamma_4-\sqrt{(\gamma_1-\gamma_4)^2+4 (\gamma_2 + \gamma_3^2)}\,\Big)\right).
\end{split}
\end{equation}
\end{itemize}
\end{proposition}

We shall also need the following additional information on the fibre operators.

\begin{lemma}\label{lem:OrderRelationOps}
	Let $\widetilde{\mathscr{H}}_\alpha=\bigoplus_{k\in\mathbb{Z}}\widetilde{A}_\alpha(k)$ be any one of the operators \eqref{eq:HalphaFriedrichs_unif-fibred}-\eqref{eq:Halpha-III_unif-fibred}. If $|k| > |k'|$, then $\mathcal{D}(\widetilde{A}_\alpha(k))\subset \mathcal{D}(\widetilde{A}_\alpha(k'))$ and $\langle h,\widetilde{A}_\alpha(k)h\rangle_{L^2(\mathbb{R})}>\langle h,\widetilde{A}_\alpha(k')h\rangle_{L^2(\mathbb{R})}$ $\forall h\in\mathcal{D}(\widetilde{A}_\alpha(k))\setminus \{0\}$.
\end{lemma}

\begin{proof}[Proof of Proposition \ref{prop:NegativeEigenvaluesKDiffZero}] (i) Any self-adjoint extension of $A_\alpha(k)$ is a suitable restriction of the adjoint \eqref{eq:Afstar}, and therefore acts as the differential operator
\[
 -\frac{\ud^2}{\ud x^2}+k^2 |x|^{2\alpha}+\frac{\,\alpha(2+\alpha)\,}{4x^2}
\]
on $L^2(\mathbb{R}^\pm)$. On each half-line this is a Schr\"{o}dinger operator of the form $-\frac{\ud^2}{\ud x^2}+V(x)$ with $V(x)\to+\infty$ as $|x|\to+\infty$ and therefore (see, e.g., \cite[Sect.~5.5]{Titchmarsh-EigenfExpans-1962}) has purely discrete spectrum. In turn, the spectrum of each extension on the direct sum $L^2(\mathbb{R}^-)\oplus L^2(\mathbb{R}^+)$ is the union of the spectra relative to each half-line, hence is itself purely discrete. In particular, $A_{\alpha,\mathrm{F}}(k)$ is strictly positive, as the Friedrichs extension preserves the lower bound \eqref{eq:lowerboundAak}: thus,  $A_{\alpha,\mathrm{F}}(k)$ has no negative spectrum. All other extensions may produce negative bound states: their number is finite because the original operator $A_\alpha(k)$ has finite deficiency index (see, e.g., \cite[Corollary 5]{GMO-KVB2017}). More precisely (see, e.g., \cite[Corollary 4]{GMO-KVB2017}), the number of negative eigenvalue of each extension is the same as the number of negative bound states of the corresponding Birman extension parameter -- the operators listed in \eqref{eq:GammaIRIL}-\eqref{eq:IIkTheorem51} -- and therefore amounts up to one for $A_{\alpha,\mathrm{R}}^{[\gamma]}(k)$, $A_{\alpha,\mathrm{L}}^{[\gamma]}(k)$, and $A_{\alpha,a}^{[\gamma]}(k)$, and up to two for $A_{\alpha}^{[\Gamma]}(k)$.

(ii) $A_{\alpha,\mathrm{R}}^{[\gamma]}(k)$ and $A_{\alpha,\mathrm{L}}^{[\gamma]}(k)$ are associated to the one-dimensional Birman parameter \eqref{eq:GammaIRIL}, which admits negative spectrum, in the precise number of one negative eigenvalue, if and only if $\beta_k<0$. Such condition is equivalent to $\gamma<-|k|/(1+\alpha)$. A further general feature of the extension theory (see, e.g., \cite[Theorem 11]{GMO-KVB2017}) is that the lowest negative eigenvalue of the Birman parameter is an upper bound of the lowest negative eigenvalue of the corresponding extension. The conditions $E_0\big(A_{\alpha,\mathrm{R}}^{[\gamma]}(k)\big)\leqslant\beta_k<0$ and $E_0\big(A_{\alpha,\mathrm{L}}^{[\gamma]}(k)\big)\leqslant\beta_k<0$ thus yield \eqref{eq:negEV-k_extIRIL}.

(iii) The reasoning is precisely the same as for part (ii), now with respect to the Birman parameter \eqref{eq:GammaTauIIa}. The condition $E_0\big(A_{\alpha,a}^{[\gamma]}(k)\big)\leqslant\tau_k<0$ yields \eqref{eq:negEV-k_extIIa}.

(iv) Now the Birman parameter is the two-dimensional hermitian matrix $T_k$ given by \eqref{eq:IIkTheorem51}. We can drop out the positive multiplicative pre-factor
\[
 \mu_k\;:=\;\frac{\,2^{\frac{1-\alpha}{1+\alpha}}(1+\alpha)^{\frac{2\alpha}{1+\alpha}}|k|^{\frac{2}{1+\alpha}}}{\Gamma(\frac{1-\alpha}{1+\alpha})}\;>\;0\,,
\]
and the number of negative eigenvalues for $T_k$ is equivalent to the number of negative eigenvalues for 
	\[
			\widetilde{T}_k\;:=\; \begin{pmatrix}
				1+ \frac{1+\alpha}{|k|} \gamma_1 & \frac{1+\alpha}{|k|}(\gamma_2 + \ii \gamma_3) \\
				\frac{1+\alpha}{|k|}(\gamma_2 - \ii \gamma_3) & 1+ \frac{1+\alpha}{|k|} \gamma_4
			\end{pmatrix}.
	\]
The latter has indeed two real eigenvalues (due to hermiticity), explicitly,
\[
  \lambda_k^\pm\;:=\;1+\frac{1+\alpha}{|k|}\Big(\gamma_1+\gamma_4\pm\sqrt{(\gamma_1-\gamma_4)^2+4 (\gamma_2^2 + \gamma_3^2)}\,\Big).
\]

Three possibilities can occur: $\lambda_k^-\leqslant\lambda_k^+<0$, $\lambda_k^-<0\leqslant\lambda_k^+$, or $0\leqslant\lambda_k^-\leqslant\lambda_k^+$. The first corresponds to the condition
\[
 |k|\;<\;-\textstyle(1+\alpha)\big(\gamma_1+\gamma_4+\sqrt{(\gamma_1-\gamma_4)^2+4 (\gamma_2^2 + \gamma_3^2)}\,\big)\,,
\]
in which case $T_k$ has two negative eigenvalues (namely $\mu_k\lambda^\pm_k$), and so too does $A_{\alpha}^{[\Gamma]}(k)$ with the upper bound $\mu_k\lambda^-_k\geqslant E_0\big(A_{\alpha}^{[\Gamma]}(k)\big)$.

The second possibility corresponds to the condition
\[
 \begin{split}
  & -\textstyle(1+\alpha)\big(\gamma_1+\gamma_4+\sqrt{(\gamma_1-\gamma_4)^2+4 (\gamma_2^2 + \gamma_3^2)}\,\big)\;\leqslant \\
  & \qquad\qquad \leqslant\; |k|\;<\;-\textstyle(1+\alpha)\big(\gamma_1+\gamma_4-\sqrt{(\gamma_1-\gamma_4)^2+4 (\gamma_2^2 + \gamma_3^2)}\,\big)\,,
 \end{split}
\]
in which case $T_k$ has only one negative eigenvalue (namely $\mu_k\lambda^-_k$), and so too does $A_{\alpha}^{[\Gamma]}(k)$, again with $\mu_k\lambda^-_k\geqslant E_0\big(A_{\alpha}^{[\Gamma]}(k)\big)$.

The third possibility is
\[
 |k|\;\geqslant\;-\textstyle(1+\alpha)\big(\gamma_1+\gamma_4-\sqrt{(\gamma_1-\gamma_4)^2+4 (\gamma_2^2 + \gamma_3^2)}\,\big)\,,
\]
for an infinite number of $k$'s it is never empty, and corresponds to the fact that $T_k$ has no negative eigenvalues, hence $A_{\alpha}^{[\Gamma]}(k)$ has neither.

Part (iv) is thus proved.
\end{proof}

\begin{proof}[Proof of Lemma \ref{lem:OrderRelationOps}]
 For each non-zero $h$ the inequality among expectations is an obvious computation:
 \[
  \langle h,\widetilde{A}_\alpha(k)h\rangle_{L^2(\mathbb{R})}-\langle h,\widetilde{A}_\alpha(k')h\rangle_{L^2(\mathbb{R})}\;=\;\langle h, |x|^{2\alpha}(|k|^2-|k'|^2) h \rangle_{L^2(\mathbb{R})}\;>\;0\,.
 \]
 So one has to prove the inclusion of domains. The actual identity $\mathcal{D}(\widetilde{A}_\alpha(k))= \mathcal{D}(\widetilde{A}_\alpha(k'))$ for $k,k'\in\mathbb{Z}\setminus\{0\}$ is a consequence of the explicit characterisation of the spaces $\mathcal{D}(\widetilde{A}_\alpha(k))$ and $\mathcal{D}(\overline{A_\alpha(k)})$ we made in \cite[Theorem 5.5 and Proposition 3.8]{GMP-Grushin2-2020}, which is independent of $k$. When instead $|k|>0$ the inclusion $\mathcal{D}(\widetilde{A}_\alpha(k))\subset \mathcal{D}(\widetilde{A}_\alpha(0))$ follows again from \cite[Theorem 5.5]{GMP-Grushin2-2020} that provides a representation of $\mathcal{D}(\widetilde{A}_\alpha(0))$ and $\mathcal{D}(\widetilde{A}_\alpha(k))$ in terms of $\mathcal{D}(\overline{A_\alpha(0)})$ and $\mathcal{D}(\overline{A_\alpha(k)})$, and from the inclusion $\mathcal{D}(\overline{A_\alpha(0)})\subset\mathcal{D}(\overline{A_\alpha(k)})$, the latter space being characterised in \cite[Proposition 3.8]{GMP-Grushin2-2020}, the former in \cite[Prop.~4.11(i)]{Bruneau-Derezinski-Georgescu-2011}. (In the notation of \cite[Prop.~4.11(i)]{Bruneau-Derezinski-Georgescu-2011}, $\overline{A_\alpha(0)}$ is the operator $L_m^\mathrm{min}$ with $m=\frac{1+\alpha}{2}$: the validity condition $m\in(0,1)$ required in \cite[Prop.~4.11(i)]{Bruneau-Derezinski-Georgescu-2011} is therefore satisfied.) 
\end{proof}

\section{Scattering on fibre}\label{sec:scattering-fibre}

We continue the analysis on fibre by discussing the scattering. Owing to Propositions \ref{prop:SpectrumK0} and \ref{prop:NegativeEigenvaluesKDiffZero}, this is a meaningful question only for the mode $k=0$, where indeed the self-adjoint realisations of the fibre operator have all essential spectrum $[0,+\infty)$ with no embedded eigenvalues -- in particular, absolutely continuous.

As all the zero-mode self-adjoint operators \eqref{eq:families-ext-Aak} act with the differential action \eqref{eq:Afstar}, we are concerned in each case with the one-dimensional scattering on the potential
\begin{equation}
 V(x)\;:=\;\frac{\alpha(2+\alpha)}{4|x|^2}
\end{equation}
and with the additional ``internal'' interaction governed by one of the boundary conditions \eqref{eq:extAak-F}-\eqref{eq:extAak-III} at $x=0$ -- understanding the cases $\mathrm{I}_\mathrm{R}$ and $\mathrm{I}_\mathrm{L}$ as interaction with the potential $V$ and with the barrier at $x=0$ only on the corresponding half-line.

In all cases the quest for scattering states leads to determining the generalised eigenfunctions relative to an energy $E>0$ as suitable solutions $g\equiv\begin{pmatrix} g^- \\ g^+\end{pmatrix}$ to the ordinary differential equation
\begin{equation}\label{eq:ODEpm}
 -g''+V g\;=\;E g
\end{equation}
on each half-line when applicable.

In complete analogy to the reasoning for \eqref{eq:evproblem0}-\eqref{eq:GenericGE} we solve the problem
\begin{equation}\label{eq:ODEpos}
 u''(x)+\big(E-{\textstyle\frac{\alpha(2+\alpha)}{4|x|^2}}\big)u(x)\;=\;0\,,\qquad x>0\,,
\end{equation}
in terms of modified Bessel function, switching now for convenience to the Hankel functions (Bessel functions of third kind) $H_\nu^{(1)}$ and $H_\nu^{(2)}$, where
\begin{equation}
 \begin{split}
  H_\nu^{(1)}(z)\;&=\;\frac{\ii}{\sin\nu\pi}\big(e^{-\ii\pi\nu}J_\nu(z)-J_{-\nu}(z)\big)\,, \\
  H_\nu^{(2)}(z)\;&=\;-\frac{\ii}{\sin\nu\pi}\big(e^{\ii\pi\nu}J_\nu(z)-J_{-\nu}(z)\big)\,,
 \end{split}
\end{equation}
and $J_\nu$ is the ordinary Bessel function of first kind \cite[Eq.~(9.1.3)-(9.1.4)]{Abramowitz-Stegun-1964}. We thus write the two linearly independent solutions  $u_E^{(1)}$ and $u_E^{(2)}$ to \eqref{eq:ODEpos} as
\begin{equation}\label{eq:v1v2soll}
 \begin{split}
  u_E^{(1)}(x)\;:=\;\sqrt{x}\, H_{\frac{1+\alpha}{2}}^{(1)}(x\sqrt{E})\,, \\
  u_E^{(2)}(x)\;:=\;\sqrt{x}\, H_{\frac{1+\alpha}{2}}^{(2)}(x\sqrt{E})\,. 
 \end{split}
\end{equation}
In turn, the general solution to \eqref{eq:ODEpm} on $\mathbb{R}$ has the form
\begin{equation}\label{eq:generalgEscatt}
 \;g_E(x)\;\equiv\;\begin{pmatrix} g_E^-(x) \\ g_E^+(x) \end{pmatrix}\;=\;
 \begin{pmatrix}
  A_1^- u_E^{(1)}(-x)+A_2^-u_E^{(2)}(-x) \\
  A_1^+ u_E^{(1)}(x)+A_2^+u_E^{(2)}(x)
 \end{pmatrix},\qquad A_1^\pm,A_2^\pm\in\mathbb{C}\,.
\end{equation}

From \eqref{eq:v1v2soll} and from the asymptotics at short \cite[Eq.~(9.1.10)]{Abramowitz-Stegun-1964} and large \cite[Eq.~(9.2.3)-(9.2.4)]{Abramowitz-Stegun-1964} distances for the Hankel functions we compute
\begin{equation}\label{eq:Hankel-small}
 \begin{split}
 u_E^{(1)}(x)\;&\stackrel{x\downarrow 0}{=}\;-\frac{\ii\,2^{\frac{1+\alpha}{2}}}{\,E^{\frac{1+\alpha}{4}}\Gamma(\frac{1-\alpha}{2})\sin(\frac{1+\alpha}{2}\pi)}\,x^{-\frac{\alpha}{2}} \\
  &\qquad\qquad+\frac{\ii\,E^{\frac{1+\alpha}{4}}e^{-\ii\pi\frac{1+\alpha}{2}}}{2^{\frac{1+\alpha}{2}}\Gamma(\frac{3+\alpha}{2})\sin(\frac{1+\alpha}{2}\pi)}\,x^{1+\frac{\alpha}{2}}+O(x^{2-\frac{\alpha}{2}})\,, \\
  u_E^{(2)}(x)\;&\stackrel{x\downarrow 0}{=}\;\frac{\ii\,2^{\frac{1+\alpha}{2}}}{\,E^{\frac{1+\alpha}{4}}\Gamma(\frac{1-\alpha}{2})\sin(\frac{1+\alpha}{2}\pi)}\,x^{-\frac{\alpha}{2}} \\
  &\qquad\qquad-\frac{\ii\,E^{\frac{1+\alpha}{4}}e^{\ii\pi\frac{1+\alpha}{2}}}{2^{\frac{1+\alpha}{2}}\Gamma(\frac{3+\alpha}{2})\sin(\frac{1+\alpha}{2}\pi)}\,x^{1+\frac{\alpha}{2}}+O(x^{2-\frac{\alpha}{2}})\,,
 \end{split}
\end{equation}
and 
\begin{equation}\label{eq:Hankel-large}
 \begin{split}
  u_E^{(1)}(x)\;&\stackrel{x\to+\infty}{=}\;\sqrt{\frac{2}{\pi}}\,E^{-\frac{1}{4}}\,e^{i(x\sqrt{E}-\frac{\pi}{4}(2+\alpha))}\,(1+O(x^{-1}))\,, \\
  u_E^{(2)}(x)\;&\stackrel{x\to+\infty}{=}\;\sqrt{\frac{2}{\pi}}\,E^{-\frac{1}{4}}\,e^{-i(x\sqrt{E}-\frac{\pi}{4}(2+\alpha))}\,(1+O(x^{-1}))\,.
 \end{split}
\end{equation}
Incidentally, \eqref{eq:Hankel-large} explains the convenience of choosing $H_{\frac{1+\alpha}{2}}^{(1)},H_{\frac{1+\alpha}{2}}^{(2)}$ to express a pair of linearly independent solutions to \eqref{eq:ODEpos}: at large distances, $u_E^{(1)}$ and $u_E^{(2)}$ behave as plane waves, with an evident advantage for their interpretation in the scattering arguments that follow.

The local behaviour of the general solution \eqref{eq:generalgEscatt} around $x=0$, and concretely speaking the boundary values \eqref{eq:bilimitsg0g1}, are given by
\begin{equation}\label{eq:boundarygpmscatt}
 \begin{split}
  g_{E,0}^\pm\;&=\;\frac{\ii\,2^{\frac{1+\alpha}{2}}}{\,E^{\frac{1+\alpha}{4}}\Gamma(\frac{1-\alpha}{2})\sin(\frac{1+\alpha}{2}\pi)}\,(A_2^\pm-A_1^\pm)\,,  \\
  g_{E,1}^\pm\;&=\;\frac{E^{\frac{1+\alpha}{4}}e^{-\ii\pi\frac{\alpha}{2}}}{2^{\frac{1+\alpha}{2}}\Gamma(\frac{3+\alpha}{2})\sin(\frac{1+\alpha}{2}\pi)}\,(A_1^\pm + e^{\ii\pi\alpha}A_2^\pm)\,,
 \end{split}
\end{equation}
as is easy to compute from \eqref{eq:Hankel-small} and \eqref{eq:bilimitsg0g1}.

Let us focus, in particular, on the $\mathrm{II}_a$-type scattering.

 \begin{proposition}\label{prop:scatteringIIa} Let $\alpha\in[0,1)$, $a\in\mathbb{C}$, and $\gamma\in\mathbb{R}$.
 \begin{itemize}
  \item[(i)] The one-dimensional Schr\"{o}dinger scattering governed by the Hamiltonian $A_{\alpha,a}^{[\gamma]}(0)$ has transmission and reflection coefficients equal respectively to
   \begin{equation}\label{eq:TR-IIa}
  \begin{split}
   T_{\alpha,a,\gamma}(E)\;&=\;\left|\frac{\,E^{\frac{1+\alpha}{2}}(1+e^{\ii\pi\alpha})\,\Gamma(\frac{1-\alpha}{2})\,\overline{a}}{\,E^{\frac{1+\alpha}{2}}\Gamma(\frac{1-\alpha}{2})(1+|a|^2)+\ii\,\gamma\,2^{1+\alpha}e^{\ii\frac{\pi}{2}\alpha}\Gamma(\frac{3+\alpha}{2})} \right|^2 \,, \\
   R_{\alpha,a,\gamma}(E)\;&=\;\left|\frac{\,E^{\frac{1+\alpha}{2}}\Gamma(\frac{1-\alpha}{2})\,(1-|a|^2\,e^{\ii\pi\alpha})+\ii\,\gamma\,2^{1+\alpha}e^{\ii\frac{\pi}{2}\alpha}\Gamma(\frac{3+\alpha}{2})}{\,E^{\frac{1+\alpha}{2}}\Gamma(\frac{1-\alpha}{2})(1+|a|^2)+\ii\,\gamma\,2^{1+\alpha}e^{\ii\frac{\pi}{2}\alpha}\Gamma(\frac{3+\alpha}{2})}\right|^2\,,
  \end{split}
 \end{equation}
 and
 \begin{equation}\label{eq:TpEi1}
  T_{\alpha,a,\gamma}(E)+R_{\alpha,a,\gamma}(E)\;=\;1\,.
 \end{equation}
 \item[(ii)] Reflection and transmission coefficients are independent of energy $E$ for all extensions with $\gamma=0$.
 \item[(iii)] The scattering is reflection-less ($R_{\alpha,a,\gamma}(E)=0$) when
 \begin{equation}\label{eq:Etransm}
   E\;=\;\bigg(\frac{\,2^{1+\alpha}\,\gamma\,\Gamma(\frac{3+\alpha}{2})\sin\frac{\pi}{2}\alpha\,}{\,\Gamma(\frac{1-\alpha}{2})(1-\cos\pi\alpha)\,}\bigg)^{\!\frac{2}{1+\alpha}},	
 \end{equation}
 provided that $\alpha\in(0,1)$, $|a|=1$, and $\gamma>0$. 
 \item[(iv)] In the high energy limit the scattering is independent of the extension parameter $\gamma$ and one has
 \begin{equation}\label{eq:highenergyscatt}
  \begin{split}
   \lim_{E\to+\infty}T_{\alpha,a,\gamma}(E)\;&=\;\frac{\,2\,|a|^2(1+\cos\pi\alpha)}{(1+|a|^2)^2}\,, \\
   \lim_{E\to+\infty}R_{\alpha,a,\gamma}(E)\;&=\;\frac{\,1+|a|^4-2|a|^2\cos\pi\alpha}{(1+|a|^2)^2}\,.
  \end{split}
 \end{equation}
 \item[(v)] In the low energy limit, for $\gamma\neq 0$,
  \begin{equation}\label{eq:lowenergyscatt}
  \begin{split}
   \lim_{E\downarrow 0}T_{\alpha,a,\gamma}(E)\;&=\;0\,, \\
   \lim_{E\downarrow 0}R_{\alpha,a,\gamma}(E)\;&=\;1\,.
  \end{split}
 \end{equation}
 \end{itemize} 
 \end{proposition}

\begin{proof} The generalised eigenfunctions $g_E$ for $A_{\alpha,a}^{[\gamma]}$ at energy $E>0$ have the form \eqref{eq:generalgEscatt} with parameters $A_1^\pm,A_2^\pm$ such that the corresponding boundary values \eqref{eq:boundarygpmscatt} satisfy
 \[
  \begin{cases}
   \qquad g_{E,0}^+\,=\,a\,g_{E,0}^- \\
   \;g_{E,1}^-+\overline{a}g_{E,1}^+\,=\,\gamma\,g_{E,0}^-
  \end{cases}
 \]
 (see \eqref{eq:extAak-IIa} above), whence
 \[\tag{a}\label{eq:taaac1}
  \begin{cases}
   \qquad A_2^+-A_1^+\,=\,a(A_2^--A_1^-) \\
   \;\displaystyle\frac{\,E^{\frac{1+\alpha}{2}}e^{-\ii\pi\frac{\alpha}{2}}\Gamma(\frac{1-\alpha}{2})}{2^{1+\alpha}\Gamma(\frac{3+\alpha}{2})}\,(A_1^-+\overline{a}\,A_1^+ + e^{\ii\pi\alpha}(A_2^-+\overline{a}\,A_2^+)\,= \\
   \;\qquad\qquad\qquad \;=\,\ii\,\gamma(A_2^--A_1^-)\,.
  \end{cases}
 \]

 By standard scattering arguments, and in view of the large distance asymptotics \eqref{eq:Hankel-large}, the occurrence of an incident \emph{plane wave} $e^{-\ii x\sqrt{E}}$ from $+\infty$ towards the origin (with ``momentum'' $\sqrt{E}$), producing after the scattering a transmitted plane wave $e^{-\ii x\sqrt{E}}$ towards $-\infty$ and a reflected plane wave $e^{\ii x\sqrt{E}}$ back towards $+\infty$, corresponds to the choice
 \[\tag{b}\label{eq:taaac2}
  A_2^+\,=\,1\,,\qquad A_2^-\,=\,0\,.
 \]
 Indeed, in this case \eqref{eq:generalgEscatt} takes the asymptotic form
 \[
  g_E(x)\;\sim\;
   {\textstyle\sqrt{\frac{2}{\pi}}}\,E^{-\frac{1}{4}}e^{\ii\frac{\pi}{4}(2+\alpha)}\begin{pmatrix}
   A_1^-\,e^{-\ii\frac{\pi}{2}(2+\alpha)}\,e^{-\ii x\sqrt{E}} \\
   A_1^+\,e^{-\ii\frac{\pi}{2}(2+\alpha)}\,e^{\ii x\sqrt{E}}+e^{-\ii x\sqrt{E}}
  \end{pmatrix}\qquad\textrm{as }|x|\to +\infty\,,
 \]
 with conventionally unit amplitude for the source incident plane wave from $+\infty$ (up to the overall multiplicative pre-factor $\sqrt{\frac{2}{\pi}}\,E^{-\frac{1}{4}}e^{\ii\frac{\pi}{4}(2+\alpha)}$), and clearly no incident plane wave from $-\infty$. The transmission and reflection coefficients for the scattering are then equal respectively to
 \[
   T^{(\leftarrow)}_{\alpha,a,\gamma}(E)\;=\;|A_1^-|^2\,,\qquad R^{(\leftarrow)}_{\alpha,a,\gamma}(E)\;=\;|A_1^+|^2\,.
 \]

 Solving \eqref{eq:taaac1} with the condition \eqref{eq:taaac2} yields
 \[
  \begin{split}
   A_1^-\;&=\;-\frac{\,E^{\frac{1+\alpha}{2}}(1+e^{\ii\pi\alpha})\,\Gamma(\frac{1-\alpha}{2})\,\overline{a}}{\,E^{\frac{1+\alpha}{2}}\Gamma(\frac{1-\alpha}{2})(1+|a|^2)+\ii\,\gamma\,2^{1+\alpha}e^{\ii\frac{\pi}{2}\alpha}\Gamma(\frac{3+\alpha}{2})} \\
   A_1^+\;&=\;\frac{\,E^{\frac{1+\alpha}{2}}\Gamma(\frac{1-\alpha}{2})\,(1-|a|^2\,e^{\ii\pi\alpha})+\ii\,\gamma\,2^{1+\alpha}e^{\ii\frac{\pi}{2}\alpha}\Gamma(\frac{3+\alpha}{2})}{\,E^{\frac{1+\alpha}{2}}\Gamma(\frac{1-\alpha}{2})(1+|a|^2)+\ii\,\gamma\,2^{1+\alpha}e^{\ii\frac{\pi}{2}\alpha}\Gamma(\frac{3+\alpha}{2})}\,.
  \end{split}
 \]

 Mimicking the reasoning above, the occurrence of an incident plane wave $e^{\ii x\sqrt{E}}$ from $-\infty$ towards the origin, producing after the scattering a transmitted plane wave $e^{\ii x\sqrt{E}}$ towards $+\infty$ and a reflected plane wave $e^{-\ii x\sqrt{E}}$ back towards $-\infty$, corresponds to the choice
 \[\tag{c}\label{eq:taaac3}
  A_2^-\,=\,1\,,\qquad A_2^+\,=\,0\,,
 \]
 in which case \eqref{eq:generalgEscatt} becomes
 \[
  g_E(x)\;\sim\;
   {\textstyle\sqrt{\frac{2}{\pi}}}\,E^{-\frac{1}{4}}e^{\ii\frac{\pi}{4}(2+\alpha)}\begin{pmatrix}
   A_1^-\,e^{-\ii\frac{\pi}{2}(2+\alpha)}\,e^{-\ii x\sqrt{E}}+e^{\ii x\sqrt{E}} \\
   A_1^+\,e^{-\ii\frac{\pi}{2}(2+\alpha)}\,e^{\ii x\sqrt{E}}
  \end{pmatrix}\qquad\textrm{as }|x|\to +\infty\,.
 \]
 The transmission and reflection coefficients for the scattering are now, respectively,
 \[
   T^{(\rightarrow)}_{\alpha,a,\gamma}(E)\;=\;|A_1^+|^2\,,\qquad R^{(\rightarrow)}_{\alpha,a,\gamma}(E)\;=\;|A_1^-|^2\,.
 \]
  Solving \eqref{eq:taaac1} with the new condition \eqref{eq:taaac3} yields
 \[
  \begin{split}
   A_1^-\;&=\;\frac{\,E^{\frac{1+\alpha}{2}}\Gamma(\frac{1-\alpha}{2})\,(|a|^2-\,e^{\ii\pi\alpha})+\ii\,\gamma\,2^{1+\alpha}e^{\ii\frac{\pi}{2}\alpha}\Gamma(\frac{3+\alpha}{2})}{\,E^{\frac{1+\alpha}{2}}\Gamma(\frac{1-\alpha}{2})(1+|a|^2)+\ii\,\gamma\,2^{1+\alpha}e^{\ii\frac{\pi}{2}\alpha}\Gamma(\frac{3+\alpha}{2})}\,, \\
   A_1^+\;&=\;-\frac{\,E^{\frac{1+\alpha}{2}}(1+e^{\ii\pi\alpha})\,\Gamma(\frac{1-\alpha}{2})\,a}{\,E^{\frac{1+\alpha}{2}}\Gamma(\frac{1-\alpha}{2})(1+|a|^2)+\ii\,\gamma\,2^{1+\alpha}e^{\ii\frac{\pi}{2}\alpha}\Gamma(\frac{3+\alpha}{2})}\,.
  \end{split}
 \]
 An easy computation shows that 
 \[
  \begin{split}
   T^{(\rightarrow)}_{\alpha,a,\gamma}(E)\;=\;T^{(\leftarrow)}_{\alpha,a,\gamma}(E)\;&=:\;T_{\alpha,a,\gamma}(E) \\
   R^{(\rightarrow)}_{\alpha,a,\gamma}(E)\;=\;R^{(\leftarrow)}_{\alpha,a,\gamma}(E)\;&=:\;R_{\alpha,a,\gamma}(E)\,,
  \end{split}
 \]
 whence the final expressions \eqref{eq:TR-IIa}, as well as the validity of \eqref{eq:TpEi1} (see also Remark \ref{rem:Jcont} below). This proves part (i).

 All the claims of part (ii), (iv), and (v) then follow straightforwardly from  \eqref{eq:TR-IIa}. Concerning part (iii), one sees from \eqref{eq:TR-IIa} that the scattering is reflection-less when
 \begin{equation*}\tag{d}\label{eq:taaac4}
   E\;=\;\bigg(\frac{\,\ii\,\gamma\,2^{1+\alpha}\,e^{\ii\frac{\pi}{2}\alpha}\,\Gamma(\frac{3+\alpha}{2})}{(|a|^2 e^{\ii\pi\alpha}-1)\Gamma(\frac{1-\alpha}{2})}\bigg)^{\!\frac{2}{1+\alpha}},
 \end{equation*}
 provided that
 \[
  0\;<\;\frac{\ii\gamma\,e^{\ii\frac{\pi}{2}\alpha}}{|a|^2e^{\ii\pi\alpha}-1}\;=\;\frac{\,(|a|^2+1)\gamma\sin\frac{\pi}{2}\alpha+\ii(|a|^2-1)\gamma\cos\frac{\pi}{2}\alpha\,}{(|a|^2\cos\pi\alpha-1)^2+|a|^4\sin^2\pi\alpha}\,,
 \]
 a condition that is only satisfied when $\alpha\in(0,1)$, \emph{and} $|a|=1$, \emph{and} $\gamma>0$. If this is the case, the quantity computed above becomes
 \[
  \frac{\gamma\sin\frac{\pi}{2}\alpha}{\,1-\cos\pi\alpha\,}
 \]
 and the expression \eqref{eq:taaac4} takes finally the form \eqref{eq:Etransm}. 
\end{proof}

 \begin{remark}\label{rem:Jcont}
  The identity $ T_{\alpha,a,\gamma}(E)+R_{\alpha,a,\gamma}(E)=1$ was checked directly in the proof above. In fact, it could have been claimed a priori, as one does with the scattering governed by a Schr\"{o}dinger operator $H_{\mathrm{Schr}}=-\frac{\ud^2}{\ud x^2}+V$ with smooth and fast decaying real potential $V$, but with a noticeable difference in the reasoning. For  $H_{\mathrm{Schr}}$, the problem $H_{\mathrm{Schr}}g=Eg$ with $E>0$ is an ODE on $\mathbb{R}$ whose (two linearly independent) solutions are continuous, and with continuous derivative, over the whole $\mathbb{R}$. For any such $g$, the probability current
  \begin{equation}\label{eq:current}
   J_g(x)\;:=\;-\ii\big(\overline{g(x)}g'(x)-g(x)\overline{g'(x)}\big)\;=\;2\,\mathfrak{Im}\big(\overline{g(x)}g'(x)\big)
  \end{equation}
  is actually conserved $\forall x\in\mathbb{R}$ because
  \[
   J_g'(x)\;=\;2\,\mathfrak{Im}\big(\overline{g(x)}g''(x)\big)\;=\;2\,\mathfrak{Im}\big((V(x)-E)|g(x)|^2\big)\;=\;0
  \]
  (having used the identity $-g''+Vg=Eg$ and the reality of $V$). Mimicking the present setting, one could also reason by saying that on each open half line $\mathbb{R}^\pm$ the ODE $H_{\mathrm{Schr}}g=Eg$ gives rise to a conserved current, and the left and right currents do coincide because, taking the limit $x\to 0^\pm$ in \eqref{eq:current}, one exploits the continuity of $g$ and $g'$ at $x=0$. In turn, the conserved current allows one to conclude the following: for a solution $g$ to $H_{\mathrm{Schr}}g=Eg$ with asymptotics
  \begin{equation}\label{eq:rtsol}
   g(x)\;\sim\;
   \begin{cases}
    \;e^{\ii x\sqrt{E}}+r\,e^{-\ii x\sqrt{E}} & \textrm{as } x\to -\infty \\
    \;\qquad t\,e^{\ii x\sqrt{E}} & \textrm{as } x\to +\infty
   \end{cases}
  \end{equation}
  (such $g$ does exists, as $V$ has fast decrease) a simple calculation yields
  \begin{equation}
   \lim_{x\to-\infty}J_g(x)\;=\;2\sqrt{E}\,(1-|r|^2)\,,\qquad  \lim_{x\to+\infty}J_g(x)\;=\;2\sqrt{E}\,|t|^2\,,
  \end{equation}
  whence $|t|^2+|r|^2=1$ because $J_g(x)$ is constant. This gives the conservation of the incident flux into the sum of reflected and transmitted flux. In the present case, \eqref{eq:ODEpos} is a differential problem consisting of two separate ODEs on each half line, to which one superimposes the boundary condition at $x=0$ characteristic for the operator $A_{\alpha,a}^{[\gamma]}$. Precisely as in the ordinary Schr\"{o}dinger case, one has probability currents
    \begin{equation}\label{eq:currents}
   J^\pm_g(x)\;:=\;2\,\mathfrak{Im}\big(\overline{g^\pm(x)}(g^\pm)'(x)\big)
  \end{equation}
  on $\mathbb{R}^\pm$, each of which is conserved, and such that for a solution of type \eqref{eq:rtsol}
    \begin{equation}
   \lim_{x\to-\infty}J^-_g(x)\;=\;2\sqrt{E}\,(1-|r|^2)\,,\qquad  \lim_{x\to+\infty}J^+_g(x)\;=\;2\sqrt{E}\,|t|^2\,.
  \end{equation}
  However, the overall conservation $J^-=J^+$, which would lead again to $|t|^2+|r|^2=1$, cannot follow from the continuity of $g$ and $g'$ at $x=0$ as for the ordinary Schr\"{o}dinger case: such functions actually diverge as $|x|\to 0$, more precisely
  \begin{equation}\label{eq:asintg}
   \begin{split}
       g(x)\;&=\;
   \begin{cases}
    \;g_0^-(-x)^{-\frac{\alpha}{2}}+g_1^-(-x)^{1+\frac{\alpha}{2}}+o(x^{\frac{3}{2}}) & \textrm{as } x\uparrow 0^- \\
    \;g_0^+x^{-\frac{\alpha}{2}}+g_1^+ x^{1+\frac{\alpha}{2}}+o(x^{\frac{3}{2}}) & \textrm{as } x\downarrow 0^+
   \end{cases} \\
   g'(x)\;&=\;
   \begin{cases}
    \;\frac{\alpha}{2}\,g_0^-(-x)^{-(1+\frac{\alpha}{2})}-(1+\frac{\alpha}{2})\,g_1^-(-x)^{\frac{\alpha}{2}}+o(x^{\frac{1}{2}}) & \textrm{as } x\uparrow 0^- \\
    \;-\frac{\alpha}{2}\,g_0^+ x^{-(1+\frac{\alpha}{2})}+(1+\frac{\alpha}{2})\,g_1^+ x^{\frac{\alpha}{2}}+o(x^{\frac{1}{2}}) & \textrm{as } x\downarrow 0^+\,,
   \end{cases}
   \end{split}
  \end{equation}
  with
   \begin{equation}\label{eq:againcond}
  \begin{cases}
   \qquad g_{0}^+\,=\,a\,g_{0}^- \\
   \;g_{1}^-+\overline{a}g_{1}^+\,=\,\gamma\,g_{0}^-\,.
  \end{cases}
 \end{equation}
    A new check is therefore needed: from \eqref{eq:currents} and \eqref{eq:asintg} one computes
  \begin{equation}
   J^\pm\;:=\;\lim_{x\to 0^\pm}J^\pm(x)\;=\;\pm \,2\,(1+\alpha)\,\mathfrak{Im}\big(\overline{g_0^\pm}\,g_1^\pm \big)
  \end{equation}
 and using \eqref{eq:againcond} one finds
  \begin{equation}
   \begin{split}
    J^-\;&=\;-2(1+\alpha)\,\mathfrak{Im}\Big(\frac{1}{\overline{a}}\overline{g_0^+}\cdot\Big(\gamma\,\frac{1}{a}g_0^+-\overline{a}g_1^+\Big)\Big) \\
    &=\;2(1+\alpha)\,\mathfrak{Im}\big(\overline{g_0^+}\,g_1^+ \big)\;=\;J^+\,.
   \end{split}
  \end{equation}
 This finally establishes the a priori information that $ T_{\alpha,a,\gamma}(E)+R_{\alpha,a,\gamma}(E)=1$.  
   \end{remark}

\section{Reconstruction of the spectral content of the fibred extensions}

The previous results on each fibre from Section \ref{sec:spectral-in-fibre} can be now assembled together to produce the spectral analysis of the fibred operators \eqref{eq:HalphaFriedrichs_unif-fibred}-\eqref{eq:Halpha-III_unif-fibred}.

\subsection{Spectrum of the direct sum}\label{subsec:SpectrumDS}~

We shall establish here the following auxiliary result.

\begin{proposition}\label{prop:spectra-in-direct-sum}
 Consider the Hilbert space orthogonal direct sum $\cH=\bigoplus_{k\in\mathbb{Z}}\mathfrak{h}_k$. Let $(S_k)_{k\in\mathbb{Z}}$ be a collection of self-adjoint operators, the $k$-th of which acts in $\mathfrak{h}_k$, and let $S$ be the self-adjoint operator direct sum $S=\bigoplus_{k\in\mathbb{Z}}S_k$ acting in $\cH$. Then:
 \begin{eqnarray}
  \sigma(S) \!\!&=&\!\!\overline{\bigcup_{k \in \mathbb{Z}} \sigma(S_k)}\,, \label{eq:spTspTk}\\
  \sigma_{\mathrm{ess}}(S) \!\!&\supset&\!\! \bigcup_{k \in \mathbb{Z}}\sigma_{\mathrm{ess}}(S_k)\,, \label{eq:spessTspessTk} \\
   \sigma_{\mathrm{disc}}(S) \!\!&\subset&\!\! \bigcup_{k \in \mathbb{Z}}\sigma_{\mathrm{disc}}(S_k)\,. \label{eq:spdiscTspdiscTk}
 \end{eqnarray}
\end{proposition}

\begin{remark}
Inclusions \eqref{eq:spessTspessTk} and \eqref{eq:spdiscTspdiscTk} may well be strict. Indeed, $\sigma_{\mathrm{ess}}(S)$ may contain accumulation points of eigenvalues, or eigenvalues with infinite multiplicity that are not contained in any $\sigma_{\mathrm{ess}}(S_k)$, and $\sigma_{\mathrm{disc}}(S_k)$ may contain eigenvalues that are contained in $\sigma_{\mathrm{ess}}(S_{k'})$ for $k' \neq k$.
\end{remark}

\begin{remark}
 In the direct sum operator structure $S=\bigoplus_{k\in\mathbb{Z}}S_k$, the self-adjointness of $S$ is equivalent to the self-adjointness of all the $S_k$'s (see, e.g., \cite[Lemma 2.2]{GMP-Grushin2-2020}).
\end{remark}

Prior to demonstrating Proposition \ref{prop:spectra-in-direct-sum}, let us recall a few standard properties of the operator direct sum, for which we skip the simple proof. If $S$ and the $S_k$'s are \emph{closed operators} with respect to the corresponding Hilbert spaces, so that the notion of resolvents $\rho(S)$ and $\rho(S_k)$ is non-trivial, then the invertibility of $S$ implies the invertibility of all the $S_k$'s (each one in the respective space), with $S^{-1}= \bigoplus_{k\in\mathbb{Z}}S_k^{-1}$. As a consequence,
\begin{equation}\label{eq:rho-in-intrhok}
 \rho(S)\;\subset\;\bigcap_{k\in\mathbb{Z}}\rho(S_k)\,.
\end{equation}

\begin{proof}[Proof of Proposition \ref{prop:spectra-in-direct-sum}]
In order to establish \eqref{eq:spTspTk}, let us first show that
\[\tag{*}\label{eq:rho-interior}
 \rho(S)\;=\;\mathrm{int} \bigg(\bigcap_{k \in \mathbb{Z}} \rho(S_k) \bigg),
\]
where `$\mathrm{int}$' here denotes the interior of the set (in the ordinary topology of $\mathbb{C}$). The `$\subset$'-inclusion follows from \eqref{eq:rho-in-intrhok} upon taking the interior on both sides, and using the fact that $\rho(S)=\mathrm{int}(\rho(S))$ (the resolvent set is open in $\mathbb{C}$). For the `$\supset$'-inclusion, pick $\lambda\in\mathrm{int} \big(\bigcap_{k \in \mathbb{Z}} \rho(S_k) \big)$. Each operator $S_k-\lambda\mathbbm{1}$ has then everywhere defined and bounded inverse in $\mathfrak{h}_k$. Moreover, because of the status of interior point, 
\[
 \varepsilon\;:=\;\mathrm{dist}\bigg(\lambda,\bigcup_{k \in \mathbb{Z}} \sigma(S_k)\bigg)\;=\;\mathrm{dist}\bigg(\lambda,\mathbb{C}\setminus\bigcap_{k \in \mathbb{Z}} \rho(S_k)\bigg)\;>\;0\,.
\]
On the subspace $\mathcal{D}$, dense in $\cH$, defined as
\[
 \mathcal{D}\;:=\;\big\{(\psi_k)_{k\in\mathbb{Z}}\in\cH\,\big|\,\psi_k\neq 0\textrm{ only for finitely many }k\big\}\,,
\]
the operator $\bigoplus_{k \in \mathbb{Z}} (S_k - \lambda \mathbbm{1}_k)^{-1}$ is well defined because the $k$-th summand is bounded in $\mathfrak{h}_k$ and when applied to an element of $\mathcal{D}$ the sum is only finite. Such operator is also bounded (in $\cH$) because, by a standard resolvent estimate,
\[
	\bigg\Vert \bigoplus_{k \in \mathbb{Z}} (S_k - \lambda \mathbbm{1}_k)^{-1} \bigg\Vert \; \leq \; \varepsilon^{-1}\,.
\]
Thus, it extends uniquely to an everywhere defined and bounded operator on the whole $\cH$ which is easily seen to invert $S-\lambda\mathbbm{1}$. This proves that $\lambda\in\rho(S)$.

Now that \eqref{eq:rho-interior} is established, we have
\[
 \begin{split}
  \sigma(S)\;&=\;\overline{\sigma(S)}\;=\;\overline{\,\mathbb{C}\setminus\mathrm{int} \bigg(\bigcap_{k \in \mathbb{Z}} \rho(S_k) \bigg)}\;=\;\overline{\,\mathbb{C}\setminus\bigcap_{k \in \mathbb{Z}} \rho(S_k)} \\
  &=\;\overline{\bigcup_{k \in \mathbb{Z}}\big(\mathbb{C}\setminus\rho(S_k)\big)}\;=\;\overline{\bigcup_{k \in \mathbb{Z}} \sigma(S_k)}\,.
 \end{split}
\]
In the first identity above we exploited the closedness of the spectrum; in the second we applied \eqref{eq:rho-interior}; in the third we used the property $\overline{\mathbb{C}\setminus X}=\overline{\mathbb{C}\setminus \mathrm{int}(X)}$. Thus, \eqref{eq:spTspTk} is proved.

Concerning \eqref{eq:spessTspessTk}, if $\lambda\in\sigma_{\mathrm{ess}}(S_k)$, then there exists a singular sequence in $\mathfrak{h}_k$ relative to $\lambda$. Since $\hh_k \subset \mathcal{H}$, such a sequence is also a singular sequence relative to $\lambda$ considered as a spectral point of $S$. Therefore, $\sigma_{\mathrm{ess}}(S_k)\subset\sigma_{\mathrm{ess}}(S)$. As $k$ was arbitrary, \eqref{eq:spessTspessTk} follows.

Last, concerning \eqref{eq:spdiscTspdiscTk}, pick $\lambda\in\sigma_{\mathrm{disc}}(S)$, i.e., $\lambda$ is an isolated eigenvalue of $S$ with finite multiplicity. Denoting by $\psi\equiv (\psi_k)_{k\in\mathbb{Z}}\in\cH$ one of the corresponding eigenvectors, one has $S\psi=\lambda \psi$, whence $S_k\psi_k=\lambda \psi_k$ $\forall k\in\mathbb{Z}$ with only finitely many non-vanishing $\psi_k$ (otherwise the finite multiplicity would be violated). For the finitely many $k$'s for which that happens, $\lambda$ is therefore an eigenvalue of $S_k$, it has necessarily finite multiplicity (again because otherwise as an eigenvalue of $S$ it would have infinite multiplicity), and it is isolated, because by assumption there is a sufficiently small neighbourhood of $\lambda$ containing, apart from $\lambda$ itself, only points of $\rho(S)$, and hence, owing to \eqref{eq:rho-in-intrhok}, only points of $\rho(S_k)$. This proves that $\lambda\in\sigma_{\mathrm{disc}}(S_k)$ and finally establishes \eqref{eq:spdiscTspdiscTk}.
\end{proof}

\subsection{Spectrum of the fibred extensions of $\mathscr{H}_\alpha$}

\begin{theorem}\label{thm:fibredHalpha-spectrum}
 Let $\alpha\in[0,1)$, $\gamma\in\mathbb{R}$, $a\in\mathbb{C}$, $\Gamma\equiv(\gamma_1,\gamma_2,\gamma_3,\gamma_4)\in\mathbb{R}^4$. The spectra of the self-adjoint extensions of type \eqref{eq:HalphaFriedrichs_unif-fibred}-\eqref{eq:Halpha-III_unif-fibred} of the operator $\mathscr{H}_\alpha$ defined in \eqref{eq:unitary_transf_pm}-\eqref{eq:actiondomainHalpha} have the following properties.
 \begin{itemize}
  \item[(i)] Essential spectrum:
  \begin{equation}
 \begin{split}
  &\sigma_{\mathrm{ess}}(\mathscr{H}_{\alpha,\mathrm{F}})\;=\;\sigma_{\mathrm{ess}}\big(\mathscr{H}_{\alpha,\mathrm{R}}^{[\gamma]}\big)\;=\;\sigma_{\mathrm{ess}}\big(\mathscr{H}_{\alpha,\mathrm{L}}^{[\gamma]}\big) \\
  &\quad =\;\sigma_{\mathrm{ess}}\big(\mathscr{H}_{\alpha,a}^{[\gamma]}\big)\;=\;\sigma_{\mathrm{ess}}\big(\mathscr{H}_{\alpha}^{[\Gamma]}\big)\;=\;[0,+\infty)\,.
 \end{split}
\end{equation}
 \item[(ii)] Discrete spectrum: the discrete spectrum of each such operator is only negative and consists of finitely many eigenvalues. In particular, any such operator is lower semi-bounded.
 \item[(iii)] For each considered extension, the essential spectrum $[0,+\infty)$ contains (countably) infinite embedded eigenvalues, each of finite multiplicity. There is no accumulation of embedded eigenvalues.
 \item[(iv)] $\mathscr{H}_{\alpha,\mathrm{F}}$ has no negative eigenvalue, thus no discrete spectrum.
 \item[(v)] Both $\mathscr{H}_{\alpha,\mathrm{R}}^{[\gamma]}$ and $\mathscr{H}_{\alpha,\mathrm{L}}^{[\gamma]}$ have negative spectrum if and only if $\gamma<0$, in which case their negative spectrum is finite, discrete, and consists respectively of $\mathcal{N}_-\big(\mathscr{H}_{\alpha,\mathrm{R}}^{[\gamma]}\big)$ and  $\mathcal{N}_-\big(\mathscr{H}_{\alpha,\mathrm{L}}^{[\gamma]}\big)$ negative eigenvalues, counted with multiplicity, with
 \begin{equation}\label{eq:multnegspec-IRIL}
  \mathcal{N}_-\big(\mathscr{H}_{\alpha,\mathrm{R}}^{[\gamma]}\big)\;=\;\mathcal{N}_-\big(\mathscr{H}_{\alpha,\mathrm{L}}^{[\gamma]}\big)\;=\;2\lfloor (1+\alpha)|\gamma|\rfloor +1\qquad(\gamma<0)\,.
 \end{equation}
 \item[(vi)] $\mathscr{H}_{\alpha,a}^{[\gamma]}$ has negative spectrum if and only if $\gamma<0$, in which case its negative spectrum is finite, discrete, and consists of $\mathcal{N}_-\big(\mathscr{H}_{\alpha,a}^{[\gamma]}\big)$ negative eigenvalues, counted with multiplicity, with 
 \begin{equation}\label{eq:multnegspec-IIa}
  \mathcal{N}_-\big(\mathscr{H}_{\alpha,a}^{[\gamma]}\big)\;=\;2\Big\lfloor \frac{1+\alpha}{\,1+|a|^2\,}|\gamma|\Big\rfloor +1\qquad(\gamma<0)\,.
 \end{equation}
  \item[(vii)] $\mathscr{H}_{\alpha}^{[\Gamma]}$ has negative spectrum if and only if $\gamma_1+\gamma_4 - \sqrt{(\gamma_1-\gamma_4)^2 +4 (\gamma_2^2+\gamma_3^2)}<0$, in which case its negative spectrum is finite, discrete, and consists of $\mathcal{N}_-\big(\mathscr{H}_{\alpha}^{[\Gamma]}\big)$ negative eigenvalues, counted with multiplicity, with 
  \begin{equation}\label{eq:multnegspec-IIImathscrHaG}
   \begin{split}
    \mathcal{N}_-\big(\mathscr{H}_{\alpha}^{[\Gamma]}\big)\;&=\;\textstyle 2\lfloor -(1+\alpha)\big(\gamma_1+\gamma_4+\sqrt{(\gamma_1-\gamma_4)^2+4 (\gamma_2^2 + \gamma_3^2)}\,\big)\rfloor \\
    & \quad + \textstyle 2\lfloor -(1+\alpha)\big(\gamma_1+\gamma_4-\sqrt{(\gamma_1-\gamma_4)^2+4 (\gamma_2^2 + \gamma_3^2)}\,\big)\rfloor \\
    & \quad + n_0(\Gamma)
   \end{split}
  \end{equation}
   with
   \begin{equation}
    n_0(\Gamma)\,:=\,\begin{cases}
   \;2 & \textrm{if}\quad\gamma_1 \, \gamma_4>\gamma_2^2 + \gamma_3^2 \quad \text{and} \quad \gamma_1+\gamma_4<0 \\
   \;1 & \textrm{if}\quad \gamma_1 \, \gamma_4<\gamma_2^2 + \gamma_3^2\quad\textrm{or}\quad
   \begin{cases}
    \;\gamma_1 \, \gamma_4 =\gamma_2^2 + \gamma_3^2 \\
    \;\gamma_1+\gamma_4<0
   \end{cases} \\
   \;0 & \textrm{if}\quad \gamma_1 \, \gamma_4\geqslant\gamma_2^2 + \gamma_3^2 \quad \text{and} \quad \gamma_1+\gamma_4>0\,.
  \end{cases}
   \end{equation}
 \end{itemize}
\end{theorem}

\begin{corollary}\label{cor:positivity}
 The operator $\mathscr{H}_{\alpha,\mathrm{F}}$ is non-negative. The operators $\mathscr{H}_{\alpha,\mathrm{R}}^{[\gamma]}$, $\mathscr{H}_{\alpha,\mathrm{L}}^{[\gamma]}$, and $\mathscr{H}_{\alpha,a}^{[\gamma]}$ are non-negative if and only if $\gamma\geqslant 0$. The operator $\mathscr{H}_{\alpha}^{[\Gamma]}$ is non-negative if and only if 
 $\gamma_1\gamma_4\geqslant\gamma_2^2 + \gamma_3^2$ and $\gamma_1+\gamma_4>0$.
\end{corollary}

\begin{proof}[Proof of Theorem \ref{thm:fibredHalpha-spectrum}]
 Let us write $\widetilde{\mathscr{H}}_\alpha=\bigoplus_{k\in\mathbb{Z}}\widetilde{A}_\alpha(k)$ for any of the operators \eqref{eq:HalphaFriedrichs_unif-fibred}-\eqref{eq:Halpha-III_unif-fibred}.

 Obviously, $\lambda\in\mathbb{R}$ is an eigenvalue of $\widetilde{\mathscr{H}}_\alpha$, and hence $\widetilde{\mathscr{H}}_\alpha \psi=\lambda \psi$ for some non-zero $\psi\equiv(\psi_k)_{k\in\mathbb{Z}}\in\cH$ (the fibred Hilbert space \eqref{eq:Hxispace}), if and only if $\widetilde{A}_\alpha(k) \psi_k=\lambda \psi_k$ for all the $k$'s in $\mathbb{Z}$ for which $\psi_k$ is non-zero in the fibre Hilbert space $\mathfrak{h}$.

 Let us focus first on the negative eigenvalues $\lambda$ of $\widetilde{\mathscr{H}}_\alpha$. They necessarily come as negative eigenvalues of some of the $\widetilde{A}_\alpha(k)$'s. Owing to Propositions \ref{prop:SpectrumK0} and \ref{prop:NegativeEigenvaluesKDiffZero}, each $\widetilde{A}_\alpha(k)$ contributes with a finite number of negative eigenvalues (counting multiplicity). Moreover, and most importantly, only a finite number of $\widetilde{A}_\alpha(k)$'s have negative eigenvalues.

 Let us demonstrate this claim first for $\widetilde{\mathscr{H}}_\alpha\equiv\mathscr{H}_{\alpha,\mathrm{R}}^{[\gamma]}$ and postpone the same control of the other extensions to the second part of the proof. Thus, when $\widetilde{A}_\alpha(k)\equiv A_{\alpha,\mathrm{R}}^{[\gamma]}(k)$ and $k\in\mathbb{Z}\setminus\{0\}$, Proposition \ref{prop:NegativeEigenvaluesKDiffZero}(ii) states that this operator admits exactly one negative, non-degenerate eigenvalue if and only if $|k|<-(1+\alpha)\gamma$. This is never the case when $\gamma\geqslant 0$, whereas when $\gamma<0$ the number of $k$-modes with negative eigenvalue is precisely $2\lfloor (1+\alpha)|\gamma|\rfloor +1$, having now added to the counting also the negative eigenvalue of $A_{\alpha,\mathrm{R}}^{[\gamma]}(0)$ when $\gamma<0$ (Proposition \ref{prop:SpectrumK0}(iii)). Therefore, $\mathscr{H}_{\alpha,\mathrm{R}}^{[\gamma]}$ with $\gamma<0$ has exactly $2\lfloor (1+\alpha)|\gamma|\rfloor +1$ negative eigenvalues, counted with multiplicity, namely a \emph{finite} number of them. As said, we defer to the second part of the proof the analogous quantification for the other types of extensions.

 The conclusion so far is that $\widetilde{\mathscr{H}}_\alpha$ has only a finite number of negative eigenvalues, counted with multiplicity, which therefore all belong to $\sigma_{\mathrm{disc}}\big(\widetilde{\mathscr{H}}_\alpha\big)$. All other eigenvalues of $\widetilde{\mathscr{H}}_\alpha$ are non-negative, hence embedded in
 \[
  [0,+\infty)\;=\;\sigma_{\mathrm{ess}}\big(\widetilde{A}_\alpha(0)\big)\;=\;\bigcup_{k\in\mathbb{Z}}\sigma_{\mathrm{ess}}\big(\widetilde{A}_\alpha(k)\big)\;\subset\;\sigma_{\mathrm{ess}}\big(\widetilde{\mathscr{H}}_\alpha\big)
 \]
 (having used \eqref{eq:sess0} from Proposition \ref{prop:SpectrumK0} for the first identity, Proposition \ref{prop:NegativeEigenvaluesKDiffZero}(i) for the second, and \eqref{eq:spessTspessTk} from Proposition \ref{prop:spectra-in-direct-sum} for the third), and therefore do not belong to $\sigma_{\mathrm{disc}}\big(\widetilde{\mathscr{H}}_\alpha\big)$. The latter set is therefore only negative and finite. This proves part (ii).

 As argued above, $[0,+\infty)\subset\sigma_{\mathrm{ess}}\big(\widetilde{\mathscr{H}}_\alpha\big)$. Should there be $\lambda<0$ in $\sigma_{\mathrm{ess}}\big(\widetilde{\mathscr{H}}_\alpha\big)$, such $\lambda$ would either be an eigenvalue of infinite multiplicity or an accumulation point of $\sigma\big(\widetilde{\mathscr{H}}_\alpha\big)$. The first option is excluded, because as observed already negative eigenvalues of $\widetilde{\mathscr{H}}_\alpha$ can only be negative eigenvalues of a finite number of the $\widetilde{A}_\alpha(k)$'s, hence all with finite multiplicity. The second option is excluded as well, owing to \eqref{eq:spTspTk} of Proposition \ref{prop:spectra-in-direct-sum}, that implies that a negative accumulation point of $\sigma\big(\widetilde{\mathscr{H}}_\alpha\big)$ must be the limit of negative points in the $\sigma\big(\widetilde{A}_\alpha(k)\big)$'s, whereas the negative spectra of the $\widetilde{A}_\alpha(k)$'s are actually non-empty only for finitely many $k$'s, and each such negative spectrum is finite. Then, necessarily $\sigma_{\mathrm{ess}}\big(\widetilde{\mathscr{H}}_\alpha\big)=[0,+\infty)$. This proves part (i).

 Concerning the non-$\sigma_{\mathrm{disc}}\big(\widetilde{\mathscr{H}}_\alpha\big)$ eigenvalues embedded in $[0,+\infty)$, once again each of them must be an eigenvalue for some of the $\widetilde{A}_\alpha(k)$'s. As each $\widetilde{A}_\alpha(k)$ with $k\neq 0$ has infinitely many positive eigenvalues (Proposition \ref{prop:NegativeEigenvaluesKDiffZero}(i)), the number of embedded eigenvalues for $\widetilde{\mathscr{H}}_\alpha$, counting multiplicity, is (countably) infinite as well. Now, the only possibility for any such eigenvalue $\lambda\geqslant 0$ to have itself infinite multiplicity is that $\lambda$ is \emph{simultaneously} an eigenvalue for infinitely many distinct $\widetilde{A}_\alpha(k)$. This cannot be the case, though. Indeed, each $\widetilde{A}_\alpha(k)$ is lower semi-bounded and the Kre\u{\i}n-Vi\v{s}ik-Birman extension scheme provides the following estimate \emph{from below} of the bottom $\mathfrak{m}\big(\widetilde{A}_\alpha(k)\big)$ of its spectrum \cite[Theorem 6]{GMO-KVB2017}: 
 \[
  \mathfrak{m}\big(\widetilde{A}_\alpha(k)\big)\;\geqslant\;\frac{\,\mathfrak{m}\big(A_\alpha(k)\big)\cdot \mathfrak{m}(\mathcal{B}_k)}{\,\mathfrak{m}\big(A_\alpha(k)\big)+\mathfrak{m}(\mathcal{B}_k)\,}\,.
 \]
 In the last formula $\mathcal{B}_k$ indicates the Birman extension parameter of $\widetilde{A}_\alpha(k)$, namely one of the operators \eqref{eq:GammaIRIL}-\eqref{eq:IIkTheorem51}, $\mathfrak{m}(\mathcal{B}_k)$ denotes the bottom of its spectrum, and $\mathfrak{m}\big(A_\alpha(k)\big)$ is the lower bound \eqref{eq:lowerboundAak}. For sufficiently large $|k|$,
 \[
  \mathfrak{m}\big(A_\alpha(k)\big)\,\sim\,|k|^{\frac{2}{1+\alpha}}\,,\qquad \mathfrak{m}(\mathcal{B}_k)\,\sim\,|k|^{\frac{1-\alpha}{1+\alpha}}\,,
 \]
 whence $ \mathfrak{m}\big(\widetilde{A}_\alpha(k)\big)\gtrsim\,|k|^{\frac{1-\alpha}{1+\alpha}}$. This shows that the lowest eigenvalue of $\widetilde{A}_\alpha(k)$ grows with $|k|$, thus making impossible for a fixed $\lambda\geqslant 0$ to be simultaneous eigenvalue of infinitely many $\widetilde{A}_\alpha(k)$'s. The conclusion is that $\lambda$ has finite multiplicity. The same reasoning excludes that the embedded eigenvalues accumulate to some limit points. This completes the proof of part (iii).

 What remains to prove now is the precise quantification of the multiplicity of the finite negative discrete spectrum of the operator $\widetilde{\mathscr{H}}_\alpha$ for each possible type $\mathscr{H}_{\alpha,\mathrm{F}},\mathscr{H}_{\alpha,\mathrm{R}}^{[\gamma]},\mathscr{H}_{\alpha,\mathrm{L}}^{[\gamma]},\mathscr{H}_{\alpha,a}^{[\gamma]},\mathscr{H}_{\alpha}^{[\Gamma]}$, mimicking the same reasoning made above for $\mathscr{H}_{\alpha,\mathrm{R}}^{[\gamma]}$. Clearly $\mathscr{H}_{\alpha,\mathrm{F}}$ has no negative spectrum because the $A_{\alpha,\mathrm{F}}(k)$'s have neither (Propositions \ref{prop:SpectrumK0}(ii) and \ref{prop:NegativeEigenvaluesKDiffZero}(i)). And the multiplicity \eqref{eq:multnegspec-IRIL} of the negative spectrum of $\mathscr{H}_{\alpha,\mathrm{R}}^{[\gamma]}$ has been already computed above -- the same clearly applies to $\mathscr{H}_{\alpha,\mathrm{L}}^{[\gamma]}$. Thus, parts (iv) and (v) are proved.

 Concerning $\mathscr{H}_{\alpha,a}^{[\gamma]}$, Proposition \ref{prop:NegativeEigenvaluesKDiffZero}(iii) states that $A_{\alpha,a}^{[\gamma]}$ with $k\in\mathbb{Z}\setminus\{0\}$ admits exactly one negative, non-degenerate eigenvalue if and only if $|k|<-\frac{1+\alpha}{1+|a|^2}\gamma$. This is never the case when $\gamma\geqslant 0$, whereas when $\gamma<0$ the number of $k$-modes with negative eigenvalue is precisely $2\lfloor \frac{1+\alpha}{\,1+|a|^2\,}|\gamma|\rfloor +1$, having now added to the counting also the negative eigenvalue of $A_{\alpha,a}^{[\gamma]}(0)$ when $\gamma<0$ (Proposition \ref{prop:SpectrumK0}(iv)). Therefore, $\mathscr{H}_{\alpha,a}^{[\gamma]}$ with $\gamma<0$ has exactly $2\lfloor \frac{1+\alpha}{\,1+|a|^2\,}|\gamma|\rfloor +1$ negative eigenvalues, counted with multiplicity. Part (vi) is proved.

 Last, concerning $\mathscr{H}_{\alpha}^{[\Gamma]}$, the counting goes as follows. In the non-zero modes, the $A_{\alpha}^{[\Gamma]}(k)$'s with \emph{exactly two} negative eigenvalues are those for which (see \eqref{eq:negEV-k_extIII_2ev} from Proposition \ref{prop:NegativeEigenvaluesKDiffZero})
 \begin{equation*}
 |k|\;<\;-\textstyle(1+\alpha)\big(\gamma_1+\gamma_4+\sqrt{(\gamma_1-\gamma_4)^2+4 (\gamma_2^2 + \gamma_3^2)}\,\big)\,,
 \end{equation*}
 therefore their number amounts to
 \[
  N_2\;:=\;\textstyle 2\lfloor -(1+\alpha)\big(\gamma_1+\gamma_4+\sqrt{(\gamma_1-\gamma_4)^2+4 (\gamma_2^2 + \gamma_3^2)}\,\big)\rfloor\,.
 \]
 The $A_{\alpha}^{[\Gamma]}(k)$'s instead with \emph{only one} negative eigenvalue are those for which (see \eqref{eq:negEV-k_extIII_1ev} above)
 \begin{equation*}
 \begin{split}
  & -\textstyle(1+\alpha)\big(\gamma_1+\gamma_4+\sqrt{(\gamma_1-\gamma_4)^2+4 (\gamma_2^2 + \gamma_3^2)}\,\big)\;\leqslant \\
  & \qquad\qquad \leqslant\; |k|\;<\;-\textstyle(1+\alpha)\big(\gamma_1+\gamma_4-\sqrt{(\gamma_1-\gamma_4)^2+4 (\gamma_2^2 + \gamma_3^2)}\,\big)\,,
 \end{split}
 \end{equation*}
 therefore their number amounts to
 \[
  \begin{split}
     & N_1\;:=\;\textstyle 2\lfloor -(1+\alpha)\big(\gamma_1+\gamma_4-\sqrt{(\gamma_1-\gamma_4)^2+4 (\gamma_2^2 + \gamma_3^2)}\,\big)\rfloor \\
     &\qquad \qquad \textstyle -2\lfloor -(1+\alpha)\big(\gamma_1+\gamma_4+\sqrt{(\gamma_1-\gamma_4)^2+4 (\gamma_2^2 + \gamma_3^2)}\,\big)\rfloor\,.
  \end{split}
 \]
 All other $A_{\alpha}^{[\Gamma]}(k)$'s do not have negative eigenvalues. Thus, the contribution in terms of number of negative eigenvalues of $\mathscr{H}_{\alpha}^{[\Gamma]}$ from its non-zero modes is $2N_2+N_1$, counting multiplicity. To this quantity we have to add the number $n_0$ of negative eigenvalues from the zero-mode, namely from $A_{\alpha}^{[\Gamma]}(0)$: owing to Proposition \ref{prop:SpectrumK0}(v),
 \[
  n_0\;=\;
  \begin{cases}
   \;2 & \textrm{if}\quad\gamma_1 \, \gamma_4>\gamma_2^2 + \gamma_3^2 \quad \text{and} \quad \gamma_1+\gamma_4<0 \\
   \;1 & \textrm{if}\quad \gamma_1 \, \gamma_4<\gamma_2^2 + \gamma_3^2\quad\textrm{or}\quad
   \begin{cases}
    \;\gamma_1 \, \gamma_4 =\gamma_2^2 + \gamma_3^2 \\
    \;\gamma_1+\gamma_4<0
   \end{cases} \\
   \;0 & \textrm{otherwise, i.e., if}\quad  \gamma_1+\gamma_4 - \sqrt{(\gamma_1-\gamma_4)^2 +4 (\gamma_2^2+\gamma_3^2)} \;\geqslant \;0\,.
  \end{cases}
 \]
 The conclusion is that the number $\mathcal{N}_-\big(\mathscr{H}_{\alpha}^{[\Gamma]}\big)$ of negative eigenvalues of $\mathscr{H}_{\alpha}^{[\Gamma]}$, counting multiplicity, is precisely $2N_2+N_1+n_0$, which yields the expression \eqref{eq:multnegspec-IIImathscrHaG}. Moreover, \emph{none} of the fibre operators $A_{\alpha}^{[\Gamma]}(k)$, $k\in\mathbb{Z}$, admits negative eigenvalues if and only if (see \eqref{eq:III0-none-neg} and \eqref{eq:negEV-k_extIII_0ev} above)
 \[
  \begin{split}
   &\gamma_1+\gamma_4 - \textstyle\sqrt{(\gamma_1-\gamma_4)^2 +4 (\gamma_2^2+\gamma_3^2)} \;\geqslant \;0 \quad\textrm{and}\\
   &|k|\;\geqslant\;-\textstyle(1+\alpha)\big(\gamma_1+\gamma_4-\sqrt{(\gamma_1-\gamma_4)^2+4 (\gamma_2^2 + \gamma_3^2)}\,\big)\quad \forall k\in\mathbb{Z}\setminus\{0\}\,,
  \end{split}
 \]
 which is equivalent to just $\gamma_1+\gamma_4 - \sqrt{(\gamma_1-\gamma_4)^2 +4 (\gamma_2^2+\gamma_3^2)}\geqslant0$. The latter condition characterises the absence of negative spectrum for $\mathscr{H}_{\alpha,a}^{[\gamma]}$, can be also interpreted as the non-negativity of the lowest eigenvalue of the hermitian matrix
 \[
  \widetilde{\Gamma}\;:=\;
  \begin{pmatrix}
   \gamma_1 & \gamma_2+\ii\gamma_3 \\
   \gamma_2-\ii\gamma_3 & \gamma_4
  \end{pmatrix},
 \]
 and is therefore equivalent to $\mathrm{Tr}(\widetilde{\Gamma})>0$ and $\det(\widetilde{\Gamma})\geqslant 0$, i.e., to $\gamma_1+\gamma_4>0$ and $\gamma_1\gamma_4\geqslant\gamma_2^2 + \gamma_3^2$. Part (vii) is thus proved.
\end{proof}

\subsection{Inverse unitary transformations}\label{sec:inverseunitarytransf}~

By inverting the unitary transformation of \eqref{eq:unitary_transf_pm}, thus exploiting the identities \eqref{eq:inverseUniary_Fri}-\eqref{eq:inverseUniary_III}, we are now able to translate the information obtained so far on the uniformly fibred self-adjoint extensions of $\mathscr{H}_\alpha$ to the corresponding self-adjoint extensions of $H_\alpha$.

\begin{proof}[Proof of Theorem \ref{thm:positivity}]
 The non-negativity is preserved between the pairs of unitarily equivalent operators in  \eqref{eq:inverseUniary_Fri}-\eqref{eq:inverseUniary_III}, therefore the theorem follows at once from Corollary \ref{cor:positivity}.
\end{proof}

\begin{proof}[Proof of Theorem \ref{thm:MainSpectral}]
 Discrete and essential spectrum are preserved between the pairs of unitarily equivalent operators in  \eqref{eq:inverseUniary_Fri}-\eqref{eq:inverseUniary_III}, as well as the multiplicity and degeneracy of eigenvalues, therefore the theorem follows at once from Theorem \ref{thm:fibredHalpha-spectrum}.
\end{proof}

\begin{proof}[Proof of Theorem \ref{thm:GroundStateCH}]
 Consider the regime when an operator of the type \eqref{eq:HalphaR_unif-fibred}-\eqref{eq:Halpha-III_unif-fibred}, collectively denoted as $\widetilde{\mathscr{H}}_\alpha=\bigoplus_{k\in\mathbb{Z}}\widetilde{A}_\alpha(k)$, has negative spectrum (Theorem \ref{thm:fibredHalpha-spectrum}). Build $\psi\equiv(\psi_k)_{k\in\mathbb{Z}}\in\cH$ whose components are all zero but for $\psi_0$, taken to be the eigenfunction in $\mathfrak{h}$ relative to the zero-mode lowest negative eigenvalue $E_0( \widetilde{A}_\alpha(0))$ (Proposition \ref{prop:SpectrumK0}). Then $\widetilde{\mathscr{H}}_\alpha\,\psi=E_0( \widetilde{A}_\alpha(0))\psi$. Moreover, as a consequence of the strict ordering established in Lemma \ref{lem:OrderRelationOps}, $E_0( \widetilde{A}_\alpha(0))$ is the ground state energy of $\widetilde{\mathscr{H}}_\alpha$ with ground state vector $\psi$. The degeneracy of $E_0( \widetilde{A}_\alpha(0))$ as lowest eigenvalue of $\widetilde{\mathscr{H}}_\alpha$ is the same as the degeneracy as lowest eigenvalue of $\widetilde{A}_\alpha(0)$. With this reasoning and Proposition \ref{prop:SpectrumK0} one characterises the ground state of $\widetilde{\mathscr{H}}_\alpha$ for each of the types  \eqref{eq:HalphaR_unif-fibred}-\eqref{eq:Halpha-III_unif-fibred}. Let now $\widetilde{H}_\alpha$ be the unitarily equivalent counterpart of $\widetilde{\mathscr{H}}_\alpha$ through the transformations in \eqref{eq:inverseUniary_IR}-\eqref{eq:inverseUniary_III}, hence one of the self-adjoint operators acting on $L^2(M,\mu_\alpha)$ and classified in Theorem \ref{thm:H_alpha_fibred_extensions} -- apart from the Friedrichs extension that clearly has no negative spectrum. The lowest negative eigenvalue and its degeneracy are preserved by unitary equivalence, hence can be immediately read out from Proposition \ref{prop:SpectrumK0}, through the above reasoning. The corresponding eigenfunction is transformed from $\cH$ to $L^2(M,\ud\mu_\alpha)$ by inverting \eqref{eq:global-unitary-pm}. Thus, in view of \eqref{eq:defF2} $(\mathcal{F}_2^{-1}\psi)(x,y)= \frac{1}{\sqrt{2\pi}}\,\psi_0(x)$, and in view of \eqref{eq:unit1} $(U_\alpha^{-1}\mathcal{F}_2^{-1}\psi)(x,y)= \frac{1}{\sqrt{2\pi}}\,|x|^{\frac{\alpha}{2}}\,\psi_0(x)$. As the original $\psi$ has only zero-mode, the transformed ground state function $U_\alpha^{-1}\mathcal{F}_2^{-1}\psi$ is \emph{constant} in $y$. Applying such scheme to the explicit eigenfunctions $g_{\alpha}^{(\mathrm{I_R})}$, $g_{\alpha}^{(\mathrm{I_L})}$, $g_{\alpha}^{(\mathrm{II}_a)}$, $g_{\alpha}^{(\mathrm{III})}$, $g_{\alpha,+}^{(\mathrm{III})}$, $g_{\alpha,-}^{(\mathrm{III})}$ identified in Proposition \ref{prop:SpectrumK0} yields at once the corresponding ground state functions for $H_{\alpha,\mathrm{R}}^{[\gamma]}$, $H_{\alpha,\mathrm{L}}^{[\gamma]}$, $H_{\alpha,a}^{[\gamma]}$, $H_{\alpha}^{[\Gamma]}$. 
\end{proof}

 \begin{proof}[Proof of Proposition \ref{prop:Friedrichs-groundstate}]
 Clearly, due to unitary equivalence, it suffices to reason in terms of $\mathscr{H}_{\alpha,\mathrm{F}}$. As argued already for the proof of Theorem \ref{thm:fibredHalpha-spectrum}, the lowest eigenvalue of $\mathscr{H}_{\alpha,\mathrm{F}}$ must be an eigenvalue for some of the fibre operators $A_{\alpha,\mathrm{F}}(k)$. $A_{\alpha,\mathrm{F}}(0)$ has no eigenvalues (Proposition \ref{prop:SpectrumK0}(ii)). The $A_{\alpha,\mathrm{F}}(k)$'s with $k\neq 0$ have indeed purely discrete positive spectrum (Proposition \ref{prop:NegativeEigenvaluesKDiffZero}(i)), and the lowest contribution to the eigenvalues of $\mathscr{H}_{\alpha,\mathrm{F}}$ only comes from the lowest, non-degenerate eigenvalue of $A_{\alpha,\mathrm{F}}(1)$ and the lowest, non-degenerate eigenvalue of $A_{\alpha,\mathrm{F}}(-1)$, owing to the operator ordering established in Lemma \ref{lem:OrderRelationOps}. Such eigenvalues are equal (because the two fibre operators are unitarily equivalent through the left$\leftrightarrow$right symmetry), but of course the two eigenfunctions in the fibre correspond to two linearly independent eigenfunctions for $\mathscr{H}_{\alpha,\mathrm{F}}$. Therefore $E_0(H_{\alpha,\mathrm{F}})$ is two-fold degenerate. Next, as the Friedrichs extension preserves the lower bound, the bottom $\mathfrak{m}(A_{\alpha,\mathrm{F}}(1))$ of the spectrum of $A_{\alpha,\mathrm{F}}(1)$ satisfies the same estimate \eqref{eq:lowerboundAak} namely
 \begin{equation*}
 \mathfrak{m}(A_{\alpha,\mathrm{F}}(1))\;\geqslant\;(1+\alpha)\big(\textstyle{\frac{2+\alpha}{4}}\big)^{\frac{\alpha}{1+\alpha}}\,.
\end{equation*}
 Then the inequality $E_0(H_{\alpha,\mathrm{F}})\geqslant \mathfrak{m}(A_{\alpha,\mathrm{F}}(1))$ yields the lower bound in \eqref{eq:estimate-Fgroundstate}. For the upper bound we perform a variational argument. Let us consider the family $(h_b)_{b>0}$ of trial functions in $\mathfrak{h}\cong L^2(\mathbb{R})$ defined by
 \[
  h_b(x)\;:=\;|x|^{1+\frac{\alpha}{2}}e^{-b|x|^{1+\alpha}}\,\mathbf{1}_{\mathbb{R}^+}(x)\,,
 \]
 where $\mathbf{1}_{\mathbb{R}^+}$ is the characteristic function of $(0,+\infty)$. Owing to \eqref{eq:extAak-F}, $h_b\in\mathcal{D}(A_{\alpha,\mathrm{F}}(1))$ $\forall b>0$. A direct computation yields
 \[
  \frac{\;\langle h_b, A_{\alpha,\mathrm{F}}(1) h_b \rangle_{L^2(\mathbb{R})}}{\langle h_b,h_b \rangle_{L^2(\mathbb{R})}}\,=\,\frac{2^{\frac{1-\alpha}{1+\alpha}}}{ \Gamma (\frac{3+\alpha}{1+\alpha})}\frac{1+b^2(1+\alpha)^2}{b^{\frac{2 \alpha}{1+\alpha}}}\,,
 \]
 whence
 \[
 \begin{split}
   E_0(H_{\alpha,\mathrm{F}})\;&\leqslant\;\min_{b>0}\,\frac{\;\langle h_b, A_{\alpha,\mathrm{F}}(1) h_b \rangle_{L^2(\mathbb{R})}}{\langle h_b,h_b \rangle_{L^2(\mathbb{R})}}\;=\;\frac{2^{\frac{1-\alpha}{1+\alpha}} (1+\alpha)^{\frac{1+3\alpha}{1+\alpha}}}{\alpha^{\frac{\alpha}{1+\alpha}} \Gamma(\frac{3+\alpha}{1+\alpha})}\,.
 \end{split}
 \]
 This provides the upper bound in \eqref{eq:estimate-Fgroundstate}. 
 \end{proof}

 \begin{proof}[Proof of Theorem \ref{thm:scattering}]
  We exploit once again the unitary equivalence and study the scattering governed by the Hamiltonian $\mathscr{H}_{\alpha,a}^{[\gamma]}$. Based on the fibred structure \eqref{eq:Halpha-IIa_unif-fibred} of the latter operator, it is clear that the scattering takes place independently in each $k$-channel. And owing to the analysis of Propositions \ref{prop:SpectrumK0} and \ref{prop:NegativeEigenvaluesKDiffZero} we see that the $k=0$ channel is the only meaningful one. In this case the scattering was studied in Proposition \ref{prop:scatteringIIa}. In particular, the free incident flux on the $k=0$ fibre at energy $E>0$ is described by the plane wave $g_E(x)\sim e^{\pm\ii x\sqrt{E}}$ as $|x|\to +\infty$: this corresponds to a scattering state $\psi\equiv(\psi_k)_{k\in\mathbb{Z}}$ for $\mathscr{H}_{\alpha,a}^{[\gamma]}$ where $\psi_0\equiv g_E$ and $\psi_k\equiv 0$ if $k\neq 0$. We thus see, by inverting the transformations \eqref{eq:unit1}-\eqref{eq:defF2}, that the large distance behaviour of the incident flux for the scattering governed by $H_{\alpha,a}^{[\gamma]}$ on the Grushin manifold $M_\alpha$ has precisely the form \eqref{eq:fscatt}. All the claimed properties are invariant under unitary transformations and then follows at once from Proposition \ref{prop:scatteringIIa}.
  \end{proof}

\def\cprime{$'$}


\begin{thebibliography}{10}

\bibitem{Abramowitz-Stegun-1964}
{\sc M.~Abramowitz and I.~A. Stegun}, {\em {Handbook of mathematical functions
  with formulas, graphs, and mathematical tables}}, vol.~55 of {National Bureau
  of Standards Applied Mathematics Series}, For sale by the Superintendent of
  Documents, U.S. Government Printing Office, Washington, D.C., 1964.

\bibitem{Agrachev-Boscain-Sigalotti-2008}
{\sc A.~Agrachev, U.~Boscain, and M.~Sigalotti}, {\em {A {G}auss-{B}onnet-like
  formula on two-dimensional almost-{R}iemannian manifolds}}, Discrete Contin.
  Dyn. Syst., 20 (2008), pp.~801--822.

\bibitem{Boscain-Beschastnnyi-Pozzoli-2020}
{\sc U.~Boscain, I.~Beschastnnyi, and E.~Pozzoli}, {\em {Quantum confinement
  for the curvature Laplacian $-\frac{1}{2}\Delta+cK$ on 2D-almost-Riemannian
  manifolds}}, arXiv:2011.03300 (2020).

\bibitem{Boscain-Laurent-2013}
{\sc U.~Boscain and C.~Laurent}, {\em {The {L}aplace-{B}eltrami operator in
  almost-{R}iemannian geometry}}, Ann. Inst. Fourier (Grenoble), 63 (2013),
  pp.~1739--1770.

\bibitem{Boscain-Neel-2019}
{\sc U.~Boscain and R.~W. Neel}, {\em {Extensions of {B}rownian motion to a
  family of {G}rushin-type singularities}}, Electron. Commun. Probab., 25
  (2020), pp.~Paper No. 29, 12.

\bibitem{Boscain-Prandi-JDE-2016}
{\sc U.~Boscain and D.~Prandi}, {\em {Self-adjoint extensions and stochastic
  completeness of the {L}aplace-{B}eltrami operator on conic and anticonic
  surfaces}}, J. Differential Equations, 260 (2016), pp.~3234--3269.

\bibitem{Boscain-Prandi-Seri-2014-CPDE2016}
{\sc U.~Boscain, D.~Prandi, and M.~Seri}, {\em {Spectral analysis and the
  {A}haronov-{B}ohm effect on certain almost-{R}iemannian manifolds}}, Comm.
  Partial Differential Equations, 41 (2016), pp.~32--50.

\bibitem{Bruneau-Derezinski-Georgescu-2011}
{\sc L.~Bruneau, J.~Derezi{\'n}ski, and V.~Georgescu}, {\em {Homogeneous
  {S}chr{\"o}dinger operators on half-line}}, Ann. Henri Poincar{\'e}, 12
  (2011), pp.~547--590.

\bibitem{Calin-Chang-SubRiemannianGeometry}
{\sc O.~Calin and D.-C. Chang}, {\em {Sub-{R}iemannian geometry}}, vol.~126 of
  {Encyclopedia of Mathematics and its Applications}, Cambridge University
  Press, Cambridge, 2009.
\newblock General theory and examples.

\bibitem{DoCarmo-Riemannian}
{\sc M.~P. do~Carmo}, {\em {Riemannian geometry}}, {Mathematics: Theory \&
  Applications}, Birkh{\"a}user Boston, Inc., Boston, MA, 1992.
\newblock Translated from the second Portuguese edition by Francis Flaherty.

\bibitem{Franceschi-Prandi-Rizzi-2017}
{\sc V.~Franceschi, D.~Prandi, and L.~Rizzi}, {\em {On the essential
  self-adjointness of singular sub-Laplacians}}, arXiv:1708.09626 (2017).

\bibitem{Fukushima-Oshima-Takeda}
{\sc M.~Fukushima, Y.~Oshima, and M.~Takeda}, {\em {Dirichlet forms and
  symmetric {M}arkov processes}}, vol.~19 of {De Gruyter Studies in
  Mathematics}, Walter de Gruyter \& Co., Berlin, extended~ed., 2011.

\bibitem{GMO-KVB2017}
{\sc M.~Gallone, A.~Michelangeli, and A.~Ottolini}, {\em
  {Kre{\u\i}n-Vi\v{s}ik-Birman self-adjoint extension theory revisited}}, in
  {Mathematical Challenges of Zero Range Physics}, A.~Michelangeli, ed.,
  {INdAM-Springer series, Vol.~42}, Springer International Publishing, 2020.

\bibitem{GMP-Grushin-2018}
{\sc M.~Gallone, A.~Michelangeli, and E.~Pozzoli}, {\em {On geometric quantum
  confinement in {G}rushin-type manifolds}}, Z. Angew. Math. Phys., 70 (2019),
  pp.~Art. 158, 17.

\bibitem{GMP-Grushin2-2020}
\leavevmode\vrule height 2pt depth -1.6pt width 23pt, {\em {Geometric
  confinement and dynamical transmission of a quantum particle in Grushin
  cylinder}}, arXiv:2003.07128 (2020).

\bibitem{Masamune-2005}
{\sc J.~Masamune}, {\em {Analysis of the {L}aplacian of an incomplete manifold
  with almost polar boundary}}, Rend. Mat. Appl. (7), 25 (2005), pp.~109--126.

\bibitem{Nenciu-Nenciu-2009}
{\sc G.~Nenciu and I.~Nenciu}, {\em {On confining potentials and essential
  self-adjointness for {S}chr{\"o}dinger operators on bounded domains in {$\Bbb
  R^n$}}}, Ann. Henri Poincar{\'e}, 10 (2009), pp.~377--394.

\bibitem{Prandi-Rizzi-Seri-2016}
{\sc D.~Prandi, L.~Rizzi, and M.~Seri}, {\em {Quantum confinement on
  non-complete {R}iemannian manifolds}}, J. Spectr. Theory, 8 (2018),
  pp.~1221--1280.

\bibitem{schmu_unbdd_sa}
{\sc K.~Schm{\"u}dgen}, {\em {Unbounded self-adjoint operators on {H}ilbert
  space}}, vol.~265 of {Graduate Texts in Mathematics}, Springer, Dordrecht,
  2012.

\bibitem{Titchmarsh-EigenfExpans-1962}
{\sc E.~C. Titchmarsh}, {\em {Eigenfunction expansions associated with
  second-order differential equations. {P}art {I}}}, {Second Edition},
  Clarendon Press, Oxford, 1962.

\end{thebibliography}
\end{document}